\documentclass[notitlepage,leqno,10pt]{article}

\usepackage{latexsym}
\usepackage{amsmath}
\usepackage{amsfonts}
\usepackage{amssymb}
\usepackage{amsthm}
\usepackage{color}
\renewcommand{\a }{\alpha }
\renewcommand{\b }{\beta }
\renewcommand{\d}{\delta }
\newcommand{\D }{\Delta }

\newcommand{\e }{\varepsilon }

\renewcommand{\l }{\lambda }

\newcommand{\Ric} {{\rm Ric}}

\newcommand{\n }{\nabla }
\newcommand{\var }{\varphi }

\newcommand{\s }{\sigma }
\renewcommand{\S }{\Sigma}

\renewcommand{\o }{\omega }

\newcommand{\pa }{\partial}

\newcommand{\ov}{\overline}

\newcommand{\be}{\begin{equation}}
\newcommand{\ee}{\end{equation}}

\newtheorem{remark}{Remark}[section]
\newcommand{\R}{\mathbb{R}}

\newcommand{\T}{\mathbb{T}}

\newcommand{\Z}{\mathbb{Z}}

\renewcommand{\P}{\mathbb{P}}
\newcommand{\DD}{\mathbb{D}}

\newcommand{\N}{\mathbb{N}}

\newcommand{\Sc}{\mathrm{Sc}}

\newcommand{\Refx}{\mathrm{Ref}_{\mathbf{e}_{x}}}

\newcommand{\SA}{S^{2}_{\tilde{A}}}
\newcommand{\la}{\langle}
\newcommand{\ra}{\rangle}
\newcommand{\ex}{\mathbf{e}_{x}}
\newcommand{\ScP}{\mathrm{Sc}_P}

\newcommand{\tvph}{\tilde{\varphi}}
\newcommand{\tth}{\tilde{\theta}}
\newcommand{\bvph}{\bar{\varphi}}
\newcommand{\bth}{\bar{\theta}}

\newcommand{\bbe}{\mathbf{e}}

\newcommand{\rd}{\mathrm{d}}

\author{Norihisa Ikoma 
\\ \vspace{-0.1cm}
\footnotesize Faculty of Mathematics and Physics
\\ \vspace{-0.1cm}
       \footnotesize       Institute of Science and Engineering
       \\       \vspace{-0.1cm}
 \footnotesize              Kanazawa University
         \\     \vspace{-0.1cm}
   \footnotesize            Kakuma, Kanazawa
              \\ \vspace{-0.1cm}
    \footnotesize           Ishikawa 9201192, JAPAN
\\
\footnotesize \texttt{ikoma@se.kanazawa-u.ac.jp}
\and  Andrea Malchiodi 
\\ \vspace{-0.1cm}
\footnotesize Scuola Normale Superiore 
\\ \vspace{-0.1cm}
\footnotesize Piazza dei Cavalieri, 7 
\\ \vspace{-0.1cm}
\footnotesize  56126 Pisa, ITALY
\\
\footnotesize \texttt{andrea.malchiodi@sns.it} 
\and  Andrea Mondino 
\\ \vspace{-0.1cm}
\footnotesize Mathematics Institute
\\ \vspace{-0.1cm}
\footnotesize Zeeman Building 
\\ \vspace{-0.1cm}
\footnotesize University of Warwick
\\ \vspace{-0.1cm}
\footnotesize Coventry CV4 7AL, UK
 \\ \footnotesize \texttt{A.Mondino@warwick.ac.uk}}

\title{Embedded area-constrained Willmore tori of small area in Riemannian three-manifolds II:  
Morse Theory}

\begin{document}

\hyphenation{ma-ni-fold}

\newtheorem{lem}{Lemma}[section]
\newtheorem{pro}[lem]{Proposition}
\newtheorem{thm}[lem]{Theorem}
\newtheorem{rem}[lem]{Remark}
\newtheorem{cor}[lem]{Corollary}
\newtheorem{df}[lem]{Definition}

\maketitle


\

\begin{abstract}
	This is the second part of a series of two papers where we construct embedded Willmore tori with small area constraint in Riemannian three-manifolds. 
	In both papers the construction relies on a Lyapunov-Schmidt reduction, the difficulty being the M\"obius degeneration of the tori.  
	In the first paper the construction was performed via minimization, here by Morse Theory. 
	To this aim  we establish  new geometric expansions of the derivative of the Willmore functional on small 
	Clifford tori (in geodesic normal coordinates) which degenerate to small geodesic spheres
	with a small handle under the action of the M\"obius group.
	By using these  sharp asymptotics we give sufficient conditions, in terms of the ambient curvature tensors and  Morse inequalities,  
	for having existence/multiplicity of embedded tori which are stationary for the Willmore functional under the constraint of prescribed (sufficiently small) area.  
\end{abstract}

\begin{center}
	
	\bigskip\bigskip
	

	\bigskip
	
	\centerline{\bf AMS subject classification: } 
	49Q10, 53C21, 53C42, 35J60, 83C99.
\end{center}

\section{Introduction}

This is the second part  of a series of two papers where  embedded area-constrained  Willmore tori  in Riemannian 3-manifolds are constructed. Here the construction is performed via Morse theory, whereas  in the previous paper \cite{IMM1} it was  achieved via minimization/maximization. 
\\

Let us start by recalling the  basic definitions and properties of  the Willmore functional. Given an immersion $i:\Sigma \hookrightarrow (M,g)$ of a closed (compact without boundary) 2-dimensional surface $\Sigma$ into a Riemannian $3$-manifold $(M,g)$, the \emph{Willmore functional} is defined by 
\[
W(i):=\int_{\Sigma} H^2 \, d\sigma
\]
where $d\sigma$ is the area form induced by the immersion and  $H$ is the mean curvature  (we adopt the convention that $H$ is the sum of the principal curvatures or, in other words, $H$ is the trace of the second fundamental form $A_{ij}$  with respect to the induced metric $\bar{g}_{ij}$, i.e. $H:=\bar{g}^{ij} A_{ij}$).

An immersion $i$ is called \emph{Willmore surface} (or Willmore immersion) if it is a critical point of the Willmore functional with respect to normal perturbations or, equivalently, if it satisfies the associated Euler-Lagrange equation
\be\label{eq:WillmoreEq}
\Delta_{\bar{g}} H + H |\mathring{A}|^2 + H \Ric(n, n)=0.
\ee
Here $\Delta_{\bar{g}}$ is the Laplace-Beltrami operator corresponding to the induced metric $\bar{g}$, $(\mathring{A})_{ij}:=A_{ij}-\frac{1}{2}H\bar{g}_{ij}$ is the trace-free second fundamental form, $n$ is a normal unit vector to $i$, and $\Ric$ is the Ricci tensor of the ambient manifold $(M,g)$. Since of course a minimal immersion (i.e. an immersion with vanishing mean curvature) satisfies the Willmore equation, Willmore surfaces are a natural higher order generalization of minimal surfaces. Analogously, area-constrained Willmore surfaces satisfy the equation
\[
\Delta_{\bar{g}} H + H |\mathring{A}|^2 + H \Ric(n, n)=\lambda H,
\]
for some $\lambda\in \R$ playing the role of Lagrange multiplier. These immersions are naturally linked to the {\em Hawking mass} 
\[
m_H(i):=\frac{\sqrt{Area(i)}}{64 \pi^{3/2}} \left(16\pi-W(i)\right),
\]
a quantity introduced in general relativity to measure the mass of a portion of space by means of the 
bending effect on light rays. Clearly,  by the latter formula, the critical points of the Hawking mass under area constraint are exactly the  area-constrained Willmore immersions (see \cite{LMS} and the references therein for more material about the Hawking mass).
\\ 

In case the ambient manifold is the \emph{Euclidean three-dimensional space}, the Willmore functional is invariant under the action of the  M\"obius group (i.e. under composition of the immersion with isometries, homotheties  and inversions with respect to spheres), so the  theory of Willmore surfaces can be seen  as a natural merging between \emph{conformal invariance} and \emph{minimal surface theory}.  This was indeed the motivation of Blaschke and Thomsen in the 1920-'30 to introduce such an energy,  rediscovered by Willmore \cite{Will}  in the 60's  and thoroughly studied in  the last twenty years by a number of authors \cite{BK,BR, KS,LY,MN,Riv1,Ros,SiL,Top} (for more details see the introduction of our first paper \cite{IMM1}). Here let us just recall that the minimum of $W$ among all immersed surfaces in $\R^3$ is  achieved by the round sphere  \cite{Will}, the minimum among immersed surfaces of strictly positive genus  is  achieved by the Clifford torus and its M\"obius deformations (the existence of a smooth  minimum among genus one surfaces was proved by Simon \cite{SiL}, the characterization of the minimum was the long standing Willmore conjecture recently proved by Marques-Neves \cite{MN}), and for every positive genus the infimum is achieved by a smooth immersion (the proof of Bauer-Kuwert \cite{BK} is built on top of Simon's work \cite{SiL}  and some geometric ideas of Kusner \cite{Kus}; see also the different approach by Rivi\`ere \cite{Riv1,Riv2})  but it is a challenging open problem to characterize such immersion.      
\\

While all the aforementioned results about Willmore surfaces concern immersions into  the Euclidean space (or, equivalently by conformal invariance, for immersions into a round sphere); the results concerning Willmore immersions into curved Riemannian manifolds are much more limited and recent. In a first stage \cite{CM,LM1,LM2,LM,LMS, Mon1,Mon2}  the existence of Willmore spheres has been investigated  in a perturbative setting. 
The global variational problem, i.e. the existence of smooth immersed spheres minimizing quadratic curvature functionals in compact  Riemannian 3-manifolds,  was then studied in \cite{KMS} by ex\-ten\-ding the Simon's ambient approach to Riemannian manifolds (see also \cite{MonSch} for the non compact case).  
In  \cite{MR1}-\cite{MR2},  a parametric approach  for weak immersions into Riemannian manifolds was developed and the  existence of branched area-constrained Willmore spheres in homotopy classes  established  (as well as the existence of Willmore spheres under various assumptions and  constraints).  
\\

Since all the above existence results  in Riemannian manifolds concern surfaces \emph{of genus $0$}, a natural question is about the \emph{existence of higher genus Willmore surfaces in general curved spaces}; in particular we will focus on the \emph{genus one} case. 

Let us mention that if the ambient space has some special symmetry then the Willmore equation  \eqref{eq:WillmoreEq}  simplifies and it is possible to  construct explicit examples (see for instance  \cite{Wang} for product manifolds and  \cite{BFLM} for  warped product metrics). See also \cite{ChenLi} for the existence of stratified weak branched immersions of arbitrary genus minimizing quadratic curvature functionals under various constraints. 
\\

The goal of the present (and the previous \cite{IMM1}) work is to construct smooth embedded Willmore tori with small area 
constraint in Riemannian 3-manifolds, under some curvature/topological condition but without any symmetry assumption.  More precisely we obtain the following  main result.

\begin{thm}[Existence]\label{thm:Exixtence}
Let $(M, g)$ be a smooth  closed orientable  three-dimensional Riemannian manifold. 
Assume that the scalar curvature is a Morse function and that at every critical point $P$ of the scalar curvature, the Ricci tensor has three distinct eigenvalues.
 Then there exists $\e_0>0$ such that for every  $\e \in (0,\e_0]$ there exists a smooth embedded Willmore torus in $(M,g)$ with constrained 
	area equal to $4 \sqrt{2} \pi^2 \e^2$.   
More precisely, the above surfaces are obtained as normal graphs over exponentiated (M\"obius transformations of) Clifford tori and the corresponding graph functions 
(dilated by a factor $1/\e $) converge to $0$ in $C^{4,\a}$-norm as $\e\to 0$ with decay rate $O(\e^2)$. 
\end{thm}

\begin{remark}\label{r:mild}
(i) The assumptions in Theorem \ref{thm:Exixtence} are generic in the metric $g$.

(ii) If the Ricci tensor is not a multiple of the identity at all points of global maximum and minimum of the scalar curvature then we have at least two critical tori, see Remark \ref{r:minmax}. 
\end{remark}

\

We also obtain a generic multiplicity result.  To state it, we need to introduce some more notation.  As above assume that $(M,g)$ is a closed  connected and orientable three-manifold, that the scalar curvature $P \mapsto \ScP$ is a Morse function and that  
at every critical point $P$ of the scalar curvature the Ricci tensor  $\Ric_P$ has three distinct eigenvalues.
For $q =0,\ldots,3$, we set 
\[
		C_q := \sharp \{ P_i \in M  \;:\;  \nabla \Sc (P_i) = 0, \quad 
		\mathrm{index}\, ( -\nabla^2 \mathrm{Sc} (P_i) )  = q  \};
\]
we then define
\begin{equation}\label{eq:deftCq}
\tilde{C}_0=\tilde{C_1}:=0, \quad \tilde{C}_2:=4 \, C_0,\quad \tilde{C}_q:=4\, C_{q-2}+2 \, C_{q-3},\, q=3,4,5, \quad \tilde{C}_6:=2 \, C_3. 	
\end{equation}
Finally, considering the Betti numbers of $M$ with $\Z_2$ coefficients 
$$
  \beta_q := \mathrm{rank}_{\Z_2} ( H_q ( M; \Z_2 ) ); \qquad q \geq 0, 
$$	
we define  
	\begin{equation}\label{eq:defbq}
	\tilde{\b}_0 = 1; \quad \tilde{\b}_1 = \b_1 + 1; \quad \tilde{\b}_2 = \tilde{\b}_3 = \b_1 + 
	\b_2 + 1; \quad \tilde{\b}_4 = \b_2 + 1; \quad \tilde{\b}_5 = 1; \quad \tilde{\b}_k = 0 
	\; \; \hbox{ for } k \geq 6. 	
	\end{equation}
	
	\begin{remark}\label{r:betti}

	\emph{(i)} The numbers $\tilde{\b}_q$ are the Betti numbers (with $\Z_2$ coefficients) 
	of the projective tangent bundle  over $M$.  By a classical result of differential topology due to Stiefel 
	(see for instance  \cite[page 148]{MilS}), three-dimensional oriented manifolds are parallelizable, 
	i.e. the tangent bundle is trivial: $TM \simeq M \times \R^3$.  As a consequence, the projective tangent bundle is homeomorphic to $M \times \R\P^2$. Since 
	 $H_k(\R\P^2, \Z_2)=\Z_2$ for $0\leq k \leq 2$ and zero otherwise, the $\tilde{\b}$'s can be 
	 computed as a direct application of K\"unneth's formula. 

   \emph{(ii)} Using the homology of $M$ with $\Z_2$ coefficients is more convenient than using  
   standard $\Z$ coefficients for a number of reasons. First of all K\"unneth's formula turns out to be easier. 
   Secondly, the Betti numbers with $\Z_2$ coefficients  of a compact manifold $X$ are always bounded below by the Betti numbers with $\Z$ coefficients, this because they also keep track of the $\Z_2$-torsion part. The precise relation between the two is given by the  Universal Coefficients Theorem (see for instance \cite[Chapter 3.A]{Hatcher}), which implies that $H_k(X, \Z_2)$ consists of 
  \begin{itemize}
  \item a $\Z_2$ summand for each $\Z$ summand  of $H_k(X,\Z)$,
  \item a $\Z_2$ summand for each $\Z_{2^n}$ summand in $H_k(X,\Z)$, $n\geq 1$,
  \item a  $\Z_2$ summand for each $\Z_{2^n}$ summand in $H_{k-1}(X,\Z)$, $n\geq 1$.
  \end{itemize} 
   In particular, in our case of $X=M\times \R\P^2$, the $\Z$-Betti numbers vanish in dimension larger than three while  the  $\Z_2$-Betti numbers do not vanish in dimension 4 and 5.   Clearly this permits stronger conclusions in terms of existence and multiplicity of critical points via Morse-theoretic arguments. 
   \end{remark}

Now we are ready to state our second main theorem. 

\begin{thm}[Generic multiplicity]\label{thm:Multiplicity}
	Let $(M, g)$ be a smooth  closed orientable  three-dimensional Riemannian manifold. 
Then for generic metrics $g$, if  $\tilde{\beta}_q-\tilde{C}_q >0$ for some $q\in\{0,\ldots,4\}$, then there exists $\e_0>0$ such that for every  $\e \in (0,\e_0]$  there are at least $\tilde{\beta}_q-\tilde{C}_q$
smooth embedded Willmore tori in $(M,g)$ with constrained  area equal to $4 \sqrt{2} \pi^2 \e^2$ and with index $q$. In particular there are at least  $\sum_{q=0}^4 (\tilde{\beta}_q-\tilde{C}_q)^+$ area-constrained Willmore tori.
\end{thm}

\begin{remark}\label{r:01}
Notice that we always have $\tilde{\beta}_q-\tilde{C}_q>0$, for $q=0,1$,  so the above result implies in particular that for generic metrics there exist at least two area-constrained Willmore tori, one with index zero and the other with index one, the  
index being intended for critical points of the Willmore functional under  area constraint. Also, as the Morse inequalities on $M$ imply $C_q \geq \beta_q$ for generic metrics, the condition $\tilde{\beta}_q-\tilde{C}_q >0$ is 
not satisfied for  $q = 5$ or $q=6$. 
\end{remark}

\

\noindent {\bf Examples.}  
If	 $M$ is homeomorphic to $S^3$, $S^2\times S^1$ or $S^1\times S^1 \times S^1$, we get the following values for $\tilde{\beta}_k$.
	\begin{eqnarray}
		M=S^3 &:& \tilde{\beta}_k= 1\; \text{ for } k=0,\ldots,5, \;  \tilde{\beta}_k=0\text{ for } k\geq 6.  \nonumber\\
		M=S^2\times S^1 &:& \tilde{\beta}_0=\tilde{\beta}_5=1, \;  \tilde{\beta}_1=\tilde{\beta}_4=2, \; \tilde{\beta}_2=\tilde{\beta}_3=3,\;  \tilde{\beta}_k=0\text{ for } k\geq 6. \nonumber\\
		M=(S^1)^3 &:& \tilde{\beta}_0=\tilde{\beta}_5=1, \; \tilde{\beta}_1=\tilde{\beta}_4=4,\; \tilde{\beta}_2=\tilde{\beta}_3=7,\; \tilde{\beta}_k=0\text{ for } k\geq 6. \nonumber
	\end{eqnarray}

\subsection*{Outline of the strategy}
As in our first paper \cite{IMM1} the proof relies on a Lyapunov-Schmidt reduction (encoding the 
variational structure of the problem, see \cite{ab1,ab2} and the book \cite{am2}). Using  such  techniques, together with the stability property of Clifford tori proved by Weiner \cite{WEINER} (see also the related gap-energy result \cite{MonNgu}), we reduce the problem of finding area-constrained Willmore tori to a \emph{finite dimensional} variational problem.  More precisely we consider the finite dimensional space of the images, via the exponential map in $(M,g)$, of   M\"obius-inverted Clifford tori with small area. Notice that since the action of the M\"obius group is non-compact, such a finite dimensional space is non-compact too, with degeneracy due to the presence of 
a shrinking handle.  

In the present work, the rough idea used to  infer existence of critical points is to exploit  the topology of the finite dimensional space ${\mathcal T}_\e$ of exponentiated and rotated  M\"obius images of Clifford tori  having area $4 \sqrt{2} \pi^2 \e^2$ and argue via a Morse-theoretical approach. To this aim, recalling that under our assumptions $M$ is parallelizable, we first observe that the space  ${\mathcal T}_\e$  is diffeomorphic to $M\times \mathbb{B}\R\P^2$,  $\mathbb{B}\R \P^2$ being the bundle of tangent vectors to $\R\P^2$ with length less than 1. The geometric situation of tori degenerating to geodesic spheres with shrinking handles corresponds 
to approaching the boundary of $M\times \mathbb{B}\R\P^2$, consisting in the vectors of length one in $T \R \P^2$. 
In order to apply Morse theory to a manifold with boundary (see for instance the classical work of Morse-Van Schaack \cite{MVS}) it is crucial to understand the normal derivative at the boundary of the manifold; this corresponds in our framework to computing the derivative with respect to the M\"obius parameter. Such a computation is quite delicate since we need sharp estimates and since the torus is degenerating  (as it is natural to expect, the computation involves singular integrals); this will take a large part of the present paper (for the final result see Proposition \ref{p:varvar2} and Remark \ref{r:varvar2}). 

A crucial role in such an expansion of the normal derivative is played by the  function $\mathcal{F}$ defined below. 
Given $P \in M$, and  an orthonormal frame $\{\mathbf{e}_{P,1}, \mathbf{e}_{P,2}, \mathbf{e}_{P,3}\}_{P \in M}$ at $P$, 
we define $\mathcal{F}(P,\cdot): SO(3) \to \R$ by 
	\[
		\mathcal{F}(P,R) := \Ric_{P}(R \mathbf{e}_{P,2}, R \mathbf{e}_{P,2}) 
		- \Ric_{P} ( R \mathbf{e}_{P,3}, R \mathbf{e}_{P,3} ). 
	\]
The assumptions of Theorem \ref{thm:Exixtence} imply indeed the following 
non-degeneracy condition for $\ScP$ and $\mathcal{F}$: for the proof of the second one see Proposition  
\ref{p:FNonDeg}.
	\begin{enumerate}
		\item[$(ND1)$] 
			The function $P \mapsto \ScP: M \to \R$ is a Morse function. 
			In particular, $\mathrm{Sc}$ has finitely many critical points 
			$P_1,\ldots,P_k$. 
		\item[$(ND2)$] 
			For each $i=1,\ldots,k$, 
			 $\mathcal{F}_i(R) := \mathcal{F}(P_i,R): SO(3) \to \R$ is a Morse function 
		for every $1 \leq i \leq k$,  and 
			$\mathcal{F}_i(R) \neq 0$ if $\nabla \mathcal{F}_i(R) = 0$. 
	\end{enumerate} 
By $(ND2)$, every $\mathcal{F}_i$ has finitely many critical points and 
we call them $R_{i,1}, \ldots, R_{i,\ell_i}$: recalling \eqref{eq:deftCq}, 
by Proposition \ref{p:FNonDeg} it turns out that 
\begin{equation}\label{eq:10-nov}
\tilde{C_q} := \frac 12 \sharp \{ (P_i, R_{i,\ell}) \in M \times SO(3)  \;:\;
		\mathrm{index}\, ( -\nabla^2 \mathrm{Sc} (P_i) ) 
		+ \mathrm{index}\, (-\nabla^2 \mathcal{F}_i (R_{i,\ell})) = q  \text{ and } \mathcal{F}_i(R_{i,\ell})<0\}. 
\end{equation}
By our energy expansions, see Section \ref{s:pf}, the  $\tilde{C}_q$'s represent the numbers of critical points of index $q$ 
for the restriction of the Willmore to the boundary of ${\mathcal T}_\e$ (defined above) such that 
the gradient of the energy points inwards ${\mathcal T}_\e$.  Notice that the factor $\frac 12$ in the definition of $\tilde{C}_q$  is a consequence of the symmetry of the degenerate Clifford torus: indeed for every degenerate Clifford torus there exists a non trivial rotation $R\in SO(3), R \neq Id$ leaving  the surface invariant (for more details see Remark \ref{r:5new}). 
The conclusion of  Theorems \ref{thm:Exixtence} and \ref{thm:Multiplicity} will then follow from 
the general results in \cite{MVS}.


\

Besides  Theorems  \ref{thm:Exixtence} and \ref{thm:Multiplicity}, the main contribution of the present paper is the aforementioned  expansion for the derivative of the Willmore energy on degenerating tori (see Proposition \ref{p:varvar2}). We believe that it might play a role in further developments of the topic, especially in ruling-out possible degeneracy phenomena under global (non-perturbative) variational approaches to the problem, as it has already happened for
the case of Willmore spheres (see for instance \cite{KMS,MR2,MonSch}).

\

\noindent 

The outline of the paper is as follows: in Section \ref{s:prel} we recall some preliminary results, as well as 
the finite-dimensional reduction of the constrained Willmore problem from \cite{IMM1}. In Section \ref{s:degtori} we
analyse in detail the M\"obius degeneration of Clifford tori to spheres, describing their asymptotics (away from the shrinking handle) 
as normal graphs. In Section  \ref{s:comp} we derive one core estimate, namely the variation of the  Willmore energy on (degenerated) Clifford tori with respect to the M\"obius parameter. In Section \ref{s:pf} we prove our main theorem via Morse theory, and 
finally in the Appendix we collect some explicit computations.

\subsection*{Acknowledgements}
The first author was supported by JSPS Research Fellowships 24-2259. 
The second author was supported by the FIRB project 
{\em Analysis and Beyond} and by the PRIN 
{\em Variational and perturbative aspects of nonlinear differential problems}. 
The third author acknowledges  the support of  the ETH fellowship. 
This work was done while the first author visited the second author 
at the University of Warwick and at SISSA: 
he would like to express his gratitude for the hospitality.

\section{Preliminaries}\label{s:prel}

Denoting by $g_0$ the flat Euclidean metric, let us first state a basic property of the Willmore functional $W_{g_0}$ for immersions $i : \Sigma \to \R^3$ 
$$
W_{}(i(\Sigma)) = \int_{\Sigma} H^2 d \s. 
$$

\begin{pro}\label{p:Mobinv}
	Let $\Sigma$ be a closed surface of class $C^2$  and let $i : \Sigma \to \R^3$ be an immersion. 
	Then, if $\l > 0$ and if  $\Phi_{x_0,\eta}$ is a M\"obius inversion (see \eqref{eq:Phix0eta}), one has the invariance properties 
	$$
	a) \quad W_{g_0}(\l i(\Sigma)) = W_{g_0}(i(\Sigma)) \quad \quad \hbox{ and } \quad \quad b) \quad W_{g_0}((\Phi_{x_0,\eta} 
	\circ i)(\Sigma)) = W_{g_0}(i(\Sigma)) 
	\quad \hbox{ provided } x_0 \not\in i(\Sigma). 
	$$
\end{pro}

\noindent We will next introduce some notation and  recall the finite-dimensional reduction 
procedure from \cite{IMM1}.

\subsection{Notation and small tori in manifolds}\label{ss:not}

We consider the standard Clifford torus $\T$ obtained via the following 
parametrization  
	\[
		\T := \left\{ X( \tvph , \tth) \;:\;
		\tvph, \tth \in [-\pi,\pi] \right\} 
	\]
where 
	\[
		X(\tvph,\tth) := 
		\left( (\sqrt{2}+\cos \tvph) \cos \tth, \ 
		(\sqrt{2} + \cos \tvph) \sin \tth, \ 
		\sin \tvph  \right).
	\]

\noindent		For $x_0 \in \R^3$ and 
$\eta > 0$, the spherical inversion 
with respect to $\partial B_{\eta}(x_0)$ 
is defined by 
\begin{equation}\label{eq:Phix0eta}
\Phi_{x_0,\eta}(x) := 
		\frac{\eta^2}{|x - x_0|^2} ( x - x_0 ) + x_0.
\end{equation}
For any smooth compact surface 
$\Sigma \subset \R^3 \backslash \{x_0\}$, 
we set $\overline{\Sigma} := \Phi_{x_0,\eta}(\Sigma)$ 
and we denote 
the volume elements of $\Sigma$ and $\overline{\Sigma}$ 
by $d \sigma_{\Sigma}$ and $d \sigma_{\overline{\Sigma}}$ respectively. 
Then it is well known that 
	\[
		d \sigma_{\overline{\Sigma}} = 
		\frac{\eta^4}{|x-x_0|^4} d \sigma_\Sigma. 
	\]
We are interested in M\"obius maps which preserve the area of $\T$: we first 
translate the torus by the vector $- ( \sqrt{2} + 1 + \xi) \mathbf{e}_x$, $\xi > 0$ 
where $\mathbf{e}_x := (1,0,0)$ (so that it will be contained in $\{x_1 < 0 \}$), 
and then choose $\xi = \xi_\eta > 0 $ depending on $\eta$ so to preserve the area (see Lemma 2.1 in \cite{IMM1}).  We set
	\[
		\T_{\xi_\eta} := \T - ( \sqrt{2} + 1 + \xi_\eta) \mathbf{e}_x, 
		\quad 
		Y(\tvph,\tth,\eta):= 
		X(\tvph,\tth) - (\sqrt{2}+1+\xi_\eta) 
		\mathbf{e}_x
	\]
and observe that 
\begin{equation}\label{2}
  4 \sqrt{2} \pi^2 = {\rm Area}(\Phi_{0,\eta}(\T_{\xi_\eta})) 
  		= \eta^4 \int_{-\pi}^{\pi} \int_{-\pi}^{\pi} 
  		\frac{\sqrt{2} + \cos \tvph}
  		{|Y(\tvph,\tth,\eta)|^4} d \tvph d \tth. 
\end{equation}

\noindent
Our aim is to describe degenerating tori, namely to understand quantitatively 
the behaviours of 
$\xi_{\eta}$ and $\Phi_{0,\eta}(\T_{\xi_{\eta}})$  as $\eta \to 0$. 
To do so, we define the following map:
	\[
		Z(\bvph,\bth,\eta) 
		:= \Phi_{0,\eta} 
		\left( Y(\eta^2 \bvph, \eta^2 \bth,\eta) \right)
		= \Phi_{0,1} \left( 
		\eta^{-2} Y(\eta^2 \bvph, \eta^2 \bth,\eta)
		\right)
	\]
for $(\bvph,\bth) \in \R^{2}$. In Section 2 of \cite{IMM1} the following result 
was proved.

	\begin{lem}\label{104} (\cite{IMM1})
For each $\eta > 0$, there exists a unique $\xi_{\eta}>0$ such that 
	\[
		{\rm Area}\, (\Phi_{0,\eta} ( \T_{\xi_{\eta}})) = 4 \sqrt{2} \pi^{2}.
	\]
Moreover, the map $\eta \mapsto \xi_{\eta}: (0,\infty) \to (0,\infty)$ 
is smooth and strictly increasing in $(0,\infty)$. 
In addition, $\xi_{\eta} \to 0$ as $\eta \to 0$ and 
$\xi_{\eta} \to \infty$ as $\eta \to \infty$. Furthermore we have the properties

\begin{description}

\item{\rm{(i)}} 
$\eta^{4} / \xi_{\eta}^2 = 4 \sqrt{2} \pi + O(\eta^2)$ as $\eta \to 0$.

\item{\rm{(ii)}} 
$\Phi_{0,\eta}(\T_{\xi_{\eta}})$ converges to 
the sphere with radius $\sqrt[4]{2\pi^{2}} $ 
centred at $-\sqrt[4]{2\pi^{2}} \mathbf{e}_x$ 
in the following sense: 
for any $R>0$ and $k \in \N$, if $\eta \leq 1/ R^4$, then 
	\[
		\left\| Z(\cdot, \cdot, \eta  ) - Z_0    \right\|_{C^k ( [-R,R]^2 ) }
		\leq C_{k} \eta^{3/2}
	\]
as $\eta \to 0$, where $C_k$ depends only on $k$ and 
$Z_0$ is defined by 
	\[
		Z_0(\bvph,\bth) := 
		\Phi_{0,1} \left(  
		- \frac{1}{2 \sqrt[4]{2\pi^{2}}} \mathbf{e}_{x}
		 + (\sqrt{2}+1) \bth \mathbf{e}_y + \bvph \mathbf{e}_z
		 \right)
	\]
where $\mathbf{e}_y := (0,1,0)$ and $\mathbf{e}_z := (0,0,1)$. 
\end{description}
\end{lem}

\

\noindent A more detailed analysis of $\xi_\eta$ will be carried out in Section \ref{s:degtori}. Incorporating 
also rotations around the $z$ axis, we obtain a smooth two-dimensional family of tori with the same area 
which includes $\T$. Its properties can be summarized in the following result. 

%

\begin{pro}\label{p:disk} (\cite{IMM1}, Section 2) There exists a smooth family of conformal immersions $T_\omega$ of $\T$ into $\R^3$, parametrized by 
$\o \in \DD$, $\DD$ being the unit disk in $\R^2$, which preserves the area of $\T$ and for which the following  hold 
\begin{description}
\item{a)}  $T_0$ = Id;  

\item{b)} for $\o \neq 0$, $T_\omega$ is an inversion with respect to a sphere centred at a point in $\R^3$ aligned to $\o$ (viewed as  
an element of $\R^3$ with null $z$-component); 

\item{c)} as $|\omega|$ approaches $1$, $T_\omega(\T)$ degenerates to a sphere of radius $\sqrt[4]{2\pi^2} $ centred at $\sqrt[4]{2\pi^2} \frac{\o}{|\o|}$. 
\end{description}

\end{pro}

\

\noindent In what follows, we will use the symbol $\T_{\omega}$ for $T_{\omega}(\T)$. We will describe next the global structure 
of exponential maps of scaled and rotated tori in the manifold $M$. 
\\For each $P \in M$ we construct a family of surfaces from $\T$, 
$R \in SO(3)$ and $T_{\o}$: 
\[
\{ \exp_{P} ( \e R \T_{\o}  ) \;:\; R \in SO(3), \ \o \in \DD \}, 
\]
where $\e > 0$ is chosen small.  Notice that, due to the rotation invariance of the Clifford torus $\T$, the above family is 4-dimensional and not 5-dimensional; indeed it is not difficult to see that it can be parametrized by 
$\mathbb{B}\R \P^2$,  the bundle of tangent vectors to $\R\P^2$ with length less than 1. Letting then $P$ vary, we obtain 
a seven-dimensional bundle over $M$ with fiber $\mathbb{B}\R \P^2$. We will see in the next subsection that the above tori 
form a family of approximate solutions to our problem, and that they may be slightly modified to become true solutions.

\begin{rem}\label{r:78}
In order to further simplify the notation we  will sometimes parametrize the space of exponentiated tori $\exp_P (\e R \T_\o)$ by $(P,R,\o)\in M\times SO(3) \times \DD$. Notice that in this way we are using an extra parameter; this has the advantage of simplifying our notation.
\end{rem}

\subsection{Finite-dimensional reduction}\label{ss:22}

We also recall the finite-dimensional procedure in \cite[Section 3]{IMM1} to attack the constrained 
Willmore problem. This procedure consists in finding first a family of approximate solutions, which will 
be then adjusted to constrained Willmore surfaces up to some Lagrange multiplier.

\

\noindent We fix a  compact set $K$ (typically, a closed ball centred at the origin) of the unit disk $\DD$ and we consider then the family 
\[
   \hat{\mathcal{T}}_{\e, K}= \left\{ \e  \, R \, \T_\o \; : \; R \in SO(3), \o \in K  \right\}. 
\]
We notice that, by construction, elements in $\hat{\mathcal{T}}_{\e,K}$ consists of  Willmore  surfaces in 
$\R^3$ all with area identically equal to $4 \e ^2 \sqrt{2} \pi^2$. 
We then construct a family of surfaces in $M$ defined by exponential maps of elements in  $\hat{\mathcal{T}}_{\e,K}$ 
from arbitrary points $P$ of $M$.
Here we remark that since $M$ is parallelizable (see Remark \ref{r:betti}), 
there exist a global orthonormal frame $\{F_{P,1}, F_{P,2}, F_{P,3}\}_{P \in M}$ 
and we may identify $TM$ with $M \times \R^3$. 
Using this identification, we may also regard the exponential 
map $\exp_P^g$ as a map from $\R^3$ into $M$ for each $P \in M$. 
Then we set 
\begin{equation} \label{eq:tildetKe}
   {\mathcal{T}}_{\e, K}= \left\{ \exp_P(\Sigma) \; : \; P \in M, \Sigma \in \hat{\mathcal{T}}_{\e, K} \right\}. 
\end{equation}
It will be convenient for us to scale coordinates in order to work with surfaces whose area is of order 1, 
exploiting the scaling invariance of the Willmore functional. 
Precisely, introduce a new metric $g_\e $ by 
	\[
		g_\e (P) := \frac{1}{\e^2} g(P).
	\]
Then we have the following facts: (see Section 3 in \cite{IMM1})

	\begin{enumerate}
		\item[(i)] 
			Write $W_g$ and $W_{g_\e}$ for the Willmore functional on 
			$(M,g)$ and $(M,g_\e)$. Then $\Sigma$ is a Willmore surface with the 
			area constraint in $(M,g)$ if and only if  it is so  in $(M,g_\e)$. 
		
		\item[(ii)] 
			The exponential maps $\exp_P^g$ on $(M,g)$ are diffeomorphic  
			on the Euclidean ball $B_{\rho_0}$ for each $P \in M$ and 
			satisfies 
				\[
					\exp_P^{g} (\e z) = \exp^{g_\e}_P(z)
				\]
			for all $|z| \leq \e^{-1} \rho_0$ where $\exp^{g_\e}_P$ is 
			the exponential map on $(M,g_\e)$. 
		
		\item[(iii)] 
			Set $g_{P} := (\exp_P^g)^\ast g$ and 
			$g_{\e ,P} := (\exp_P^{g_\e})^\ast g_{\e}$. Then $g_{\e ,P,\a \b}$ 
			has the following expansion: 
				\begin{equation}\label{eq:ge=d+eh}
					g_{\e ,P,\a \b} (y) = \delta_{\a \b} + \e^2 h^{\e}_{P,\a \b} (y) 
					\qquad 
					\text{for each $|y|_{g_0} \leq \e^{-1} \rho_0$ }
				\end{equation}
			where $h^\e_{P,\a \b}(y)$ satisfies 
				\begin{align}
						& h_{P,\a \b}^\e (y) = 
						\frac{1}{3} R_{\alpha \mu \nu \beta} y^\mu y^\nu 
						+ \tilde{R}_{\a \b}(\e ,y) , \quad 
						\sum_{i = 0}^\ell  \left| \nabla^i \tilde{R}(\e ,\cdot) \right| 
						\leq C_\ell \e^3, 
						\label{eq:ex-g-2}
						\\
						& |y|^{-2} | h^\e _{P,\a \b} (y) | 
						+ |y|^{-1} | \nabla_y h^\e_{P,\a \b} (y) | + 
						\sum_{i=2}^\ell | \nabla^i h^\e_{P,\a \b} (y) | 
						\leq h_{0,\ell},
						\label{eq:defh}
						\\
						& |y|^{-2} |D^{k+1}_P g_{\e ,P,\a \b} (y)| 
						+ |y|^{-1} |D^{k+1}_P \nabla_y g_{\e ,P,\a \b}(y) | 
						+ \sum_{i=2}^\ell | D_P^{k+1} \nabla_y^i g_{\e ,P,\a \b}(y) | 
						\leq C_{k,\ell} \e^2
						\label{eq:met-deri}
				\end{align}
			for all $|y|_{g_0} \leq \e^{-1} \rho_0$, $k, \ell \in \N$. 
			Here $D_P$ denotes the differential by $P$ in the original metric of $M$. 
			
		\item[(iv)] 
			The family $\mathcal{T}_{\e ,K}$ is expressed as 
				\[
					\mathcal{T}_{\e ,K} = \{  \exp^{g_\e}_P ( R \T_\o ) \; : \; 
					P \in M , \ R \in SO(3), \ \o \in K  \}. 
				\]
	\end{enumerate}

%
%

We recall next the following well-known result concerning variations of $W_{g_\e}$ 
(see for example Section 3 in \cite{LMS}). 
\begin{pro}\label{p:1-2-var}
For an immersion $i : \Sigma \to (M,g_\e)$ one has 
\begin{equation}\label{eq:varWvarW}
  d W_{g_\e}(i(\Sigma))[\var] 
  = \int_{\S } \left( L H + \frac{1}{2} H^3 \right) \var \, d \s = - 
  \int_{\Sigma}  \left\{ \Delta H +  
  \left(|\mathring{A}|^2 + \Ric(n,n) \right)H  \right\} \var d \s, 
\end{equation}
where $L$ is the elliptic, self-adjoint operator 
\[
 L \var  := - \D  \var - \var \left( |A|^2 + \Ric(n, n) \right). 
\]
We also write $W_{g_\e}'(i(\Sigma)) := LH + H^3/2$. 
%
%
\end{pro}

%
%

\

\noindent 
$\mathcal{T}_{\e, K}$ form a family of approximate solutions to our problem. In fact, 
let us recall the following result.

\begin{lem}\label{l:appsol} (\cite{IMM1}, Section 3)
Consider the rescaled framework described above.
Fix $K$ as before, $\ell \in \N$ and $\gamma \in (0,1)$. 
There exists a constant $C_{K,\ell}$ such that for $\e$ small 
$$
  \| W_{g_\e}' (\Sigma) \|_{C^{\ell,\gamma}(\Sigma)} \leq C_{K,\ell} \e^2 
  \qquad \quad \hbox{ for every } \Sigma \in  \mathcal{T}_{K , \e}. 
$$  
\end{lem}

\

\noindent Next we consider small perturbations of the surfaces in $\mathcal{T}_{\e,K}$ 
in the following way. 
As in Section 3 of \cite{IMM1}, for $(P,R,\o) \in M \times SO(3) \times \DD$, 
we denote by $g_{\e ,P,R,\o}$ 
the pull back of $g_{\e ,P}$ via the map $R \circ T_\o $: 
$g_{\e ,P,R,\o} := (R \circ T_\o)^\ast g_{\e ,P} 
= T_\o^\ast \circ R^\ast \circ (\exp_{P}^{g_\e})^\ast g_\e $. 
Observe that $(\T,g_{\e ,P,R,\o})$ is isometric to 
$(\exp_{P}^{g_\e} (R \T_\o), g_\e  )$ and $(R \T_\o , g_{\e ,P})$. 
We write $n_{\e ,P,R,\o}$ for the unit outer normal to $(\T,g_{\e ,P,R,\o})$. 
Then for regular functions $\varphi : \T \to \R$, 
we consider perturbations of $(\T,g_{\e ,P,R,\o})$ as follows: 
	\begin{equation} \label{eq:defSw}
		\begin{aligned}
			&(\T[\varphi])_{\e ,P,R,\o} := 
			\{ p + \varphi(p) n_{\e ,P,R,\o}(p) \;:\; p \in \T \},
			\\
			& \left( R \T_\o [\varphi] \right)_{\e ,P} 
			:= \left\{ R T_\o ( p  + \varphi(p) n_{\e ,P,R,\o} (p)  ) \;:\; p \in \T \right\},
			\quad 
			\Sigma_{\e ,P,R,\o} [\varphi] 
			:= \exp_P^{g_\e} \left( \left( R \T_\o [\varphi] \right)_{\e ,P}  \right).
		\end{aligned}
	\end{equation}
Noting that $(R \T_\o [0])_{\e, P} = R \T_\o $, let us also set 
\be \label{eq:defSw2}
\Sigma_{\e,P,R,\omega} = \Sigma_{\e,P,R,\omega} [0] 
=\exp^{g_\e}_P( R \T_\omega  ).
\ee 
Given a positive constant $\ov{C}$, we  define next the family of functions
\[
   \mathcal{M}_{\e ,P,R,\o} = \left\{  \var \in C^{4,\gamma}(\T,\R)  \; : \; \|\var \|_{C^{4,\gamma}(\T)} \leq 
    \ov{C} \e^2  \text{ and such that} \left|\Sigma_{\e,P,R,\omega}[\var]\right|_{g_\e} 
    = 4 \sqrt{2} \pi^2
   \right\}.
\]
Here we remark that since we only consider small perturbations, 
$\Sigma_{\e ,P,R,\o}[\varphi]$ can be expressed as a normal graph 
of $\T$. Hence, we pull back all geometric quantities of $\Sigma_{\e ,P,R,\o}[\varphi]$ 
onto $\T$. 
Finally, on $\T$, we consider Jacobi fields $Z_{i,R,\o}$, $i = 1, \dots, 7$ 
for $R \T_\o $ which generate conformal maps 
preserving the area of the torus (see also the notation in \cite{IMM1}). 
%
%
%
%
Exploiting the non-degeneracy property  from \cite{WEINER}, 
one can prove the following result.

\begin{pro}\label{p:lyap} (\cite{IMM1}, Section 3)
Fix a compact subset $K$ of  $\DD$ as above. 
Then there exist positive constants  $\bar{C}_K$ and $\bar{\e}_K$ such that 
for any $\e\in (0,\bar{\e}_K]$ and 
every $(P,R,\o) \in M \times SO(3) \times K$, 
there exists a  function 
$\var_{\e} = \var_{\e}(P,R,\omega) \in C^{5,\gamma}(\T)$ such that 
$$
  a) \quad  W_{g_\e}'(\Sigma_{\e,P,R,\omega} [\var_{\e}(P,R,\omega)]) = 
  \b_0 H_{\e ,P,R,\o}[\var_\e] + 
     \sum_{i=1}^7 \beta_i Z_{i,R,\o};  \qquad  
   \quad b) \quad |\Sigma_{\e,P,R,\omega}  [\var_{\e}]|_{g_\e} = 4 \sqrt{2} \pi^2, 
$$
for some numbers $\b_0, \dots, \b_7$. Here $\Sigma_{\e,P,R,\omega} [\var]$ is  as in \eqref{eq:defSw}, 
while $H_{\e,P,R,\o}[\var_\e]$ stands for the mean curvature of  $\Sigma_{\e,P,R,\omega} [\var_\e]$. 
Moreover, the map $M\times SO(3) \times K\to C^{5,\gamma}(\T)$ defined by  $(P,R,\omega)\mapsto \var_\e(P,R,\omega)$ is smooth and satisfies 
	\[
		\sum_{k=0}^2 \left\| D^{k}_{P,R,\o} \varphi_\e (P,R,\o) 
		\right\|_{C^{5,\gamma}(\T)} \leq C_K \e^2.
	\]
In particular, $\varphi_\e(P,R,\o) \in \mathcal{M}_{\e ,P,R,\omega}$. 
\end{pro}

\

\noindent We can finally encode the variational structure of the problem by means for the following result. 

\begin{pro}\label{p:variational} (\cite{IMM1}, Section 3)
Let $K\subset \subset \DD$, $\bar{\e}_{K}$ and ${\var}_{\e}$ be as in Proposition \ref{p:lyap}. For $\e\in[0, \bar{\e}_K]$ define the function $\Phi_\e : \mathcal{T}_{\e, K}\to \R$ by
\[
  \Phi_\e(P,R,\o) := W_{g_\e}(\Sigma_{\e,P,R,\omega} [\var_{\e}(P,R,\omega)]).
\]
Then there exists $\bar{\e}_{K}'\in (0, \bar{\e}_{K}]$ such that, if $\e\in (0, \bar{\e}_{K}']$ we have  
\begin{equation}\label{eq:estC0Phie}
  \left| \Phi_\e(P,R,\o)- W_{g_\e}(\Sigma_{\e,P,R,\omega})\right| \leq C_K \e^4. 
\end{equation}
For such $\e$'s,  if $(P_\e,R_\e,\o_\e)\in M\times SO(3)\times K$  is critical for $\Phi_\e$, then the surface $\Sigma_{\e,P_\e,R_\e,\omega_\e} [\var_{\e}(P_\e,R_\e, \omega_\e)]$ satisfies the area-constrained Willmore equation. 
\end{pro}

\section{On degenerating Clifford tori} \label{s:degtori}

In this section we analyse M\"obius-degenerating tori. In particular we improve the accuracy 
of the estimate $(i)$ in Lemma \ref{104} and  derive the asymptotics of degenerate tori 
viewed as normal graphs on the limit sphere (except for the small handle), see $(ii)$ in 
Lemma \ref{104}.

\subsection{Precise asymptotics of $\xi_\eta$}\label{ss:pa}

The following estimate on  $\xi_\eta$ will be needed below.

\begin{lem}\label{l:xi-xi'}
In the notation of Lemma \ref{104}, as $\eta \to 0$, we have 
$$
  2 \xi_\eta - \xi'_\eta \eta = O(\eta^4). 
$$
\end{lem}

\begin{proof}
Recall that $\Phi_{0,\eta}(\T_{\xi_{\eta}})$ has  fixed area 
$4 \sqrt{2} \pi^{2}$ (see \eqref{2} and Lemma \ref{104}): 
	\begin{equation}\label{eq:area}
		4 \sqrt{2} \pi^{2} = \eta^{4} \int_{-\pi}^{\pi} \int_{-\pi}^{\pi} 
		\frac{\sqrt{2}+\cos \tvph}{|Y(\tvph,\tth,\eta)|^{4}} 
		d \tvph d \tth. 
	\end{equation}
Next, we claim that 
		\begin{equation}\label{eq:deri-xi-eta}
			\frac{\xi'_{\eta}}{\eta} = \frac{1}{\sqrt[4]{2 \pi^{2}}} 
			+ O(\eta^{2}). 
	\end{equation}
Differentiating \eqref{eq:area} with respect to $\eta$, we have 
	\[
		0=4 \eta^{3} \int_{-\pi}^{\pi} \int_{-\pi}^{\pi} 
		\frac{\sqrt{2} + \cos \tvph}{|Y(\tvph,\tth,\eta)|^{4}} 
		d \tvph d \tth  
		- 2 \xi_{\eta}' \eta^{4} 
		\int_{-\pi}^{\pi} \int_{-\pi}^{\pi} 
		\frac{(\sqrt{2}+\cos \tvph) f(\tvph, \tth, \eta)}
		{|Y(\tvph,\tth,\eta)|^{6}} d \tvph d \tth 
	\]
where 
	\[
		f(\tvph,\tth,\eta) := 
		2 \left\{ (\sqrt{2}+1) 
		- ( \sqrt{2} + \cos \tvph) \cos \tth + \xi_{\eta} \right\}. 
	\]
Multiplying $\eta$ by the above equality, it follows from \eqref{eq:area} that  
	\[
		\frac{\xi_{\eta}'}{\eta} \eta^{6} 
		\int_{-\pi}^{\pi} \int_{-\pi}^{\pi} 
		\frac{(\sqrt{2}+\cos \tvph) f(\tvph, \tth, \eta)}
		{|Y(\tvph,\tth,\eta)|^{6}} d \tvph d \tth 
		= 8 \sqrt{2} \pi^{2}. 
	\]
Therefore, to prove \eqref{eq:deri-xi-eta}, it suffices to show 
	\begin{equation}\label{eq:eta6-Y6}
		\eta^{6} 
		\int_{-\pi}^{\pi} \int_{-\pi}^{\pi} 
		\frac{(\sqrt{2}+\cos \tvph) f(\tvph, \tth, \eta)}
		{|Y(\tvph,\tth,\eta)|^{6}} d \tvph d \tth 
		= 2^{15/4} \pi^{5/2} + O(\eta^{2}).
	\end{equation}

		To this end, we use the following decomposition: 
	\[
		I_{\eta} := \left\{ (\tvph, \tth ) \in [-\pi,\pi]^{2} \; : \;
		\tvph^{2} + \left( \sqrt{2} + 1 \right)^{2} \tth^{2} \leq \eta^{2} 
		\right\}, \qquad 
		J_{\eta} := [-\pi,\pi]^{2} \setminus I_{\eta}. 
	\]
First, we show 
	\begin{equation}\label{eq:eta6-Y6-J-eta}
		\eta^{6} \int_{J_{\eta}} 
		\frac{(\sqrt{2}+\cos \tvph) f(\tvph, \tth, \eta)}
		{|Y(\tvph,\tth,\eta)|^{6}} d \tvph d \tth 
		=O(\eta^{4}). 
	\end{equation}
By a Taylor expansion at the origin, 
we notice that  
	\begin{equation}\label{eq:exp-Y}
		\begin{aligned}
			|Y(\tvph,\tth,\eta)|^{2} 
			&= ( \sqrt{2}+1 )^{2} \tth^{2} + \tvph^{2} 
			+ \xi_{\eta}^{2} + \left( \tvph^{2} + ( \sqrt{2}+1 ) \tth^{2} \right) 
			\xi_{\eta} + O( \tvph^{4} + \tth^{4} ), 
			\\ 
			f(\tvph,\tth,\eta) 
			&= 
			\tvph^{2} + (\sqrt{2}+1) \tth^{2} + 2 \xi_{\eta} 
			+ O(\tvph^{4} + \tth^{4}). 
		\end{aligned}
	\end{equation}
Since  by Lemma \ref{104} $(i)$ it holds $\xi_{\eta} = A \eta^{2} + O(\eta^{4})$,  
where $A>0$, there exist $C_{0}, C_{1}>0$, 
which are independent of $\eta \in (0,1]$, such that 
	\[
		\begin{aligned}
			|Y(\tvph, \tth, \xi_{\eta})|^{2} 
			&\geq  C_{0} \left( \tvph^{2} + ( \sqrt{2}+1 )^{2} \tth^{2} \right),
			\\
			|f(\tvph,\tth,\xi_{\eta})| &\leq 
			C_{1} ( \tvph^{2} + (\sqrt{2}+1)^{2} \tth^{2} )
		\end{aligned}
	\]
for every $(\tvph,\tth) \in J_\eta$. 
Thus, using the change of variables 
$(\tvph,\tth) = \left(r \cos \Theta, (\sqrt{2}+1)^{-1} r \sin \Theta 
\right)$ and 
noting that $J_{\eta} \subset \{ (r,\Theta) \ |\  \eta \leq r \leq \pi, \ 
0 \leq \Theta \leq 2 \pi \}$, we have (by definition of $J_\eta$)
	\[
		\begin{aligned}
			\int_{J_{\eta}} 
			\frac{(\sqrt{2}+\cos \tvph) f(\tvph, \tth, \eta)}
			{|Y(\tvph,\tth,\eta)|^{6}} d \tvph d \tth 
			\leq C_{2} \int_{J_{\eta}} \frac{d \tvph d \tth} 
			{\left\{ \tvph^{2} + (\sqrt{2}+1)^{2} \tth^{2} \right\}^{2}} 
			\leq C_{2} 
			\int_{\eta}^{\pi} \int_{0}^{2\pi} \frac{1}{r^{3}} dr d \Theta
			\leq C_{3} \eta^{-2}. 
		\end{aligned}
	\]
Multiplying by $\eta^{6}$, we get \eqref{eq:eta6-Y6-J-eta}.

		For the integral on $I_{\eta}$, 
we consider the following two quantities:
	\[
		\widehat{I}_{1} := \eta^{6} \int_{I_{\eta}} 
		\frac{(\sqrt{2}+1) ( \tvph^{2} + (\sqrt{2}+1) \tth^{2})}
		{|Y(\tvph,\tth,\eta)|^{6}} d \tvph d \tth, \quad \qquad 
		\widehat{I}_{2} := \eta^{6} \int_{I_{\eta}} 
		\frac{(\sqrt{2}+1) 2 \xi_{\eta}}{|Y(\tvph,\tth,\eta)|^{6}}
		d \tvph d \tth. 
	\]
We first claim $\widehat{I}_{1}=O(\eta^{2})$. In fact, noting 
\eqref{eq:exp-Y} and $\xi_{\eta} = A\eta^{2} + O(\eta^{4})$, and 
recalling the above computations  together with 
$I_{\eta} = \{ (r, \Theta) \; : \; 0 \leq r \leq \eta, \ 0 \leq \Theta \leq 2 \pi\}$, 
one has 

	\[
		\widehat{I}_{1} \leq C_{4} \eta^{6} \int_{I_{\eta}} 
		\frac{d \tvph d \tth}
		{\left\{ \tvph^{2} + (\sqrt{2}+1)^{2} \tth^{2} + \xi_{\eta}^{2} 
		\right\}^{2}} 
		\leq C_{5} \eta^{6} \int_{0}^{\eta} \frac{r}{(r^{2} + \xi_{\eta}^{2})^{2}} d r
		\leq C_{6} \eta^{6} \frac{1}{\xi_{\eta}^{2}} \leq C_{7} \eta^{2}.
	\]

		Next, we compute $\widehat{I}_{2}$. 
First, it follows from \eqref{eq:exp-Y} and $\xi_\eta = O(\eta^2)$ that 
	\[
		\left\{ 1 - O (\eta^{2}) \right\}
		\left\{ \tvph^{2} + ( \sqrt{2} + 1 )^{2} \tth^{2} \right\} + \xi_\eta^2 
		\leq 
		|Y(\tvph,\tth,\xi_{\eta})|^{2} 
		\leq 
		\left\{ 1 + O(\eta^{2}) \right\} 
		 \left\{ \tvph^{2} + ( \sqrt{2} + 1 )^{2} \tth^{2} \right\} + \xi_\eta^2 
	\]
for all $ (\tvph,\tth) \in I_{\eta}$. 
Therefore, instead of $\widehat{I}_{2}$, it suffices to compute 
	\[
		\widehat{I}_{3} := \eta^{6} \int_{I_{\eta}} \frac{(\sqrt{2}+1) 2 \xi_{\eta}}
		{\left[ \left\{1 + O(\eta^{2}) \right\} 
		\left\{ \tvph^{2} + (\sqrt{2}+1)^{2} \tth^{2}   \right\} + \xi_\eta^2 \right]^{3}}
		d \tvph d \tth.
	\]
Using the same change of variables as above, we get 
	\[
		\begin{aligned}
			\widehat{I}_{3} 
			&= 
			\eta^{6} \int_{0}^{\eta} \int_{0}^{2\pi} 
			\frac{2 (\sqrt{2}+1) \xi_{\eta} r}
			{\left[ \left\{1+O(\eta^{2}) \right\} r^{2} + \xi_{\eta}^{2}\right]^{3}}
			\frac{d r d \Theta}{\sqrt{2}+1} 
			= \eta^{6} 4 \pi \xi_{\eta} 
			\int_{0}^{\eta} \frac{r}{ \left[ \{ 1 + O(\eta^{2}) \} r^{2} + \xi_{\eta}^{2} 
			\right]^{3} } dr 
			\\
			&= \frac{\pi}{1 + O(\eta^{2})} \eta^{6} \xi_{\eta}
			\left\{ \frac{1}{\xi_{\eta}^{4}} 
			- \frac{1}{ \left( \eta^{2} + \xi_{\eta}^{2} + O(\eta^{4}) \right)^{2} } 
			\right\}
			= \pi \left( 1 + O(\eta^2) \right) 
			\left\{ \frac{\eta^6}{\xi_\eta^3} - 
			\frac{\eta^6 \xi_\eta}{\left( \eta^2 + \xi_\eta^2 + O(\eta^4) \right)^2} \right\}. 
		\end{aligned}
	\]
Recalling Lemma \ref{104} $(i)$, there holds 
	\[
		\frac{\eta^{6}}{\xi_{\eta}^{3}} = 2^{15/4} \pi^{3/2} + O(\eta^{2}), 
		\quad 
		\frac{\eta^{6} \xi_{\eta}}
		{\left\{ \eta^{2} + \xi_{\eta}^{2} + O(\eta^{4}) \right\}^{2} } 
		= O(\eta^{4}). 
	\]
Hence, 
one observes that 
	\[
		\widehat{I}_{2} = \widehat{I}_{3} + O(\eta^{2}) 
		=  2^{15/4} \pi^{5/2} + O(\eta^{2}).   
	\]
Since we have 
	\[
		(\sqrt{2}+\cos \tvph) f(\tvph,\tth,\eta) 
		= (\sqrt{2} + 1 ) ( \tvph^{2} + (\sqrt{2}+1) \tth^{2} + 2 \xi_{\eta}) 
		+ O(\eta^{4}) \qquad \text{on}\ I_\eta,
	\]
noting $\xi_{\eta}=O(\eta^{2})$ 
and the estimates of $\widehat{I}_{1}$ and $\widehat{I}_{2}$, 
\eqref{eq:eta6-Y6} follows. 
Since \eqref{eq:eta6-Y6} implies \eqref{eq:deri-xi-eta}, Lemma \ref{l:xi-xi'} holds. 
\end{proof}

\subsection{Jacobi field generated by M\"obius inversions}\label{ss:32}

Here we analyse the variation of  M\"obius inversions on degenerating tori. In particular we 
derive  the asymptotics of the normal vector field induced by this variation. 
Define 
\[
		\Phi_{\eta}(x) := \frac{\eta^{2}}{|x|^{2}} x; \qquad \quad 
		\Psi_{\eta}(x) := (\Refx \circ \Phi_{\eta})(x), 
	\]
where $\Refx$ stands for the reflection 
$\Refx (y) := y - 2 \langle y, \mathbf{e}_{x} \rangle \mathbf{e}_{x}$. 
Recall that, for $\xi >0$, we have set 
	\[
		\T_{\xi} := \T - (\sqrt{2}+1 + \xi ) \mathbf{e}_{x}. 
	\]
Recall also that we used the following parametrizations of $\T$ and $\T_{\xi_{\eta}}$: 
for $( \tvph , \tth ) \in [0,2\pi]^2$, 
\begin{align} 
			X(\tvph ,\tth ) 
			&= \left( ( \sqrt{2} + \cos \tvph ) \cos \tth, \ 
			( \sqrt{2} + \cos \tvph ) \sin \tth, \ 
			\sin \tvph  \right), 
			\nonumber
			\\
			Y(\tvph,\tth, \eta)  \label{eq:YYY}
			&= X(\tvph , \tth ) 
			- \left( \sqrt{2} + 1 + \xi_{\eta} \right) \mathbf{e}_{x} \\ 
			& = \nonumber  
			\left( ( \sqrt{2} + \cos \tvph ) \cos \tth 
			- ( \sqrt{2} + 1 + \xi_{\eta} ) , \ 
			( \sqrt{2} + \cos \tvph ) \sin \tth, \ 
			\sin \tvph  \right). 
		\end{align}
As  unit normal to $\T_{\xi_{\eta}}$, we choose the outward one 
	\[
		n(\tvph, \tth ) 
		= ( \cos \tvph \cos \tth, \ 
		\cos \tvph \sin \tth, \ \sin \tvph ) .
	\]
We put also 
	\[
			\mathcal{Z}( \tvph, \tth, \eta ) 
			:= \Psi_{\eta} ( Y(\tvph , \tth, \eta) ), 
			\quad 
			n_{0,\eta} ( \tvph, \tth ) 
			:= \frac{ (D_{x} \Psi_\eta) (Y(\tvph, \tth, \eta)) [n]} 
			{ |(D_{x} \Psi_\eta) (Y( \tvph, \tth, \eta)) [n] |} 
			= \frac{\left(\Refx \circ (D_x\Phi_{\eta})(Y(\tvph, \tth, \eta))
			\right)[n]}{|(D_x\Phi_{\eta})(Y(\tvph, \tth, \eta))[n]|}.
	\]
Recalling Lemma \ref{104}, we easily see that 
$\mathcal{Z}(\eta^2 \bvph , \eta^2 \bth, \eta) 
\to \Refx \circ Z_0( \bvph , \bth )$ in $C^\infty_{\rm loc}(\R^2)$ as $\eta \to 0$ and 
$n_{0,\eta}$ is an outward unit normal to $\Psi_{\eta}(\T_{\xi_{\eta}})$ 
since $\Psi_{\eta}$ is a conformal map.
Finally, we define the normal component of the variation of $\Psi_{\eta}$ by 
\begin{equation}\label{eq:phieta}
\varphi_{\eta}(\tvph, \tth) := 
		\left\langle \frac{\partial \mathcal{Z}}{\partial \eta}( \tvph, \tth, \eta), 
		\ n_{0,\eta}( \tvph, \tth )  \right\rangle. 
\end{equation}
Our aim here is to prove the next proposition. 
	\begin{pro}\label{p:psi0}
		Set $2 \tilde A := \lim_{\eta \to 0} \eta^{2} / \xi_{\eta} > 0$ and 
			\[
				\psi_{\eta}( \bvph, \bth ) 
				:= \frac{ \varphi_{\eta} ( \eta^{2} \bvph, \eta^{2} \bth ) }
				{\eta} \quad 
				{\rm for}\ (\bvph, \bth) \in \R^2. 
			\]
		Then there holds 
			\[
				\begin{aligned}
					\psi_{\eta} ( \bvph, \bth ) 
					\to \psi_{0}(\bvph, \bth) 
					&= 
					 - \frac{1}{ \bvph^{2} + (\sqrt{2} + 1)^{2} \bth^{2} 
					+ 1/ ( 4 \tilde{A}^{2})} 
					\left\{ \bvph^{2} + (\sqrt{2}+1) \bth^{2} 
					- \frac{\sqrt{2}}{8 \tilde{A}^{2}} \right\} 
					\\
					&\quad {\rm in}\ C^{\infty}_{\rm loc}(\R^{2}).
				\end{aligned}
			\]
		Noting that as $\eta \to 0$, $\mathcal{Z}(\T_{\xi_{\eta}})$ 
		converges to the sphere of radius $\tilde{A}$ and 
		centred at $\tilde A \mathbf{e}_{x}$ (denoted by $S^{2}_{\tilde A}$), 
		we observe that $\psi_0 \in C^\infty (\SA \setminus \{0\}) 
		\cap L^\infty(\SA)$. 
		Moreover, using the following polar coordinates 
			\[
				x = \tilde{A} ( 1 + \cos \theta ), \quad 
				y = \tilde{A} \sin \theta \cos \varphi, \quad 
				z = \tilde{A} \sin \theta \sin \varphi, \quad 
				(\theta, \varphi) \in [0,\pi] \times [0,2\pi],
			\]
		$\psi_{0}$ is expressed as follows: 
\begin{equation}\label{eq:tranfsphere} \
\begin{aligned}
					\psi_{0} 
					&= 
					-\frac{1}{2 \tilde{A} x} 
					\left(
					z^{2} + y^{2} - (2 - \sqrt{2}) y^{2} \right) 
					+ \frac{\sqrt{2}}{4 \tilde{A}} x
					\\
					&= \frac{\sqrt{2}}{2} \cos \theta 
					+ \frac{2-\sqrt{2}}{4} ( 1 - \cos \theta ) \cos 2 \varphi 
					\\
					&= \frac{1}{2} (  \cos \theta - 1 ) + \frac{2 - \sqrt{2}}{2} 
					( 1 - \cos \theta ) \cos^{2} \varphi 
					+ \frac{\sqrt{2}}{4} ( 1 + \cos \theta ).
				\end{aligned}  
\end{equation}
	\end{pro}

%
%
%
%

\

\noindent To prove the above proposition we need two preliminary lemmas.

\begin{lem}\label{l:bdd-psi}
For each $k \in \N$ and $R>0$, $(\psi_{\eta})$ is bounded in $C^{k}([-R,R]^{2})$ 
as $\eta \to 0$.  
\end{lem}

	\begin{proof}
We first show that $\varphi_{\eta}$ is expressed as follows: 
	\begin{equation}\label{eq:phi-eta}
		\begin{aligned}
			\varphi_{\eta} &= - 
			\frac{\eta}{|Y(\tvph, \tth,\eta)|^{2}} 
			\left\{ 2 \left( \sqrt{2} \cos \tvph + 1 
			- ( \sqrt{2}+1) \cos \tvph \cos \tth \right) 
			+ ( \xi_{\eta}' \eta - 2 \xi_{\eta}) \cos \tvph \cos \tth 
			\right\}
			\\
			&=: - \frac{\eta}{|Y(\tvph, \tth,\eta)|^{2}} 
			\left( h(\tvph, \tth) 
			+ (\xi_{\eta}' \eta - 2 \xi_{\eta}) \cos \tvph \cos \tth 
			\right).
		\end{aligned} 
	\end{equation}

		To this end, from the definition of $\mathcal{Z}$, we have 
	\[
		\begin{aligned}
			\frac{\partial \mathcal{Z}}{\partial \eta} 
			&= 2 \eta \frac{\Refx(Y(\tvph, \tth,\eta))}
			{|Y(\tvph, \tth,\eta)|^{2}} 
			+ \left( \Refx \circ D_x \Phi_{\eta} 
			(Y(\tvph, \tth,\eta)) \right)
			 \left[ \frac{\partial Y}{\partial \eta} \right]
			 \\
			 &=2 \eta \frac{\Refx(Y(\tvph, \tth,\eta))}
			{|Y(\tvph, \tth,\eta)|^{2}} 
			+ \left( \Refx \circ D_x \Phi_{\eta} 
			(Y(\tvph, \tth,\eta)) \right)
			 \left[ - \xi_{\eta}' \mathbf{e}_{x} \right]. 
		\end{aligned}
	\]
From the fact that $\Refx \in O(3)$ and the formula
	\begin{equation}\label{eq:DPhi}
		D_x\Phi_{\eta}(x) = \frac{\eta^{2}}{|x|^{2}} 
		\left( {\rm Id}_{\R^{3}} - 2 \frac{x}{|x|} \otimes \frac{x}{|x|} \right), 
	\end{equation}
it follows that 
	\[
		| D_x \Psi_{\eta} (Y) [n] | = 
		| D_x \Phi_{\eta} (Y) [n] | = \frac{\eta^{2}}{|Y|^{2}}. 
	\]
Since $D_x\Phi_{\eta}$ is conformal (cf. \eqref{eq:DPhi}), we obtain 
	\[
		\begin{aligned}
			\varphi_{\eta} &= 
			\left\langle \frac{\partial \mathcal{Z}}{\partial \eta}, \ 
			n_{0,\eta} \right\rangle 
			= \left\langle  2 \eta \frac{\Refx(Y)}
			{|Y|^{2}} 
			+ 
			\left( \Refx \circ D \Phi_{\eta} (Y) \right)
			 \left[ - \xi_{\eta}' \mathbf{e}_{x} \right], \ 
			 n_{0,\eta}
			   \right\rangle 
			 \\
			 &=
			 \left\langle 
			 2 \eta \frac{\Refx(Y)}{|Y|^{2}}, \ 
			 n_{0,\eta} 
			 \right\rangle
			 - \xi_{\eta}' 
			 \left\langle D\Phi_{\eta}(Y)[\mathbf{e}_{x}], \ 
			 \frac{D\Phi_{\eta}(Y)[n]}{|D\Phi_{\eta}(Y)[n]|} 
			   \right\rangle 
			   \\
			  &= 2 \eta \left\langle  \frac{\Refx(Y)}{|Y|^{2}}, \ 
			 n_{0,\eta} \right\rangle
			 - \xi_{\eta}' \frac{\eta^{4}}{|Y|^{4}} 
			 \frac{1}{|D\Phi_{\eta}(Y)[n]|} 
			 \langle \mathbf{e}_{x}, n \rangle 
			 \\
			 &= 2 \eta \left\langle  \frac{\Refx(Y)}{|Y|^{2}}, \ 
			 n_{0,\eta} \right\rangle 
			 - \frac{\eta^{2}}{|Y|^{2}} \xi_{\eta}' \langle \mathbf{e}_{x},  n \rangle.
		\end{aligned}
	\]

		On the other hand, 
using \eqref{eq:DPhi}, one sees that  
	\[
		n_{0,\eta} = 
		\Refx \left( \frac{D\Phi_{\eta}(Y)[n]}{|D\Phi_{\eta}(Y)[n]|} \right) 
		= 
		\Refx \left( n - 2 \left\langle \frac{Y}{|Y|}, \ n \right\rangle 
		\frac{Y}{|Y|} \right).
	\]
Thus it follows that 
	\[
		\left\langle \frac{\Refx(Y)}{|Y|^{2}}, \ n_{0,\eta} \right\rangle 
		= \frac{1}{|Y|^{2}} 
		\left\langle Y, \ n - 2 \left\langle \frac{Y}{|Y|}, n \right\rangle 
		\frac{Y}{|Y|} \right\rangle 
		= - \frac{1}{|Y|^{2}} \langle Y, n \rangle,
	\]
which implies 
	\[
		\varphi_{\eta} 
		= \frac{1}{|Y|^{2}} 
		\left( -2 \eta \langle Y, n \rangle - \eta^{2} \xi_{\eta}' 
		\langle \mathbf{e}_{x} , n \rangle  \right)
		= -\frac{\eta}{|Y|^{2}} 
		\left( 2 \langle Y, n \rangle + \eta \xi_{\eta}' 
		\langle \mathbf{e}_{x} , n \rangle  \right). 
	\]
Noting that 
	\[
		\begin{aligned}
			\langle \mathbf{e}_{x}, n \rangle 
			= \cos \tvph 
			\cos \tth, \quad 
			\langle Y, n \rangle 
			&= \left\langle 
			\begin{pmatrix} (\sqrt{2}+\cos \tvph) \cos \tth 
			-(\sqrt{2}+1+\xi_{\eta}) \\
			(\sqrt{2}+\cos \tvph) \sin \tth \\
			\sin \tvph
			 \end{pmatrix}
			 ,
			 \begin{pmatrix}
			 	\cos \tvph \cos \tth \\
				\cos \tvph \sin \tth \\
				\sin \tvph
			 \end{pmatrix}
			 \right\rangle
			\\
			&= \sqrt{2} \cos \tvph + 1 
			- ( \sqrt{2} + 1) \cos \tvph \cos \tth 
			- \xi_{\eta} \cos \tvph \cos \tth,
		\end{aligned}
	\]
we get \eqref{eq:phi-eta}.

		Next, recalling the definition of $h$ in \eqref{eq:phi-eta} and using 
a Taylor expansion, one observes that 
	\begin{equation}\label{eq:Y-h}
		Y(\tvph, \tth, \eta )
		= 
		\begin{pmatrix} 
			- \xi_{\eta}  \\ 
			( \sqrt{2} + 1 ) \tth \\
			\tvph
		\end{pmatrix}
		+ R_{Y} ( \tvph, \tth ),
		\qquad 
		h(\tvph, \tth ) 
		= 
		\tvph^{2} + (\sqrt{2}+1) \tth^{2}  
		+ R_{h} ( \tvph, \tth )
	\end{equation}
where $R_{Y}, R_{h}$ are smooth functions satisfying
	\begin{equation}\label{eq:RYRh}
		\left| D^{\alpha} R_{Y} ( \tvph , \tth ) \right| 
		\leq
		C_k \left( \tvph^{(2- | \alpha |)_+ } 
		+ \tth^{(2 - |\alpha|)_+ } \right),
		\qquad 
		\left| D^{\beta} R_{h} ( \tvph, \tth ) \right| 
		\leq C_k \left( \tvph^{(4 - | \beta |)_+ } 
		+ \tth^{(4 - |\beta|)_+ } \right) 
	\end{equation}
for all $\alpha, \beta \in \Z_{+}^{2}$ with $|\alpha|, |\beta| \leq k$ 
in a neighbourhood of the origin. Therefore, for $(\bvph, \bth) \in \R^2$, we have 
	\[
		\begin{aligned}
			Y(\eta^{2} \bvph, \eta^{2} \bth, \eta ) 
			&= \eta^{2}
			\left( -\frac{1}{2 \tilde{A}}, \ 
				(\sqrt{2}+1) \bth, \ \bvph
			\right)
			+ R_{Y}(\eta^{2} \bvph, \eta^{2} \bth ) + 
			\left( \frac{\eta^{2}}{2 \tilde{A}} - \xi_{\eta} \right) {\bf e}_x,
			\\
			h( \eta^{2} \bvph, \eta^{2} \bth) 
			&= \eta^{4} ( \bvph^{2} + (\sqrt{2}+1) \bth^{2} ) 
			+ R_{h} ( \eta^{2} \bvph, \eta^{2} \bth ). 
		\end{aligned}
	\]
Notice that from \eqref{eq:RYRh} it follows that 
if $\eta \leq 1/R^4$, then 
	\[
		| D_{(\bvph, \bth)}^{\alpha} \left( R_{Y} 
		(\eta^{2} \bvph, \eta^{2} \bth ) \right) | 
		\leq C_{|\alpha|} \eta^{7/2} , 
		\qquad 
		| D_{(\bvph, \bth)}^{\alpha} \left( R_{h} 
		( \eta^{2} \bvph, \eta^{2} \bth ) \right) | 
		\leq C_{|\alpha|} \eta^{7}
	\]
for all $\alpha \in \Z_{+}^{2}$ and 
$ ( \bvph, \bth ) \in [- R, R]^{2}$ where $C_{|\a |}$ depends only on 
$|\alpha|$. 
By $ 2 \tilde{A} = \lim_{\eta \to 0} ( \eta^{2} / \xi_{\eta} )$ and 
Lemma \ref{104} (i), we have $ \eta^{2} / ( 2 \tilde{A} ) - \xi_{\eta} = O(\eta^{4})$. 
Hence, 
	\[
		\eta^{-2} Y( \eta^{2} \bvph, \eta^{2} \bth, \eta ) 
		= \left( -\frac{1}{2 \tilde{A}}, \ (\sqrt{2}+1) \bth, \ 
			\bvph\right) 
			+ \bar{R}_{Y} ( \bvph, \bth ), 
			\quad 
		\eta^{-4} h( \eta^{2} \bvph, \eta^{2} \bth ) 
		= \bvph^{2} + (\sqrt{2}+1) \bth^{2}   
			+ \bar{R}_{h} ( \bvph, \bth )
	\]
where $\bar{R}_{Y} = O_k(\eta^{3/2})$ and $\bar{R}_{h} = O_k(\eta^{3})$ 
in $C^{k}([-R,R]^{2})$ sense provided $\eta \leq 1/R^4$. 
Here $O_k(\eta^i)$ means $\| O_k (\eta^i) \|_{C^k( [-R,R]^2 )} \leq C_k \eta^i$ and 
$C_k$ does not depend on $R$. 
Hence, one observes that 
	\[
			\psi_{\eta}(\bth, \bvph ) 
			= 
			\frac{\varphi_{\eta}( \eta^{2} \bth, \eta^{2} \bvph )}{\eta} 
			= - \frac{\bvph^{2} + (\sqrt{2}+1) \bth^{2} 
			+\eta^{-4} (\xi_{\eta}' \eta - 2 \xi_{\eta}) 
			\cos ( \eta^{2} \bvph ) \cos ( \eta^{2} \bth ) }
			{(\sqrt{2}+1)^{2}\bth^{2} + \bvph^{2} 
			+ \frac{1}{4 \tilde{A}^{2}}}  + R_{\psi} ( \bvph, \bth )
	\]
where $R_{\psi} ( \bvph, \bth ) =O_k(\eta^{5/4})$. 
By Lemma \ref{l:xi-xi'}, there holds 
$\xi_{\eta}'\eta - 2 \xi_{\eta} = O(\eta^{4})$, hence, 
	\begin{equation}\label{eq:psi-eta}
		\psi_{\eta} ( \bvph, \bth ) 
		= - \frac{\bvph^{2} + (\sqrt{2}+1) \bth^{2} 
			+\eta^{-4} (\xi_{\eta}' \eta - 2 \xi_{\eta}) }
			{(\sqrt{2}+1)^{2}\bth^{2} + \bvph^{2} 
			+ \frac{1}{4 \tilde{A}^{2}}}  + O_k(\eta^{5/4}), 
	\end{equation}
which implies that $(\psi_{\eta})$ is bounded in $C^{k}([-R,R]^{2})$. 
\end{proof}

		From Lemmas \ref{l:xi-xi'} and \ref{l:bdd-psi}, and \eqref{eq:psi-eta}, 
taking a subsequence $(\eta_{k})$, we may assume 
	\begin{equation}\label{eq:psi0}
		\psi_{\eta_{k}} \to \psi_{0} = 
		- \frac{ \bvph^{2} + (\sqrt{2}+1) \bth^{2} + c_{0}}
		{(\sqrt{2}+1) \bth^{2} + \bvph^{2} + 
		\frac{1}{4 \tilde{A}^{2}}} 
		\qquad 
		{\rm in}\ C^{\ell}_{\rm loc}(\R^{2})
	\end{equation}
for every $\ell \in \N$, where 
	\[
		c_{0} = \lim_{k \to \infty} \eta_k^{-4} ( \xi_{\eta_{k}}' \eta_{k} 
		- 2 \xi_{\eta_{k}} ). 
	\]
Note that $\psi_0$ is a bounded function. 
Furthermore, using the map 
	\[
		( \bvph , \bth ) \mapsto 
		\Psi_1 \left( -\frac{1}{2 \tilde A} \mathbf{e}_x + (\sqrt{2} + 1) \bth \mathbf{e}_y 
		+ \bvph \mathbf{e}_z \right) : 
		\R^2 \to S^2_{\tilde A}
	\]
as the parametrization of $S^2_{\tilde A}$, 
by the conformality of $\Psi_1$ and \eqref{eq:DPhi}, we have 
	\[
		d \s = \frac{\sqrt{2} + 1}{\left\{ (\sqrt{2}+1)^2 \bth^2 
			+ \bvph^2 + \frac{1}{4 \tilde{A}^2} \right\}^2 } 
			d \bvph d \bth. 
	\]
Hence, $\psi_0$ is integrable on $S^2_{\tilde A}$. 

\

\noindent Next we prove that the function $\psi_{0}$  has null mean value on the limit sphere  $S^{2}_{\tilde A}$.

	\begin{lem} \label{l:int0S2}
	Viewed as a real function on the limit sphere (through the above parameterization), 
	the function $\psi_{0}$ in \eqref{eq:psi0} satisfies 
		\begin{equation}\label{eq:cond-psi0}
			\int_{S^{2}_{\tilde A}} \psi_{0} \, d \sigma = 0.
		\end{equation}
	\end{lem}

	\begin{rem}\label{r:conv-psi-eta}
		From \eqref{eq:cond-psi0} we shall prove 
		$c_0 = - \sqrt{2}/ ( 8 \tilde{A}^2 )$ in the proof of Proposition \ref{p:psi0}. 
		Here we remark that since $c_{0}$ is independent of the choice of subsequence $(\eta_{k})$, 
		as $\eta \to 0$, we obtain 
		\[
		\eta^{-4} ( \xi_{\eta}' \eta - 2 \xi_{\eta} ) \to  
		- \frac{\sqrt{2}}{8 \tilde{A}^{2}}, \qquad 
		\psi_{\eta} \to \psi_{0} \quad 
		{\rm in}\ C^{\ell}_{\rm loc} (\R^{2}). 
		\]
	\end{rem}

	\begin{proof}[Proof of Lemma \ref{l:int0S2}]
We argue by contradiction and suppose that 
	\[
		\int_{S^{2}_{\tilde A}} \psi_{0} \, d \sigma = A \neq 0.
	\]
Set $\Sigma_{\eta} := \Psi_{\eta}(\T_{\xi_{\eta}})$. 
Since $\Psi_{\eta}$ preserves the area of $\T_{\xi_{\eta}}$, 
it follows from the definition of $\varphi_{\eta}$ that 
	\[
		0= \frac{d}{d \eta}  {\rm Area} (\Sigma_{\eta}) 
		= \int_{\Sigma_{\eta}} H_{\Sigma_{\eta}} \varphi_{\eta} d \sigma, 
	\]
where $H_{\Sigma_{\eta}}$ is the mean curvature of  $\Sigma_{\eta}$.

		Let $\delta > 0$ and decompose $\Sigma_{\eta}$ into 
two parts:
	\[
		\Sigma_{\eta} = ( \Sigma_{\eta} \cap B_{\delta}(0)) 
		\cup ( \Sigma_{\eta} \cap (B_{\delta}(0))^{c} ) 
		=: \Sigma_{\eta,1} + \Sigma_{\eta,2}. 
	\] 
First we prove that there exist $C>0$, independent of $\eta$ and $\delta$,  
and $\eta_{0},\delta_0>0$ such that 
	\begin{equation}\label{eq:sig-1}
		\left| 
		\int_{\Sigma_{\eta,1}} H_{\Sigma_{\eta}} \varphi_{\eta} d \sigma 
		\right|
		\leq C \delta \eta
	\end{equation}
for each $\eta \leq \eta_{0}$ and $\delta \leq \delta_0$. 
In fact, from  H\"older's inequality, one observes that 
	\[		
		\left| 
		\int_{\Sigma_{\eta,1}} H_{\Sigma_{\eta}} \varphi_{\eta} d \sigma 
		\right| 
		\leq \left( \int_{\Sigma_{\eta,1}}  H_{\Sigma_{\eta}}^{2} d \sigma \right)^{1/2} 
		\left( \int_{\Sigma_{\eta,1}}  \varphi_{\eta}^{2} d \sigma \right)^{1/2} 
		\leq \left( \int_{\Sigma_{\eta}}  H_{\Sigma_{\eta}}^{2} d \sigma \right)^{1/2} 
		\left( \int_{\Sigma_{\eta,1}}  \varphi_{\eta}^{2} d \sigma \right)^{1/2}. 
	\]
By the conformal invariance of the Willmore functional (see Proposition \ref{p:Mobinv}), we have 
	\[
		\int_{\Sigma_{\eta}}  H_{\Sigma_{\eta}}^{2} d \sigma 
		= W_{g_0}(\Sigma_{\eta}) = W_{g_0}(\T). 
	\]
On the other hand, we remark that 
$|\mathcal{Z}(\tvph, \tth,\eta)| \leq \delta$ is equivalent to 
$|Y(\tvph,\tth,\eta)| \geq \eta^{2} \delta^{-1}$. 
Since $\Sigma_\eta$ is parametrized by $\mathcal{Z}(\tvph, \tth,\eta)$ 
($ (\tvph, \tth) \in [-\pi, \pi]^2 $) and 
$\Phi_\eta$ is conformal, we observe that 
the area element of $\mathcal{Z}$ is given by 
	\[
		d \sigma = \left( \sqrt{2} + \cos \tvph   \right) 
		\frac{\eta^4}{|Y(\tvph, \tth, \eta) |^4 }
		d \tvph d \tth. 
	\]
Hence, we have 
	\[
		\begin{aligned}
			\int_{\Sigma_{\eta,1}} \varphi_{\eta}^{2} d \sigma 
			= \int_{\T_{\xi_{\eta}} \cap 
			\Psi_{\eta}^{-1} (B^{c}_{\eta^{2} \delta^{-1}})} 
			\varphi_{\eta}^{2} ( \tvph, \tth ) 
			( \sqrt{2} + \cos \tvph  ) \frac{\eta^{4}}
			{|Y(\tvph, \tth,\eta)|^{4}} 
			d \tth d \tvph. 
		\end{aligned}
	\]

		Next, recalling \eqref{eq:exp-Y} (or \eqref{eq:Y-h}), 
we may find $C_{0}, C_{1}>0$ and $\eta_{0}>0$ such that 
	\begin{equation}\label{eq:|Y|}
		C_{1} ( \tvph^{2} + (\sqrt{2}+1)^{2} \tth^{2} 
		+ \xi_{\eta}^{2} ) 
		\geq | Y( \tvph, \tth, \eta ) |^{2} 
		\geq C_{0} ( \tvph^{2} + (\sqrt{2}+1)^{2} \tth^{2} )
	\end{equation}
for all $\tvph, \tth \in [-\pi,\pi]$ and $\eta \leq \eta_{0}$.
Noting that $\xi_{\eta}= \eta^{2}/ (2\tilde{A}) + O(\eta^{4})$ 
by \eqref{eq:deri-xi-eta}, 
we may assume that there exists $C_{2}>0$ (independent of $\delta$) 
satisfying
	\[
		\eta \leq \eta_{0}, \quad 
		|Y(\tvph, \tth, \eta)|^{2} \geq \frac{\eta^{4}}{\delta^{2}} 
		\quad \Rightarrow \quad \tvph^{2} 
		+ (\sqrt{2}+1)^{2} \tth^{2} 
		\geq C_{2} \frac{\eta^{4}}{\delta^{2}}
	\]
for $0<\delta\leq \bar{\delta} (C_1)$.	
Moreover, we claim that 
	\begin{equation}\label{eq:est-var-eta}
		|\varphi_{\eta}( \tvph, \tth ) | 
		\leq C_{3} \eta \qquad 
		{\rm for\ every}\ (\tvph, \tth) \in [-\pi,\pi]^{2}.
	\end{equation}
In fact, recall \eqref{eq:phi-eta}:
	\[
		\varphi_{\eta} ( \tvph, \tth ) 
		= - \frac{\eta}{|Y(\tvph, \tth, \eta )|^{2}} 
		\left[ h( \tvph, \tth )
		+ ( \eta \xi_{\eta}' - 2 \xi_{\eta} ) \cos \tvph \cos \tth 
		\right]. 
	\]
From Lemmas \ref{104} and \ref{l:xi-xi'}, and 
$|Y(\tvph, \tth, \eta)| \geq 
|\langle Y, \mathbf{e}_{x} \rangle | \geq \xi_{\eta} \geq C_{4} \eta^{2}$ 
for each $(\tvph, \tth) \in [-\pi,\pi]^{2}$, 
there holds 
	\[
		\frac{|(\eta \xi_{\eta}' - 2 \xi_{\eta}) \cos \tvph \cos \tth |}
		{|Y(\tvph, \tth, \eta )|^{2}} 
		\leq C_{5} \qquad 
		{\rm for\ each}\ (\tvph, \tth) \in [-\pi,\pi]^{2} \ 
		{\rm and}\ \eta \leq \eta_{0}. 
	\]

		On the other hand, by \eqref{eq:Y-h}, we have  
	\[
		\left| h( \tvph, \tth )
		\right| \leq C_{6} \left( \tvph^{2} + 
		( \sqrt{2}+1 )^{2} \tth^{2} \right) \qquad 
		{\rm for\ all}\ ( \tvph, \tth ) \in [-\pi,\pi]^{2}.
	\]
Thus by \eqref{eq:|Y|}, we obtain 
	\[
		\frac{ \left| h( \tvph, \tth ) \right| }
		{|Y(\tvph, \tth, \eta )|^{2}} 
		\leq C_{7}
	\]
for all $(\tvph, \tth) \in [-\pi,\pi]^{2}$ and $\eta \leq \eta_{0}$. 
Combining the two estimates above, \eqref{eq:est-var-eta} holds.

		Now setting 
	\[
		A_{\eta,\delta} = \{ (\tvph, \tth) \; : \;  
		C_{2} \eta^{4} \delta^{-2} \leq 
		\tvph^{2} + (\sqrt{2}+1)^{2} \tth^{2}
		\leq 8 \pi^{2} \},
	\]
and using \eqref{eq:est-var-eta} and polar coordinates as in the proof of Lemma \ref{l:xi-xi'}, we get 
	\[
		\int_{\Sigma_{\eta,1}} \varphi_{\eta}^{2} d \sigma 
		\leq 
		C_{8} \eta^{6} \int_{A_{\eta,\delta}} 
		\frac{1}{|Y|^{4}} d \tvph d \tth 
		\leq 
		C_{9} \eta^{6}\int_{A_{\eta,\delta}} 
		\left\{\tvph^{2} + (\sqrt{2}+1)^{2} \tth^{2} \right\}^{-2} 
		d \tth d \tvph 
		\leq C_{10} \eta^{6} \left( \frac{\delta}{\eta^{2}} \right)^{2} 
		= C_{10} \eta^{2} \delta^{2}
	\]
for all $\eta \leq \eta_{0}$. Thus \eqref{eq:sig-1} holds 
for some $\eta_0>0$ and $\delta_0=\bar{\delta}(C_1)>0$.

		Next, we consider the integral on $\Sigma_{\eta,2}$. 
We first remark that $|\mathcal{Z}(\tvph, \tth, \eta)| \geq \delta$ 
is equivalent to $|Y(\tvph, \tth, \eta)| \leq \eta^{2} \delta^{-1}$ and 
the following holds: (see \eqref{eq:|Y|})
	\[
		|Y(\tvph, \tth,\eta)| \leq \eta^{2} \delta^{-1} 
		\quad \Rightarrow \quad 
		\tvph^{2} + (\sqrt{2}+1)^{2}\tth^{2}  
		\leq C_{11}^{2} \eta^{4} \delta^{-2}
	\]
for every $\eta \leq \eta_{0}$. 
Since 
	\[
		\mathcal{Z}(\eta_{k}^{2} \bvph, \eta^{2}_{k} \bth, \eta_k ) 
		\to \Psi_{1} \left( -\frac{1}{2 \tilde A}, 
		(\sqrt{2} + 1 )  \bth, \bvph \right), 
		\quad 
		\eta^{-1}_{k} \varphi_{\eta_{k}} ( \eta_{k}^{2} \bvph, 
		\eta_{k}^{2} \bth ) \to \psi_{0} (\bvph, \bth)
		\  {\rm in}\ C^{\ell}([-C_{11}\delta^{-1},C_{11}\delta^{-1}]^{2})
	\]
for any $\ell \in \N$ and noticing that the maps 
	\[
		\mathcal{Z}(\eta_{k}^{2} \bvph, \eta_{k}^{2} \bth, \eta_k) 
		=\Psi_{\eta_{k}} ( Y(\eta_{k}^{2} \bvph, \eta_{k}^{2} 
		\bth, \eta_{k}) ) 
		= \Psi_{1} \left( \eta_{k}^{-2} Y(\eta_{k}^{2} \bvph, 
		\eta_{k}^{2} \bth, \eta_{k}) \right), \quad 
		( \bvph, \bth ) \mapsto 
		\Psi_{1} \left( -\frac{1}{2 \tilde A}, 
		(\sqrt{2} + 1 )  \bth, \bvph \right)
	\]
are parametrizations of $\Sigma_{\eta_{k},2}$ and $S_{\tilde{A}}^{2}$, 
we obtain
	\[
		H_{\Sigma_{\eta_{k}}}
		( \mathcal{Z} ( \eta_{k}^{2} \bvph, \eta_{k}^{2} \bth, \eta_k ) )
		\to \frac{2}{\tilde{A}}  \qquad 
		{\rm in}\ C^{0} ( [-C_{11} \delta^{-1}, C_{11} \delta^{-1}]^{2} ) 
	\]
and 
	\[
		\begin{aligned}
			\lim_{k \to \infty} 
			\eta^{-1}_{k} \left| \int_{\Sigma_{\eta_{k},2}} 
			H_{\Sigma_{\eta_{k}}}\varphi_{\eta_{k}} d \sigma 
			\right| 
			& 
			= \frac{2}{\tilde{A}} \left| \int_{S^{2}_{\tilde{A}}\cap B_{\delta}^{c}(0)} 
			\psi_{0} d \sigma 
			\right|. 
		\end{aligned}
	\]
Since $\psi_0$ is integrable, we may find $0 < \delta_2 \leq \delta_0=\bar{\delta}_0(C_1)$ 
so that if $\delta \leq \delta_2$, then 
	\[
		\frac{2}{\tilde{A}} \left| \int_{S^{2}_{\tilde{A}}\cap B_{\delta}^{c}(0)} 
		\psi_{0} d \sigma 
		\right| \geq \frac{|A|}{\tilde{A}} > 0. 
	\]
Therefore, by \eqref{eq:sig-1}, for all $\eta_{k} \leq \eta_{0}$ 
and $\delta \leq \delta_2$, it follows that 
	\[
		0 = \eta_{k}^{-1} \int_{\Sigma_{\eta_{k}}} 
		H_{\Sigma_{\eta_{k}}} \varphi_{\eta_{k}} d \sigma 
		= \eta_{k}^{-1} \int_{\Sigma_{\eta_{k},2}} H_{\Sigma_{\eta_{k}}} 
		\varphi_{\eta_{k}} d \sigma 
		+ \eta_{k}^{-1} \int_{\Sigma_{\eta_{k},1}} H_{\Sigma_{\eta_{k}}} 
		\varphi_{\eta_{k}} d \sigma 
		\left\{ 
			\begin{aligned}
				& \geq \frac{A}{\tilde{A}} - C \delta & &{\rm if}\ A>0,\\
				& \leq -\frac{A}{\tilde{A}} + C \delta & &{\rm if}\ A< 0.
			\end{aligned}
		\right.
	\]
Noting that $C$ does not depend on $\delta$, 
choosing $\delta>0$ sufficiently small and $k$ sufficiently large, 
we get a contradiction and the Lemma holds. 
	\end{proof}

\

\noindent	We are	now ready to prove Proposition \ref{p:psi0}

\begin{proof}[Proof of Proposition \ref{p:psi0}]
Notice first that 
	\[
		S^{2}_{\tilde A} \backslash \{0\}= 
		\tilde{A} ( S^{2} + \mathbf{e}_{x} ) \backslash \{0\}
		= \bigcup \left\{ \Psi_{1} \left( -\frac{1}{2\tilde A}, (\sqrt{2} + 1) y , z \right) 
		\; : \; y, z \in \R \right\}
	\]
and set 
	\[
		B := \bvph^{2} + (\sqrt{2}+1)^{2} \bth^{2} 
		+ \frac{1}{4 \tilde{A}^{2}}, 
		\quad 
		(x,y,z) := \Psi_{1}\left( -\frac{1}{2 \tilde{A}}, (\sqrt{2}+1)\bth, 
		\bvph \right) 
		= \left( \frac{1}{B} \frac{1}{2 \tilde{A}}, 
		\frac{(\sqrt{2}+1)\bth}{B}, 
		\frac{\bvph}{B} \right).
	\]
Since 
	\[
		(x - \tilde{A})^{2} + y^{2} + z^{2} = \tilde{A}^{2} 
		\quad \Leftrightarrow \quad 
		x^{2} - 2 \tilde{A} x + y^{2} + z^{2} = 0, 
	\]
by the definition of $B$ and $(x,y,z)$, we have 
	\[
		B = (Bz)^{2} + (By)^{2} + (Bx)^{2}, \qquad 
		\hbox{ therefore } \qquad 
		B = \frac{1}{x^{2} + y^{2} + z^{2}} = \frac{1}{2 \tilde{A} x}.
	\]
Recalling $c_{0} = \lim_{\eta_k \to 0} \eta^{-4}_k (\xi_{\eta_k}' \eta_k - 2 \xi_{\eta_k})$, 
it follows from the above formulas and \eqref{eq:psi0} that 
	\begin{equation}\label{eq:psi0-2}
		\begin{aligned}
			\psi_{0} 
			&= - \frac{1}{B} 
			\left( B^{2} z^{2} + \frac{B^{2}}{\sqrt{2}+1} y^{2} 
			+ c_{0}  \right) 
			= - B \left( z^{2} + \frac{y^{2}}{\sqrt{2}+1} \right) 
			 - 2 \tilde{A} c_{0} x
			\\
			&= -\frac{1}{2 \tilde{A} x} 
			\left(
			z^{2} + y^{2} - (2 - \sqrt{2}) y^{2} \right) 
			- 2 \tilde{A} c_{0} x
			\\
			&= - \frac{1}{2 \tilde{A} x} 
			\left(  -x^{2} + 2 \tilde{A} x - 
			( 2 - \sqrt{2}) y^{2} + 4 \tilde{A}^{2} c_{0} x^{2}
			\right)
			\\
			&= - \frac{1}{2 \tilde{A} x} \left( ( 4 \tilde{A}^{2} c_{0} - 1 )x^{2}  
			+ 2 \tilde{A} x - ( 2 - \sqrt{2}) y^{2}\right). 
		\end{aligned}
	\end{equation}
Now substituting $x=\tilde{A}(\cos \theta + 1)$, 
$y = \tilde{A} \sin \theta \cos \varphi$ and 
$z = \tilde{A} \sin \theta \sin \varphi$, 
and using that  $\sin^{2} \theta = 1 - \cos^{2} \theta$, $\cos^{2}\varphi = (1+\cos 2 \varphi ) / 2$, we have 
	\begin{equation}\label{eq:psi0-3}
		\begin{aligned}
			\psi_{0} 
			&= - \frac{1}{2 (1 + \cos \theta)} 
			\left\{ ( 4 \tilde{A}^{2} c_{0} - 1) (1+\cos \theta)^{2} 
			+ 2  ( 1 + \cos \theta) - (2 - \sqrt{2}) \sin^{2}\theta 
			\cos^{2} \varphi \right\}
			\\
			& = - \frac{1}{2} 
			\left\{ ( 4 \tilde{A}^{2} c_{0} - 1) ( 1 + \cos \theta) 
			+ 2  - ( 2 - \sqrt{2}) ( 1 - \cos \theta ) \frac{1+\cos 2 \varphi}{2} 
			\right\}
			\\
			& = - \frac{1}{2} 
			\left\{ 4 \tilde{A}^{2} c_{0} + \frac{\sqrt{2}}{2} + 
			\left( 4 \tilde{A}^{2} c_{0} - \frac{\sqrt{2}}{2} \right) \cos \theta 
			- \frac{2 - \sqrt{2}}{2} ( 1 - \cos \theta ) \cos 2 \varphi
			\right\}. 
	 	\end{aligned}
	\end{equation}
Integrating this equality over $S^2_{\tilde A}$ and 
noting the area element in the above coordinate is given by 
$d \sigma = \tilde{A}^2 \sin \theta$, 
Lemma \ref{l:int0S2} yields 
	\begin{equation}\label{eq:c}
		c_{0} = - \frac{\sqrt{2}}{8 \tilde{A}^{2}}.
	\end{equation}
In particular, we may observe that $c_0$ is independent of choices of subsequence $(\eta_k)$. 
Hence, as $ \eta \to 0$,  we obtain 
	\[
		\eta^{-4} ( \xi_\eta' \eta - 2 \xi_\eta ) \to - \frac{\sqrt{2}}{8 \tilde{A}^2}, 
		\quad 
		\psi_\eta \to \psi_0 \quad {\rm in}\ C^\infty_{\rm loc} (\R^2).
	\]
Now substituting \eqref{eq:c} into \eqref{eq:psi0-2} and \eqref{eq:psi0-3}, 
we get 
	\[
		\begin{aligned}
		\psi_{0} &= 
		- \frac{1}{ \bvph^2 + (\sqrt{2} + 1 )^2 \bth^2 + 1/ ( 4 \tilde{A}^2 )    }
		\left\{ \bvph^2 + (\sqrt{2} + 1) \bth^2 - \frac{\sqrt{2}}{8 \tilde{A}^2} \right\}
		\\
		&= - \frac{1}{2 \tilde{A} x} 
		\left\{ z^2  + y^2 - (2 - \sqrt{2}) y^2  \right\} 
		+ \frac{\sqrt{2}}{4 \tilde{A}} x
		\\
		&=\frac{\sqrt{2}}{2} \cos \theta 
		+ \frac{2 - \sqrt{2}}{4} ( 1 - \cos \theta) \cos 2 \varphi 
		\\
		&= \frac{1}{2} ( \cos \theta - 1 ) + \frac{2 - \sqrt{2}}{2} 
			( 1 - \cos \theta ) \cos^{2} \varphi + \frac{\sqrt{2}}{4} ( 1 + \cos \theta ).
		\end{aligned}
	\]
This completes the proof. 
	\end{proof}

\section{Asymptotics of Willmore energy on degenerating tori}\label{s:comp}

In this section we consider inverted tori embedded in manifolds, which degenerate to a sphere 
joint to a small handle. We 
estimate then the derivative of the Willmore energy 
with respect to the variation of  the M\"obius parameter. We first recall some basic facts, and separate the 
handle contribution to the derivative from the spherical one. We then compute the leading 
order term arising from the curvature of the ambient metric,  postponing 
some explicit computations to an appendix.

\subsection{Basic material and handle decomposition}

The goal of this section is to estimate the derivative of the Willmore energy on degenerating tori 
with respect to the M\"obius parameter $\o$ for $|\o |$ close to $1$, namely to prove  
 Proposition \ref{p:varvar2} below.

\noindent Let us recall the following result from \cite{IMM1}, which regards the asymptotics of 
Willmore energy for degenerating tori of small area. In the degenerate limit, apart from the 
handle contribution, one recovers up to high order the energy of a small geodesic sphere 
(see \cite{Mon1}).

\begin{pro}\label{p:expdegtorus} (\cite{IMM1}, Proposition 4.6)
There exists $C_0>0$, which is independent of $\e $, such that 
$$
	\limsup_{r\uparrow 1} \; \sup_{P \in M, R\in SO(3), |\o|=r} \;
	\left| \frac{1}{\e^{2}} \left(  
	W_{g_\e}(\Sigma_{\e,P,R,\omega}) - 8\pi^{2} + \frac{8 \sqrt{2}}{3} \pi^{2} \e^{2} 
	{\rm Sc}_{P} \right) \right| \leq C_0 \e
$$
for all sufficiently small $\e >0$. 
%
\end{pro}

In the next proposition we state one of the main technical results of the paper;  to this aim recall the notation introduced in   \eqref{eq:defSw2} for the surfaces $\Sigma_{\e,P,Id,\omega}$. 

\begin{pro}\label{p:varvar2} 
Let $\d \in (0,1/2)$, $R = Id$ and $\o = |\omega| \mathbf{e}_x$ 
with $1 - |\omega| = \eta$. 
Then there exist $0 < \eta_{\delta}$, $C_0$ and $C_\delta$ 
such that for every $ \tilde{\eta} \in (0, \eta_{\delta})$, 
one may find $C_{\tilde{\eta}} > 0$ satisfying 
$$
	\left|
	  \frac{\pa }{\pa \o} W_{g_\e}(\Sigma_{\e,P,Id,\omega})  
	  - \eta \e^2 \frac{16}{3} \pi B \tilde{A} ( R_{22} - R_{33} ) 
	  \right| 
	  \leq 
	  \left[ C_0 \delta + C_\delta \left\{o_\eta(1) + \e \right\} \right] \eta \e^2
	  + C_{\tilde{\eta}} \e^4
$$
for all $\e \in (0,1/2]$ and $\eta \in [\tilde{\eta}, \eta_{\delta}]$ 
where $C_0$ is independent of $\delta,\eta,\e $, 
$C_\delta$ depends only on $\delta>0$, 
$o_\eta(1) := | \eta^{-4}  (\xi_\eta' \eta - 2 \xi_\eta) - c_0) | +\eta^{3/2}$, 
$c_0 := -\sqrt{2}/(8\tilde{A}^2)$ (See Remark \ref{r:conv-psi-eta} and \eqref{eq:c}), 
$o_\eta(1) \to 0$ as $\eta \to 0$, 
$\tilde A = \sqrt[4]{2\pi^2} $, $B=(2-\sqrt{2})/4$ and  
$R_{ij}$ are the components of the Ricci tensor $\Ric_P$. 
\end{pro}

\

\noindent The proof of this proposition is quite involved and 
will be worked out in the present and the next   subsection and in  the Appendix. 
After scaling the metric as in \eqref{eq:ge=d+eh}, 
we will apply  formula \eqref{eq:varWvarW} to the case $\S \in {\mathcal{T}}_{\e,K}$, for a surface corresponding 
to a value $\o_0$ of the parameter $\o$ which is very close to $1$ in modulus. 
We shall write 
	\[
		W_{g_\e}'(\Sigma_{\e ,P,R,\o}) 
		:= L H + \frac{1}{2} H^3, \qquad 
		d W_{g_\e}(\Sigma_{\e ,P,R,\o}) [\varphi] 
		:= \int_{\Sigma_{\e,P,R,\o}} W_{g_\e}'(\Sigma_{\e ,P,R,\o}) \varphi d \sigma. 
	\]
Using the notation 
$$
  \o = (1 - \eta) {\bf e}_x; \qquad \quad \eta \simeq 0, 
$$
we will take the function $\var_\eta$ in \eqref{eq:phieta} as normal variation $\var$.

\

\noindent It is again convenient to exploit the conformal invariance of the Euclidean Willmore 
functional $W_{g_0}$ in order to write that 
\[
    d W_{g_\e}[\var] 
    = d W_{g_0}[\var] + \left( d W_{g_\e}[\var] - d W_{g_0}[\var] \right) 
    =  \left( d W_{g_\e}[\var] - d W_{g_0}[\var] \right). 
\]
The right-hand side is easier to deal with because some cancellations will occur, but on the 
other hand we will pick up terms of order $\e^2$ from the curvature of the ambient metric $g_\e $, see \eqref{eq:ex-g-2}. 

As already seen in Lemma \ref{104}, degenerating tori geometrically look like spheres with small handles attached near 
the origin of geodesic normal coordinates. In order to evaluate the above derivative it is convenient to 
localize the normal variation near the handle and away from it. 
%
%
%
%
%
%
%
%
%
%
For a small but fixed $\d >0$ we then choose  a radial cut-off function $\chi_\d$ on the 
degenerate torus such that 

$$
  \chi_\d(x) = \begin{cases}
  1 & \hbox{ for } |x| \leq \d; \\ 
  0 & \hbox{ for } |x| \geq 2 \d, 
  \end{cases}
$$
and write 
\begin{equation}\label{eq:var1var2}
  {\var}_\eta = {\var}_{1,\delta,\eta} + {\var}_{2,\delta,\eta} 
  := \chi_\d {\var}_\eta + (1 - \chi_\d) {\var}_\eta.  
\end{equation}
We then have 
\begin{equation}\label{eq:differencediff}
\left( d W_{g_\e}[\var_\eta] - d W_{g_0}[\var_\eta] \right) 
= \left( d W_{g_\e}[\var_{1,\d,\eta}] - d W_{g_0}[\var_{1,\d, \eta}] \right) + 
   \left( d W_{g_\e}[\var_{2,\d,\eta}] - d W_{g_0}[\var_{2,\d,\eta}] \right). 
\end{equation}
Next we compute the contribution of the handle region to the derivative.

\begin{pro}\label{p:varvar1}
There exists  $C_{0}>0$ such that 
for any $\delta, \e , \eta \in (0,1/2)$ one has
$$
   \left| d W_{g_\e}[{\var}_{1,\d, \eta}] - d W_{g_0}[{\var}_{1,\d, \eta}] \right| 
   \leq C_{0} \, \e^{2} \, \eta \, \delta.
$$
\end{pro}

To prove the above proposition, we first prepare the notation, recalling  Section \ref{s:degtori}. 
Let us denote by $(g_{\e ,P,\eta})_{ij}$ the induced metric on 
$\Psi_{\eta}(\T_{\xi_{\eta}})$ from $g_{\e,P}$ in the coordinate 
$\mathcal{Z}(\tvph, \tth,\eta)$ where 
$\partial_{1} = \partial_{\tvph}$ and 
$\partial_{2} = \partial_{\tth}$. 
Furthermore, we  write 
$d \sigma_{\e,P,\eta}$, 
$(\Gamma_{\e ,P,\eta})_{ij}^k$, $(A_{\e ,P,\eta})_{i}^{j}$, 
$\Delta_{\e,P,\eta}$ and $n_{\e ,P,\eta}$ 
for the area element of $\Psi_{\eta}(\T_{\xi_{\eta}})$, 
the Christoffel symbols, the second fundamental form, 
the Laplace-Beltrami operator and 
unit outer normal, respectively. 
Finally, let us denote by $\Ric_{g_{\e,P}}$ the Ricci tensor 
for the ambient space $(B_{10},g_{\e,P})$. 
For these quantities, we have

	\begin{lem}\label{l:basic-esti} Recalling \eqref{eq:YYY}, 
		there exists $C_{0}>0$ such that 
		for all $P \in M$ and $\e,\eta \in (0,1/2)$, 	
		\begin{enumerate}
		\item[{\rm (i)}] 
			 The area elements satisfy 
			\[
				d \sigma_{0,\eta} = 
				( \sqrt{2} \cos \tvph + 1 ) \frac{\eta^{4}}{|Y|^{4}} 
				d \tvph d \tth, 
				\qquad 
				| d \sigma_{\e,P,\eta} - d \sigma_{0,\eta} | 
				\leq C_{0} \e^{2} \frac{\eta^{8}}{|Y|^{6}} d \tvph 
				d \tth.
			\]
		\item[{\rm (ii)} ]
			$| (g_{\e ,P,\eta})^{ij} | \leq C_{0} |Y|^{4}/ \eta^{4} $ and 
			$ | (g_{\e ,P,\eta})^{ij} - (g_{0,\eta})^{ij} | \leq C_{0}\e^{2} | Y |^{2}$.
		\item[{\rm (iii)}] 
			$| (\Gamma_{\e ,P,\eta})^k_{ij} | \leq C_{0}/ |Y|$ and 
			$| (\Gamma_{\e ,P,\eta})^k_{ij} - (\Gamma_{0 ,\eta})^k_{ij} | 
			\leq C_{0} \e^{2} \eta^{4} / |Y|^{3}$. 
		\item[{\rm (iv)}] 
			$| (A_{\e ,P,\eta})^i_j | \leq C_{0} |Y| / \eta^{2}$ and 
			$| (A_{\e ,P,\eta})^i_j - (A_{0 ,\eta})^i_j | 
			\leq C_{0} \e^{2} \eta^{2} / |Y|$. 
		\item[{\rm (v)}] 
			$ | \Ric_{g_{\e,P}}(x) |_{g_{0}} \leq C_{0} \e^{2}$. 
		\item[{\rm (vi)}] 
			$| n_{\e ,P,\eta} - n_{0,\eta} | \leq C_0 \e^2 \eta^4 / |Y|^2$, 
			$| \partial_i (n_{\e ,P,\eta} - n_{0,\eta})  | \leq C_0 \e^2 \eta^4 / |Y|^3$ 
			and $| \partial_i \partial_j ( n_{\e ,P, \eta} - n_{0,\eta} ) | 
			\leq C_0 \e^2 \eta^4 / |Y|^4$. 
		\end{enumerate}
	\end{lem}

	\begin{proof}
Since $\mathcal{Z}(\tvph, \tth,\eta) 
= \Psi_{\eta} ( Y(\tvph, \tth,\eta))$ and 
$|D_{(\tilde \var , \tth)}^\alpha Y|$ are uniformly bounded with respect to $\e$ and $\eta$ 
for any $\alpha \in \Z_+^2$, we have 
	\begin{equation}\label{eq:deri-Z}
			| \mathcal{Z} | = \frac{\eta^{2}}{|Y|}, \quad 
			| \partial_{i} \mathcal{Z} | = \frac{\eta^{2}}{|Y|^{2}} |\partial_{i} Y| 
			\sim C_{0} \frac{\eta^{2}}{|Y|^{2}}, \quad 
			| \partial_{i} \partial_{j} \mathcal{Z} | 
			\leq C_{0} \left( \frac{\eta^{2}}{|Y|^{3}} 
			+ \frac{\eta^{2}}{|Y|^{2}} \right) \leq C_{0} \frac{\eta^{2}}{|Y|^{3}}. 
	\end{equation}
Furthermore, by the conformality of $\Psi_{\eta}$ and 
$g_0( \partial_i Y , \partial_j Y  ) = |\partial_i Y| |\partial_j Y| \delta_{ij}$, 
one also sees that 
	\begin{equation}\label{eq:Z-orth}
		g_{0}[ \partial_{i} \mathcal{Z}, \partial_{j} \mathcal{Z}] = |\partial_{i} \mathcal{Z}|_{g_{0}} 
		| \partial_{j} \mathcal{Z} |_{g_{0}} \delta_{ij}. 
	\end{equation}
Notice also that $f_{0,i} = \partial_i \mathcal{Z} / | \partial_i \mathcal{Z} |_{g_0}$ ($i=1,2$) 
form an orthonormal basis of $T_{\mathcal{Z}}\Psi_{\eta} (\T_{\xi_\eta})$ and 
$n_{0,\eta}$ is given by  (see Subsection \ref{ss:32} for the definition of $n$) 
	\[
		n_{0,\eta}(\tvph, \tth) = 
		\frac{(D_x\Psi_\eta) ( Y(\tvph, \tth, \eta) ) [n (\tvph, \tth ) ]}
		{ | (D_x\Psi_\eta) ( Y(\tvph, \tth, \eta) ) 
		[n (\tvph, \tth ) ]|_{g_0} }. 
	\]

Since $g_{\e, P,\alpha \beta} (x) 
= \delta_{\alpha \beta} + \e^{2} h^{\e}_{P,\alpha \beta}(x)$ and 
$h^{\e}_{P,\alpha \beta}$ satisfies \eqref{eq:defh} 
uniformly with respect to $\e$ and $P$, 
using \eqref{eq:deri-Z} and \eqref{eq:Z-orth}
the above claims follow from direct computations. 
	\end{proof}

\

\noindent		 We now prove Proposition \ref{p:varvar1}.

\begin{proof}[Proof of Proposition \ref{p:varvar1}]
We first recall \eqref{eq:phi-eta}:
	\[
		\begin{aligned}
			\varphi_{\eta} &= - 
			\frac{\eta}{|Y(\tvph, \tth,\eta)|^{2}} 
			\left\{ 2 \left( \sqrt{2} \cos \tvph + 1 
			- ( \sqrt{2}+1) \cos \tvph \cos \tth \right) 
			+ ( \xi_{\eta}' \eta - 2 \xi_{\eta}) \cos \tvph \cos \tth 
			\right\}
			\\
			&= - \frac{\eta}{|Y(\tvph, \tth,\eta)|^{2}} 
			\left( h(\tvph, \tth) 
			+ (\xi_{\eta}' \eta - 2 \xi_{\eta}) \cos \tvph \cos \tth 
			\right).
		\end{aligned} 
	\]
From Lemma \ref{l:xi-xi'} and a Taylor expansion of $Y$ and $h$ 
(see \eqref{eq:Y-h}, \eqref{eq:RYRh} and \eqref{eq:|Y|}), 
one may find $C_{0}>0$ such that 
	\begin{equation}\label{eq:esti-deri-phi}
		| \varphi_{\eta} | \leq C_{0} \eta, \qquad 
		| \partial_{i} \varphi_{\eta} | \leq C_{0} \frac{\eta}{|Y|}, \qquad 
		| \partial_{i} \partial_{j} \varphi_{\eta} | 
		\leq C_{0} \frac{\eta}{|Y|^{2}}
	\end{equation}
for all $(\tvph, \tth) \in [-\pi,\pi]^{2}$ and 
$\eta \in (0,1/2)$.

		Next, denote by $H_{\e,P,\eta}$ the mean curvature of $\Psi_\eta (\T_{\xi_\eta})$ 
with the ambient metric $g_{\e ,P}$. Recalling \eqref{eq:varWvarW} and noting 
	\begin{equation}\label{eq:diff-W'}
		\begin{aligned}
			& dW_{g_\e}[\varphi_{1,\delta,\eta}] - dW_{g_0}[\varphi_{1,\delta,\eta}] 
			\\
			= &\int_{\Psi_{\eta}(\T_{\xi_{\eta}})} 
			\left\{ - H_{\e,P,\eta} \Delta_{\e ,P,\eta} \varphi_{1,\delta,\eta} 
			- H_{\e,P,\eta} \left( |A_{\e ,P,\eta}|^{2} + 
			\Ric_{g_{\e,P}} ( n_{\e,P,\eta}, n_{\e,P,\eta} ) 
			- \frac{1}{2} H_{\e,P,\eta}^{2} \right) \varphi_{1,\delta,\eta} 
			\right\} 
			\\
			& \qquad \times 
			( d \sigma_{\e,P,\eta} - d \sigma_{0,\eta} ) 
			\\
			& - \int_{\Psi_{\eta} (\T_{\xi_{\eta}}) } 
			\left( H_{\e,P,\eta} \Delta_{\e ,P,\eta} \varphi_{1,\delta,\eta} 
			- H_{0,\eta} \Delta_{0,\eta} \varphi_{1,\delta,\eta} 
			\right) d \sigma_{0,\eta}  
			\\
			&- \int_{\Psi_{\eta}(\T_{\xi_{\eta}})} 
			H_{\e ,P,\eta}\Ric_{g_{\e,P}} ( n_{\e,P,\eta}, n_{\e,P,\eta} ) \varphi_{1,\delta,\eta} 
			d \sigma_{0,\eta} 
			\\
			& - \int_{\Psi_{\eta}(\T_{\xi_{\eta}})} 
			\left\{ H_{\e,P,\eta} |A_{\e ,P,\eta}|^{2} 
			- H_{0,\eta} |A_{0,\eta}|^{2} \right\}
			\varphi_{1,\delta,\eta} d \sigma_{0,\eta} 
			+ \frac{1}{2} \int_{\Psi_{\eta}(\T_{\xi_{\eta}})} 
			\left( H_{\e,P,\eta}^{3} - H_{0,\eta}^{3}  \right) 
			\varphi_{1,\delta,\eta} d \sigma_{0,\eta},
 		\end{aligned}
	\end{equation}
we estimate each term in the above using Lemma \ref{l:basic-esti}.
For this purpose, we first remark that 
	\[
		|\mathcal{Z}| \leq 2 \delta \qquad \Leftrightarrow \qquad 
		|Y| \geq \frac{\eta^{2}}{2\delta}. 
	\]
Moreover, as in the proof of Lemma \ref{l:int0S2} 
(see \eqref{eq:|Y|}), by $\xi_\eta = O(\eta^2)$, 
we may find a $C_{1}>0$, 
which is independent of $\delta$ and $\eta$, such that 
	\[
		I_{\delta}:= \left\{ (\tvph, \tth) \in [-\pi,\pi]^{2} \; : \;
		\left( \tvph^{2} + ( \sqrt{2} + 1)^{2} \tth^{2} 
		\right) \geq C_{1} \frac{\eta^{4}}{\delta^{2}}
		\right\}  \supset 
		\left\{ (\tvph, \tth) \in [-\pi,\pi]^{2} \; : \;
		|Y(\tvph, \tth,\eta)| \geq \frac{\eta^{2}}{2\delta}  \right\}
	\]
for all $\delta , \eta \in (0,1/2)$.

		First, we estimate the last two terms in \eqref{eq:diff-W'}. 
Since $H_{\e,P,\eta} = (A_{\e ,P,\eta})^{i}_{i}$ and 
$|A_{\e ,P,\eta}|^{2} = (A_{\e,P,\eta})^{i}_{j}  (A_{\e ,P,\eta})^{j}_{i}$, 
by Lemma \ref{l:basic-esti}, it is easily seen that 
	\[
		\left| H_{\e,P,\eta} |A_{\e ,P,\eta}|^{2} 
		- H_{0,\eta} |A_{0,\eta}|^{2} \right| 
		+ \left| H_{\e,P,\eta}^{3} - H_{0,\eta}^{3}  \right| 
		\leq C_{0} \e^{2} \frac{|Y|}{\eta^{2}}. 
	\]
Hence, from \eqref{eq:esti-deri-phi}, \eqref{eq:|Y|}, 
${\rm supp}\, \varphi_{1,\delta,\eta} \subset \overline{B_{2\delta}(0)}$ 
and a change of variables, it follows that 
	\begin{equation}\label{eq:AandH}
		\begin{aligned}
			&\int_{\Psi_{\eta}(\T_{\xi_{\eta}})} 
			\left| H_{\e,P,\eta} \left|A_{\e ,P,\eta}\right|^{2} 
			- H_{0,\eta} \left|A_{0,\eta}\right|^{2} \right|
			|\varphi_{1,\delta,\eta} | d \sigma_{0,\eta} 
			+ \frac{1}{2} \int_{\Psi_{\eta}(\T_{\xi_{\eta}})} 
			\left| H_{\e,P,\eta}^{3} - H_{0,\eta}^{3} \right|
			| \varphi_{1,\delta,\eta} | d \sigma_{0,\eta}
			\\
			\leq & C_{0} \e^{2} \int_{I_{\delta}} 
			\frac{\eta^{3}}{|Y|^{3}} d \tilde \var d \tth 
			\leq C_{0} \e^{2} \int_{I_{\delta}} 
			\frac{\eta^{3}}{(\tvph^{2} + (\sqrt{2}+1)^{2} 
			\tth^{2})^{3/2}} 
			d \tvph d \tth 
			\leq C_{0} \e^{2} \eta^{3} \int_{C_{1} \eta^{2} / \delta}^{10} 
			r^{-2} d r 
			\leq C_{2} \e^{2} \eta \delta, 
		\end{aligned}
	\end{equation}
where $C_{2}$ is independent of $\e$, $\eta$ and $\delta$. 
Similarly, for the Ricci tensor, we have 
	\begin{equation}\label{eq:Ric}
		\begin{aligned}
			\int_{\Psi_{\eta}(\T_{\xi_{\eta}})} 
			| H_{\e ,P,\eta} | 
			| \Ric_{g_{\e,P}}(n_{\e,P,\eta},n_{\e,P,\eta}) | 
			| \varphi_{1,\delta,\eta} | d \sigma_{0,\eta} 
			\leq C_{0} \e^{2} \int_{I_{\delta}} \frac{\eta^{3}}{|Y|^{3}} 
			d \tvph d \tth 
			\leq C_{2} \e^{2} \eta \delta.
		\end{aligned}
	\end{equation}

		In order to deal with the first two terms in \eqref{eq:diff-W'}, 
we estimate 
	\[
		\Delta_{\e ,P,\eta} \varphi_{1,\delta,\eta} 
		\qquad 
		{\rm and} \qquad 
		\left( \Delta_{\e ,P,\eta} - \Delta_{0,\eta} 
		\right) \varphi_{1,\delta,\eta}. 
	\]
First, by \eqref{eq:deri-Z} and the definition of $\chi_{\delta}$, 
there holds 
	\[
		\left| \partial_{i} \left( \chi_{\delta} \left( \mathcal{Z}(\tvph, \tth,\eta) 
		\right) \right) \right| 
		\leq C_{0} \frac{1}{\delta} \frac{\eta^{2}}{|Y|^{2}}, 
		\qquad 
		\left| \partial_{i} \partial_{j} \left( \chi_{\delta} \left( 
		\mathcal{Z}(\tvph, \tth ,\eta) \right) \right) \right| 
		\leq C_{0} 
		\left( \frac{1}{\delta^{2}} \frac{\eta^{4}}{|Y|^{4}} 
		+ \frac{1}{\delta} \frac{\eta^{2}}{|Y|^{3}} \right).
	\]
Write ${\rm Hess}_{\e ,P,\eta}$ for the Hessian 
of $(\Psi_\eta (\T_{\xi_\eta}), g_{\e ,P,\eta} )$. 
From \eqref{eq:esti-deri-phi} and Lemma \ref{l:basic-esti}, it follows that 
	\[
		\begin{aligned}
			\left| \left( {\rm Hess}_{\e ,P,\eta} ( \varphi_{1,\delta,\eta} ) \right)_{ij} \right|  
			&= \left| \partial_{i} \partial_{j} (\varphi_{1,\delta,\eta} ) 
			- (\Gamma_{\e ,P,\eta})^{k}_{ij} \partial_{k} \varphi_{1,\delta,\eta} 
			 \right|
			 \\
			 & \leq 
			 C_{0} \left\{ \left( \frac{1}{\delta^{2}} \frac{\eta^{4}}{|Y|^{4}} 
			 + \frac{1}{\delta} \frac{\eta^{2}}{|Y|^{3}} \right) \eta
			 + \frac{1}{\delta} \frac{\eta^{2}}{|Y|^{2}} \frac{\eta}{|Y|} 
			 + \frac{\eta}{|Y|^{2}} 
			 \right\}  
			 + \frac{C_{0}}{|Y|} \left( \frac{1}{\delta} \frac{\eta^{2}}{|Y|^{2}} \eta 
			 + \frac{\eta}{|Y|} \right) 
			 \\
			 & \leq C_{0} \eta \left( \frac{1}{\delta^{2}} \frac{\eta^{4}}{|Y|^{4}} 
			 + \frac{1}{\delta} \frac{\eta^{2}}{|Y|^{3}} 
			 + \frac{1}{|Y|^{2}} \right) 
		\end{aligned}
	\]
and 
	\[
		\begin{aligned}
			\left| \left( {\rm Hess}_{\e ,P,\eta} - {\rm Hess}_{0,\eta} 
			\right) \varphi_{1,\delta,\eta} \right|  
			&= 
			\left| (\Gamma_{\e ,P,\eta})^{k}_{ij} - (\Gamma_{0,\eta})^{k}_{ij} 
			\right| | \partial_{k} \varphi_{1,\delta,\eta} | 
			\leq C_{0} \e^{2} \frac{\eta^{4}}{|Y|^{3}}  
			\left( \frac{\eta^{3}}{\delta |Y|^{2}} + \frac{\eta}{|Y|} \right)
			\\
			&= C_{0} \e^{2} \eta 
			\left( \frac{\eta^{6}}{ \delta |Y|^{5}} + \frac{\eta^{4}}{|Y|^{4}} \right). 
		\end{aligned}
	\]
Recalling $\Delta_{\e ,P,\eta} f = (g_{\e ,P,\eta})^{ij} ({\rm Hess}_{\e ,P,\eta} f)_{ij} $, 
we get 
	\[
		\left| \Delta_{\e ,P,\eta} \varphi_{1,\delta,\eta} \right| 
		\leq C_{0} \eta \left( \frac{1}{\delta^{2}} + \frac{|Y|}{\delta \eta^{2}} 
		+ \frac{|Y|^{2}}{\eta^{4}} \right)
	\]
and 
	\[
		\begin{aligned}
			\left| \left( \Delta_{\e,P,\eta} - \Delta_{0,\eta} 
			\right) \varphi_{1,\delta,\eta} \right|  
			& \leq \left| (g_{\e ,P,\eta})^{ij} - (g_{0,\eta})^{ij} \right| 
			\left| {\rm Hess}_{\e ,P,\eta} \varphi_{1,\delta,\eta} \right| 
			+ \left| (g_{\e ,P,\eta})^{ij} \right| 
			\left| \left( {\rm Hess}_{\e ,P,\eta} - {\rm Hess}_{0,\eta} 
			\right) \varphi_{1,\delta,\eta} \right|  
			\\
			& \leq C_{0} \e^{2} \eta 
			\left( \frac{1}{\delta^{2}} \frac{\eta^{4}}{|Y|^{2}} 
			+ \frac{1}{\delta} \frac{\eta^{2}}{|Y|} + 1 \right)
			+ C_{0} \e^{2} \eta \left( \frac{\eta^{2}}{\delta |Y|} + 1 \right) 
			\\
			& \leq C_{0} \e^{2} \eta 
			\left( \frac{1}{\delta^{2}} \frac{\eta^{4}}{|Y|^{2}} 
			+ \frac{1}{\delta} \frac{\eta^{2}}{|Y|} + 1  \right).
		\end{aligned}
	\]
From these estimates and Lemma \ref{l:basic-esti}, one may observe that 
	\begin{equation}\label{eq:diff-area}
		\begin{aligned}
			&\int_{\Psi_{\eta}(\T_{\xi_{\eta}})} 
			\left| H_{\e,P,\eta} \Delta_{\e ,P,\eta} \varphi_{1,\delta,\eta} 
			+ H_{\e,P,\eta} \left( |A_{\e ,P,\eta}|^{2} + 
			\Ric_{g_{\e,P}} ( n_{\e,P,\eta}, n_{\e,P,\eta} ) 
			- \frac{1}{2} H_{\e,P,\eta}^{2} \right) \varphi_{1,\delta,\eta} 
			\right| 
			\\
			& \qquad \times 
			\left| d \sigma_{\e,P,\eta} - d \sigma_{0,\eta} \right|
			\\
			\leq & C_{0} \e^{2} \eta \int_{I_{\delta}}\left\{ 
			\frac{|Y|}{\eta^{2}}
			\left( \frac{1}{\delta^{2}} + \frac{|Y|}{\delta \eta^{2}} 
			+ \frac{|Y|^{2}}{\eta^{4}} \right)  
			+ \frac{|Y|^{3}}{\eta^{6}} + \e^2 \frac{|Y|}{\eta^2}  \right\} \frac{\eta^{8}}{|Y|^{6}} 
			d \tvph d \tth 
			\\
			\leq & C_{0} \e^{2} \eta \int_{C_{1}\eta^{2}/\delta}^{10} 
			\left\{ \frac{\eta^{6}}{\delta^{2} r^{4}} 
			+ \frac{\eta^{4}}{\delta r^{3}} + \frac{\eta^{2}}{r^{2}} 
		    + \e^2 \frac{\eta^{6}}{r^{4}} \right\} dr 
			\leq C_{2} \e^{2} \eta \delta
		\end{aligned}
	\end{equation}
and 
	\begin{equation}\label{eq:Laplacian}
		\begin{aligned}
			 &\int_{\Psi_{\eta} (\T_{\xi_{\eta}}) } 
			\left| H_{\e,P,\eta} \Delta_{\e ,P,\eta} \varphi_{1,\delta,\eta} 
			- H_{0,\eta} \Delta_{0,\eta} \varphi_{1,\delta,\eta} \right|
			d \sigma_{0,\eta}  
			\\
			\leq & \int_{\Psi_{\eta}(\T_{\xi_{\eta}})} 
			\left\{ \left| H_{\e,P,\eta} - H_{0,\eta} \right| 
			\left| \Delta_{\e ,P,\eta} \varphi_{1,\delta,\eta} \right| 
			+ | H_{0,\eta} | \left| \left( \Delta_{\e ,P,\eta} - 
			\Delta_{0,\eta} \right) \varphi_{1,\delta,\eta} \right|
			\right\} d \sigma_{0,\eta}
			\\
			\leq & C_{0} \e^{2} \eta \int_{I_{\delta}} 
			\left\{ \frac{\eta^{2}}{\delta^{2}|Y|} + \frac{1}{\delta} 
			+ \frac{|Y|}{\eta^2} 
			\right\}\frac{\eta^{4}}{|Y|^{4}} 
			d \tvph d \tth \leq C_{2} \e^{2} \eta \delta.
		\end{aligned}
	\end{equation}
The conclusion of Proposition easily follows from 
\eqref{eq:diff-W'}, \eqref{eq:AandH}, \eqref{eq:Ric}, 
\eqref{eq:diff-area} and \eqref{eq:Laplacian}.
\end{proof}

\subsection{Metric dependence}

The  goal of this subsection is to estimate the contribution from $\var_{2,\delta,\eta}$ (see \eqref{eq:var1var2}) to the derivative of the 
Willmore energy. $\var_{2,\delta,\eta}$ is supported in a region of the degenerating torus where the curvature stays bounded.  
The main contribution of $\var_{2,\delta,\eta}$ will be due to the curvature of $M$ and to the deviation of the 
tori from a purely spherical shape. Our aim is to prove the following result, which quantifies both  effects.

\begin{pro}\label{p:new}
Let the limit sphere $S^2_{\tilde{A}}$ and  $\psi_0$ be as in Proposition \ref{p:psi0}. 
For $\delta \in (0,1/2]$ and $\varphi_{\eta, \delta, 2}$ as in \eqref{eq:var1var2}, 
there exist $C_0>0$, $C_\delta >0$ and $\eta_\delta>0$ such that 
	\[
		\begin{aligned}
		&\, 
		\left| dW_{g_\e}\left[ \frac{\var_{\eta,\delta,2}}{\eta} \right] - dW_{g_0} 
		\left[ \frac{\varphi_{\eta,\delta,2}}{\eta} \right] 
		+ \e^2 \left[  \int_{\SA} (1-\chi_\delta) 
		\left( F\Delta_{\SA} \psi_0 + \Ric_{P}(n_{0},n_{0}) 
		H_{\SA} \psi_0 \right) d \sigma_0 \right] \right|
		\\
		\leq &\, C_0 \delta \e^2 
		+ C_\delta \left( o_\eta(1) \e^2 + \e^3 \right)
		\end{aligned}
	\]
holds for any $\eta \in (0,\eta_\delta]$ and $\e \in (0,1/2]$ 
where $C_0$ is independent of $\delta$, $C_\delta$ depends only on $\delta$,   
 $o_\eta (1)$ is as in Proposition \ref{p:varvar2}, and 
$F$ is given by 
\[
\begin{aligned}
F &:= -\sum_{i=1}^{2} e_{i}(h_{ni}) + \sum_{i,j=1}^{2} 
h_{nj} \langle \nabla^{\R^{3}}_{e_{i}} e_{i}, e_{j} \rangle 
- \frac{1}{2} h_{nn} H_{\SA} 
+ \frac{1}{2} \sum_{i=1}^{2} \frac{\partial h}{\partial n_0} ( e_{i}, e_{i} ), 
\\
H_{\SA} &:= \text{the mean curvature of $\SA$ in $(\R^3,g_0)$,}
\\
\mathcal{X}(\theta, \varphi) 
&:= \tilde{A} \left( \cos \theta + 1, \sin \theta \cos \varphi , \sin \theta \sin \varphi \right)
\qquad (\theta, \varphi) \in (0,\pi) \times [0,2\pi], 
\\
e_1 &:= \tilde{A}^{-1} \partial_\theta \mathcal{X} 
= (-\sin \theta , \cos \theta \cos \varphi, \cos \theta \sin \varphi ),
\\
e_2 &:= (\tilde A \sin \theta)^{-1} \partial_\varphi \mathcal{X} 
= (0, - \sin \varphi, \cos \varphi),
\\
n_0 &:= (\cos \theta, \sin \theta \cos \varphi, \sin \theta \sin \varphi),
\\
(h(x))_{\a \b} &:= \frac{1}{3} R_{\a \mu \nu \b} x^\mu x^\nu, 
\quad 
h_{ni} := h(x)(n_0,e_i), \quad h_{nn} := h(x) (n_0,n_0). 
\end{aligned}
\]
\end{pro}

%
%
%
%
%

\

\begin{rem}
The term $F$  above will turn out to be the metric derivative of  the
mean curvature of $(\SA, g_t)$ at $t = 0$ where 
$g_{t , \a \b} (x) := \delta_{\a \b} + t h_{\a \b}(x)$. 
Hence, $F$ is smooth on $\SA$. 
 	\end{rem}

\

%
%
%
%
%
%
%
%
%
%
%
%
%

\noindent Before proving Proposition \ref{p:new} we collect some useful preliminary material and lemmas. 
Recalling the expansion of the metric $g$ in the normal coordinates and 
setting $t = \e^2$, we observe that 
	\[
		g_{t,P,\a \b} (x) := g_{\e ,P,\a \b} (x)  
		= \delta_{\alpha \beta} + t h_{P,\alpha \beta} (t, x)
	\]
and $t \mapsto g_{t,P,\a \b}(x) : [0,t_0] \to C^k(B_{10})$ 
is of class $C^{1,1/2}$ for each $k \in \N$. Moreover, 
	\begin{equation}\label{eq:defh-2}
		\frac{\partial}{\partial t} g_{t,P,\a \b} (x) \big|_{t=0} 
		= h_{P,\alpha \beta} (0,x) = \frac{1}{3} R_{\alpha \mu \nu \beta} x^\mu x^\nu .
	\end{equation}

\noindent 
Next, we denote by $\Delta_{g_{t,P},\eta}$, $A_{g_{t,P},\eta}$, 
$\mathring{A}_{g_{t,P} ,\eta}$, 
$H_{g_{t,P},\eta}$ and $n_{g_{t,P},\eta}$ the Laplace-Beltrami operator, 
the second fundamental form, its traceless part, 
the mean curvature and the unit outer normal of $(\Psi_\eta (\T), g_{t,P} )$. 
We also write $\Ric_{g_{t,P}}$ and $dW(t,P,\eta)$ 
for the Ricci tensor of $(B_{10},g_{t,P})$ and 
the derivative of the Willmore functional at $(\Psi_\eta (\T), g_{t,P} )$.

	\begin{lem}\label{l:dot-W'}
		For each  $\delta \in (0,1/2)$, one may find 
		$\eta_\delta > 0$ and $C_\delta$ so that 
		if $0 < \eta \leq \eta_\delta$ and $0 < t \leq 1/2$, then 
			\begin{equation}\label{eq:dot-W'}
				\begin{aligned}
					&
					\left| dW (t,P,\eta)[\psi_{2,\delta,\eta}] - dW(0,\eta)[\psi_{2,\delta,\eta}]
					+ t \tilde{W}_P[\psi_{2,\delta,0}] 
					\right|
					\leq C_\delta ( o_\eta (1) t + t^{3/2}), 
				\end{aligned}
		\end{equation}
where $C_\delta$ depends only on $\delta$, $o_\eta(1)$ is as in Proposition \ref{p:varvar2} and 
	\[
		\begin{aligned}
			\psi_{2,\delta,\eta} &:= \frac{\varphi_{2,\delta,\eta}}{\eta} 
			= \frac{(1-\chi_{\delta}) \varphi_{\eta}}{\eta},
			\quad \psi_{2,\delta,0} := 
			(1 - \chi_{\delta}) \psi_0,
			\\
			\tilde{W}_P[\psi] 
			&:= \int_{\SA} \left\{  
			\left( \Delta_{\SA} \frac{d H_{g_{t,P},0}}{d t} \Big|_{t=0}  \right) \psi 
			+ \Ric_{P}(n_{0,0}, n_{0,0}) H_{\SA}  \psi \right\} d \sigma_{g_0}. 
		\end{aligned}
	\]
	\end{lem}

	\begin{proof} 
We first fix a $\delta \in (0,1/2)$. 
Recall from Proposition \ref{p:1-2-var} that 
	\begin{equation}\label{eq:wdot}
		dW(t,P,\eta) [\psi] = 
		- \int_{\Psi_\eta(\T_{\xi_\eta})} 
		\left\{ \Delta_{g_{t,P},\eta} H_{g_{t,P},\eta} 
		+ \left(|\mathring{A}_{g_{t,P},\eta}|^{2} 
		+ \Ric_{g_{t,P}} (n_{g_{t,P},\eta},n_{g_{t,P},\eta}) 
		\right)H_{g_{t,P},\eta}
		\right\} \psi d \sigma_{g_{t,P},\eta}. 
	\end{equation}
Since $|\mathcal{Z}| \geq \delta$ is equivalent to $|Y| \leq \eta^2 / \delta$, 
from the parameterization of 
$\mathcal{Z}(\eta^2 \bvph, \eta^2 \bth,\eta)$ 
for $\Psi_\eta(\T_{\xi_\eta})$ and \eqref{eq:|Y|}, it is easily seen that 
there exist $C_1>0$, which is independent of $\delta$ and $\eta$, 
such that if $ 0< \eta \leq 1/\delta$, then  
	\[
		\Psi_\eta(\T_{\xi_\eta}) \cap (B_{\delta} (0) )^c 
		\subset \left\{ \mathcal{Z}(\eta^2 \bvph, \eta^2 \bth, \eta) 
		\ |\ (\bvph, \bth ) \in I_\delta
		\right\} \qquad 
		\text{where } I_\delta := 
		\left[ - \frac{C_1}{\delta}, \frac{C_1}{\delta}  \right]^2. 
	\]
We apply Lemma \ref{104} (ii) for $R = C_1/\delta$. 
Then one may find a $\eta_\delta >0$ such that 
for every $k \in \N$, there exists a $C_{k}>0$ satisfying
	\begin{equation}\label{eq:covZ}
		\left\| \mathcal{Z}( \eta^2 \cdot, \eta^2 \cdot ) - 
		\Refx \circ Z_0 \right\|_{C^k([-R,R]^2)} 
		\leq C_{k} \eta^{3/2}
	\end{equation}
provided $\eta \in (0, \eta_\delta]$.

		Now, due to the cut-off function $\chi_\delta$, 
it is sufficient to consider the quantities on $I_\delta$. 
We also suppose $ 0 < \eta \leq \eta_\delta$. 
Since $t \mapsto g_{t,P,\a \b}(x)$ is of class $C^{1,1/2}$ and 
the convergence \eqref{eq:covZ} holds, we observe that 
	\[
		\begin{aligned}
			\Delta_{g_{t,P},\eta} f 
			&= \Delta_{g_0,\eta} f + \frac{d}{dt} \Delta_{g_{t,P},\eta} f \big|_{t=0}  t
			+ O_{\delta,1}(t^{3/2} \| f \|_{C^2(I_\delta)}), 
			&
			H_{g_{t,P},\eta} &= H_{g_0,\eta} + \frac{d}{d t} H_{g_{t,P},\eta} \big|_{t=0} t + 
			O_{\delta,2}(t^{3/2}), 
			\\
			|\mathring{A}_{g_{t,P},\eta}|^2 
			&= |\mathring{A}_{g_0,\eta}|^2 + \frac{d}{d t} 
			|\mathring{A}_{g_{t,P},\eta}|^2 \big|_{t=0} t + O_{\delta,2}(t^{3/2}),
			&
			\Ric_t &= \Ric_{g_0} +  \frac{d}{d t} \Ric_{g_{t,P}} \big|_{t=0} t 
			+ O_{\delta,2}(t^{3/2}), 
			\\
			d \sigma_{g_{t,P},\eta} &= d \sigma_{g_0,\eta} 
			+ \frac{d}{d t} d \sigma_{g_{t,P},\eta} \big|_{t=0} t 
			+ O_{\delta,2}(t^{3/2}),
			&
			n_{g_{t,P},\eta} &= n_{g_0,\eta} + \frac{d}{dt} n_{g_{t,P},\eta} \big|_{t=0} t 
			+ O_{\delta,2}(t^{3/2})
		\end{aligned}
	\]
where 
$|O_{\delta,1}(t^{3/2} \|f \|_{C^2} )| \leq C_{1,\delta} \| f \|_{C^2} t^{3/2}$ and 
$\|O_{\delta,2}(t^{3/2})\|_{C^2(I_\delta)}  \leq C_{2,\delta} t^{3/2}$, and 
$C_{\delta,i}$ depend only on $\delta$. 
Substituting these formula into \eqref{eq:wdot} and 
noting $\Ric_0 = 0$, we obtain 
	\begin{equation}\label{eq:W-dot-2}
		\begin{aligned}
			& dW(t,P,\eta) [\psi_{2,\delta,\eta}]
			\\
			= & dW(0,\eta) [\psi_{2,\delta,\eta}]
			- t \int_{ \Psi_\eta ( \T_{\xi_\eta} ) } 
			\left\{ 
			\Delta_{0,\eta} \frac{d}{d t} H_{g_{t,P},\eta} \big|_{t=0} 
			+ \frac{d}{d t} \Delta_{g_{t,P},\eta}  H_{g_0,\eta} \big|_{t=0} 
			\right\} \psi_{2,\delta,\eta} d \sigma_{g_0,\eta}
			\\
			&- t\int_{\Psi_\eta (\T_{\xi_\eta})} 
			\left\{ 
			\frac{d}{d t} |\mathring{A}_{g_{t,P},\eta}|^2 \big|_{t=0}  H_{g_0,\eta}
			+ | \mathring{A}_{g_0,\eta} |^2 \frac{d}{dt} H_{g_{t,P},\eta} \big|_{t=0} + 
			\frac{d}{d t} \Ric_{g_{t,P}} \big|_{t=0} (n_{g_0,\eta},n_{g_0,\eta}) 
			H_{g_0,\eta}
			\right\} \psi_{2,\delta,\eta} d \sigma_{g_0,\eta}
			\\
			& - t \int_{ \Psi_\eta (\T_{\xi_\eta}) } 
			\left\{ \Delta_{g_0,\eta} H_{g_0,\eta} 
			+ |\mathring{A}_{g_0,\eta}|^2 H_{g_0,\eta}
			\right\} \psi_{2,\delta,\eta}\frac{d}{d t} d\sigma_{g_{t,P},\eta} \big|_{t=0} 
			+ O_\delta(t^{3/2})
		\end{aligned}
	\end{equation}
where $|O_\delta(t^{3/2})| \leq C_{\delta,3} t^{3/2}$.

		Next, we observe the behaviours of the above quantities as $\eta \to 0$. 
By \eqref{eq:covZ} and the fact that $\Refx \circ Z_0$ is a position vector of $\SA$, 
it follows from $H_{g_0,0} = 2 / \tilde{A}$ and $(A_{g_0,0})^i_j = \delta^i_j / \tilde{A}$ 
that 
	\[
		\left\| H_{g_0,\eta} - \frac{2}{\tilde A} \right\|_{C^2(I_\delta)} 
		+ \left\| (A_{g_0,\eta})^i_j - \frac{\delta^i_j}{\tilde A} \right\|_{C^2(I_\delta)} 
		+ \left\| \mathring{A}_{g_0,\eta} \right\|_{C^0(I_\delta)}
		\leq C_{\delta,4} \eta^{3/2}. 
	\]
Hence, 
	\[
		\begin{aligned}
			& \left\| \frac{d}{d t} \Delta_{g_{t,P},\eta}  H_{g_0,\eta} \big|_{t=0} 
			 \right\|_{C^0(I_\delta)} 
			+ \left\| \frac{d}{d t} |\mathring{A}_{g_{t,P},\eta}|^2 \big|_{t=0} 
			 \right\|_{C^0(I_\delta)}
			 \\
			 & \qquad \qquad 
			+ \left\| \Delta_{g_0,\eta} H_{g_0,\eta} \right\|_{C^0(I_\delta)} 
			+ \left\| |\mathring{A}_{g_0,\eta}|^2 H_{g_0,\eta} \right\|_{C^0(I_\delta)} 
			\leq C_{\delta,4} \eta^{3/2}.
		\end{aligned}
	\]
Recalling \eqref{eq:psi-eta}, we also observe that 
	\[
		\begin{aligned}
		&\left\| \psi_{2,\delta,\eta} - (1 - \chi_\delta) \psi_0 \right\|_{C^0(I_\delta)} 
		+ \left\| \Delta_{g_0,\eta}\frac{d}{d t} H_{g_{t,P},\eta} \big|_{t=0} 
		- \Delta_{S^2_{\tilde A}} \frac{d}{dt} H_{g_{t,P},0} \big|_{t=0} 
		 \right\|_{C^0(I_\delta)} 
		\\
		& \qquad \qquad
		+ \left\| \frac{d}{dt} \Ric_{g_{t,P}} |_{t=0}(n_{g_0,\eta},n_{g_0,\eta}) 
		- \frac{d}{dt} \Ric_{g_{t,P}} |_{t=0} (n_{g_0,0},n_{g_0,0}) \right\|_{C^0(I_\delta)} 
		\leq C_{\delta,4} o_\eta(1) 
		\end{aligned}
	\]
where $o_\eta (1) = |\eta^{-4}  (\xi_\eta' \eta - 2 \xi_\eta) - c_0| + \eta^{3/2} \to 0 $ 
as $\eta \to 0$ by  Lemma \ref{l:xi-xi'} and Remark \ref{r:conv-psi-eta}. 
Therefore, by \eqref{eq:W-dot-2}, 
in order to show \eqref{eq:dot-W'}, it is sufficient to prove 
	\begin{equation}\label{eq:Ric-dot}
		\frac{d}{dt} \Ric_{g_{t,P}} \big|_{t=0} (n_{g_0,0},n_{g_0,0}) 
		= \Ric_P(n_{g_0,0},n_{g_0,0}).
	\end{equation}

To this end, let $(x^1,x^2,x^3)$ denote 
the coordinates of $(B_{10}, g_{t,P})$ 
with $g_0( (\partial_{\a})_x, ( \partial_{\b} )_x) = \d_{\a \b}$ and 
define $R_{t,\a \b}$ as the component of the Ricci tensor in these coordinates:  
	\[
		R_{g_{t,P},\a \b} (x) = \Ric_{g_{t,P}}(x) 
		\left( (\partial_{\a})_x, (\partial_{\b})_x  \right). 
	\]
We also write $\Gamma^{\gamma}_{g_{t,P}, \lambda \nu}$ 
for the Christoffel symbol in the above coordinates. 
Then arguing as in the proof of Lemma 4.2 in \cite{IMM1}, we obtain 
	\begin{equation}\label{eq:Gamma-dot}
		\frac{d}{d t} \Gamma^{\kappa}_{g_{t,P}, \lambda \mu} \Big|_{t=0} 
		= \frac{1}{2} \delta^{\kappa \xi} 
		\left( \partial_{\lambda} h_{P,\xi \mu} + \partial_{\mu} h_{P,\xi \lambda} 
		- \partial_{\xi} h_{P,\lambda \mu} \right). 
	\end{equation}
We also remark that $\Gamma^{\kappa}_{g_0, \lambda \nu} \equiv 0$. 
Hence, from the formula 
	\[
		R_{g_{t,P}, \alpha \beta} 
		= \partial_{\rho} \Gamma^{\rho}_{g_{t,P}, \beta \alpha} 
		- \partial_{\beta} \Gamma^{\rho}_{g_{t,P}, \rho \alpha} 
		+ \Gamma^{\rho}_{g_{t,P}, \rho \lambda} 
		\Gamma^{\lambda}_{g_{t,P}, \beta \alpha} 
		- \Gamma^{\rho}_{g_{t,P}, \beta \lambda} 
		\Gamma^{\lambda}_{g_{t,P}, \rho \alpha}
	\]
it follows that 
	\[
		\frac{d}{dt} R_{g_{t,P}, \alpha \beta} \Big|_{t=0} 
		= \partial_{\rho} 
		\frac{d}{d t} \Gamma^{\rho}_{g_{t,P}, \beta \alpha} \Big|_{t=0} 
		- \partial_{\beta} 
		\frac{d}{d t} \Gamma^{\rho}_{g_{t,P}, \rho \alpha} \Big|_{t=0}.
	\]

		Now, by\eqref{eq:defh-2}, one observes that 
	\[
		\partial_{\zeta} \partial_{\xi} h_{P,\alpha \beta} 
		= \frac{1}{3} R_{\alpha \mu \nu \beta} 
		( \delta^{\mu}_{\xi} \delta^{\nu}_{\zeta} 
		+ \delta^{\nu}_{\xi} \delta^{\mu}_{\zeta} ) 
		= \frac{1}{3} ( R_{\alpha \xi \zeta \beta} + R_{\alpha \zeta \xi \beta}).
	\]
Thus from \eqref{eq:Gamma-dot}, we see that 
	\[
		\begin{aligned}
			\partial_{\rho} 
		\frac{d}{d t} \Gamma^{\rho}_{g_{t,P}, \beta \alpha} \Big|_{t=0}
		&= \frac{1}{2} \delta^{\rho \xi} 
		(\partial_{\rho} \partial_{\beta} h_{P,\xi \alpha} 
		+ \partial_{\rho} \partial_{\alpha} h_{P,\xi \beta} 
		- \partial_{\rho} \partial_{\xi} h_{P,\alpha \beta} ) 
		\\
		&= \frac{1}{2} \sum_{\rho = 1}^{3} 
		( \partial_{\rho} \partial_{\beta} h_{P,\rho \alpha} 
		+ \partial_{\alpha} \partial_{\rho} h_{P,\rho \beta} 
		- \partial_{\rho} \partial_{\rho} h_{P,\alpha \beta} )
		\\
		&= \frac{1}{6} \sum_{\rho = 1}^{3} 
		( R_{\rho \rho \beta \alpha} + R_{\rho \beta \rho \alpha} 
		+ R_{\rho \rho \alpha \beta} + R_{\rho \alpha \rho \beta} 
		+ 2 R_{\rho \alpha \rho \beta} ) 
		= \frac{2}{3} R_{\alpha \beta}, 
		\\
		\partial_{\beta} 
		\frac{d}{d t} \Gamma^{\rho}_{g_{t}, \rho \alpha} \Big|_{t=0}
		&= \frac{1}{2} \delta^{\rho \kappa} 
		( \partial_{\beta} \partial_{\rho} h_{\kappa \alpha} 
		+ \partial_{\beta} \partial_{\alpha} h_{\kappa \rho} 
		- \partial_{\beta} \partial_{\kappa} h_{\rho \alpha} ) 
		= \frac{1}{2} \partial_{\beta} \partial_{\alpha} 
		\sum_{\rho=1}^{3} h_{\rho \rho} 
		= - \frac{1}{3} R_{\alpha \beta}.
		\end{aligned}
	\]
Hence, we have 
	\[
		\frac{d}{dt} R_{g_{t,P}, \alpha \beta} \Big|_{t=0} 
		= R_{\alpha \beta},
	\]
which yields \eqref{eq:Ric-dot}, and we complete the proof. 
\end{proof}

\

\begin{proof}[Proof of Proposition \ref{p:new}] 
By Lemma \ref{l:dot-W'} and $t=\e^2$, for every $\delta\in (0,1/2]$, 
we find $\eta_\delta > 0$ and $C_\delta $ such that 
	\[
		\left| d W_{g_\e}[\psi_{2,\delta,\eta}] 
		- d W_{g_0}[\psi_{2,\delta,\eta}] + \e^2 
		\tilde{W}_P[ \psi_{2,\delta,0}] \right| 
		\leq C_\delta \left( o_\eta(1) \e^2 + \e^3  \right)
	\]
for all $0 < \eta \leq \eta_\delta$ and $ 0 < \e \leq 1/2$. 
We remark that $d H_{g_{t,P},0} / dt |_{t=0}$ is smooth on $\SA$ and 
it follows from the proof of \cite[Lemma 4.2]{IMM1} that 
	\[
		F(q) = \frac{d }{dt} H_{g_{t,P},0} \big|_{t=0} (q) 
		\qquad (q \in \SA). 
	\]
Noting that $ 0 \in \SA$ and $\psi_{2,\delta,0} = (1-\chi_\delta) \psi_0$ 
is also smooth on $\SA$, one has 
	\begin{equation}\label{eq:new-est-chi-d}
		\int_{\SA} (\Delta_{\SA} F) \psi_{2,\delta,0} d \sigma 
		= \int_{\SA} \left\{\Delta_{\SA} (F - F(0)) \right\} (1-\chi_\delta) \psi_0 
		d \sigma 
		= \int_{\SA} (F - F(0)) \Delta_{\SA} 
		\left\{ (1-\chi_\delta) \psi_0 \right\}.
	\end{equation}
Therefore, to prove Proposition \ref{p:new}, it is enough to show that 
	\[
		\int_{\SA} (F -F(0)) \Delta_{\SA} \left\{ (1-\chi_\delta) \psi_0 \right\} 
		d \sigma 
		= \int_{\SA} (1 - \chi_\delta) F \Delta_{\SA} \psi_0 d \sigma
		+ O(\delta). 
	\]
Since
	\[
		\Delta_{\SA} \left( ( 1 - \chi_{\delta} ) \psi_{0} \right) 
		= - (\Delta_{\SA} \chi_{\delta} ) \psi_{0} 
		- 2 g_{\SA} ( {\nabla}_{\SA} \chi_{\delta} , {\nabla}_{\SA} \psi_{0} ) 
		+ ( 1 - \chi_{\delta} ) \Delta_{\SA} \psi_{0}, 
	\]
it suffices to prove that 
	\begin{equation}\label{eq:est-chi-d}
		\int_{\SA} |F-F(0)| 
		 \left\{ \left|\Delta_{\SA} \chi_{\delta} \right| | \psi_0 | + 
		\left| g_{\SA} (\nabla_{\SA} \chi_\delta, \nabla_{\SA} \psi_0) \right| 
		\right\} d \sigma 
		= O(\delta) 
		= \int_{\SA} F(0) (1-\chi_\delta) 
		\Delta_{\SA} \psi_0 d \sigma. 
	\end{equation}

Since we may suppose that $\chi_{\delta}$ is radially symmetric, i.e. 
$\chi_{\delta}(x)=\chi_{\delta}(|x|)$, 
we observe that $\chi_{\delta}(|\mathcal{X}(\theta,\varphi)|)$ depends only on $\theta$. 
For the definition of $\mathcal{X}(\theta,\varphi)$, see Proposition \ref{p:new}. 
Furthermore, we may also assume 
	\begin{equation}\label{eq:deri-chi-d}
		\begin{aligned}
			&|\chi_{\delta}'(|x|)| \leq C_{0} \delta^{-1}, \qquad 
			|\chi_{\delta}''(|x|)| \leq C_{0} \delta^{-2}, 
			\\
			&
			{\rm supp}\, ( \chi_{\delta} ) \cap S^{2}_{\tilde A} 
			\subset 
			\{ (\theta, \varphi ) \in [0,\pi] \times [0,2\pi] 
			\;:\; | \theta - \pi | \leq C_{0} \delta \} 
			=:I_\delta
		\end{aligned}
	\end{equation}
for all $\delta \in (0,1/2)$ where $C_{0}>0$ is independent of $\delta$. 
Using $\mathcal{X}(\theta,\varphi)$ as a coordinate of $\SA$ and 
writing $\psi_0 = A \cos \theta + B (1-\cos \theta) \cos 2 \varphi$ 
where $(A,B) = (\sqrt{2}/2, (2-\sqrt{2}) / 4)$ by \eqref{eq:tranfsphere}, 
it is easily seen that 
	\[
		\begin{aligned}
			\Delta_{\SA} \psi_0 &= 
			\frac{1}{\tilde{A}^2} \left\{ 2 \cos \theta (- A + B \cos 2 \varphi ) 
			- \frac{4B  (1-\cos \theta) \cos 2 \varphi}{\sin^2 \theta} \right\}
			\\
			&= \frac{1}{\tilde{A}^2} \left[ - 2 A \cos \theta 
			+ 2 B \cos 2 \varphi 
			\left\{ \cos \theta - \frac{2 (1-\cos\theta)}{\sin^2 \theta} \right\} \right]. 
		\end{aligned}
	\]
Thus, by \eqref{eq:deri-chi-d} and 
the fact that $\chi_\delta(|\mathcal{X}(\theta,\varphi)|)$ depends only on $\theta$, 
we have 
	\[
		\begin{aligned}
			&\int_{\SA} F(0) (1-\chi_\delta) \Delta_{\SA} \psi_0 d \sigma 
			\\
			= \,& 
			F(0) \int_0^\pi \int_0^{2\pi} 
			(1-\chi_\delta) 
			 \left[ - 2 A \cos \theta 
			 + 2 B \cos 2 \varphi 
			 \left\{ \cos \theta - \frac{2 (1-\cos\theta)}{\sin^2 \theta} \right\} \right] 
			 \sin \theta 
			 d \varphi d \theta 
			 \\
			 = \, &
			 - 4\pi A F(0) \int_0^\pi 
			 (1-\chi_\delta) \sin \theta \cos \theta d \theta = 
			 4 \pi A F(0) \int_{0}^{\pi} \chi_\delta \sin \theta \cos \theta d \theta 
			 =  O(\delta^2). 
		\end{aligned}
	\]

		On the other hand, 
since $\mathcal{X}= \tilde{A} (n_0 + \ex)$ where 
$n_0$ is the outer unit normal to $\SA$, 
$|\mathcal{X}(\theta,\varphi)|^2 = 2 \tilde{A}^2 ( 1 + \cos \theta)$ 
and $ \sin \theta \sim \pi - \theta \sim \sqrt{1+\cos \theta}$ 
for $| \theta - \pi | \leq C_0 \delta$, 
one observes that for $(\theta,\varphi) \in I_\delta$, 
	\begin{equation}\label{eq:deri-chi-d-2}
		\begin{aligned}
			\left| 
			\partial_{\theta} \left\{ \chi_{\delta} ( |\mathcal{X}(\theta,\varphi)| ) \right\} 
			\right|
			&= 
			\left| \chi_{\delta}'(|\mathcal{X}|) \left\la \frac{\mathcal{X}}{|\mathcal{X}|}, 
			\partial_{\theta} \mathcal{X} \right\ra \right| 
			= \frac{\tilde{A}}{\sqrt{2}} \left| \chi_\delta'(|\mathcal{X}|) 
			\frac{\sin \theta}{\sqrt{1+\cos \theta}} \right|
			\leq C_1 \delta^{-1},
			\\
			\left| 
			\partial_{\theta}^{2} \left\{ \chi_{\delta} ( |\mathcal{X}(\theta,\varphi)| ) \right\}
			\right|
			&\leq \left| \chi_{\delta}''(|\mathcal{X}|) 
			\left(\frac{ \la \mathcal{X}, \partial_{\theta} \mathcal{X} \ra }
			{|\mathcal{X}|} \right)^2 \right| 
			+ \left| \chi_{\delta}'(|\mathcal{X}|) 
			\partial_\theta \left\la \frac{\mathcal{X}}{|\mathcal{X}|}, 
			\partial_{\theta} \mathcal{X} \right\ra \right| 
			\\
			& \leq C_1 \left( \frac{1}{\delta^2} 
			+ \frac{1}{\delta \sin \theta } \right) .
		\end{aligned}
	\end{equation}
By \eqref{eq:tranfsphere}, we have 
$\partial_\theta \psi_0 = - A \sin \theta + B \sin \theta \cos 2 \varphi$. 
Therefore, we obtain
	\[
		\begin{aligned}
		\left|\Delta_{\SA} 
		\left\{ \chi_{\delta} \left( |\mathcal{X}(\theta,\varphi)| \right) \right\} \right|
		&= 
		\left| \frac{1}{\tilde{A}^{2} \sin \theta} 
		\partial_{\theta} \left[ \sin \theta \partial_{\theta} 
		\left\{ \chi_{\delta} \left( |\mathcal{X}| \right) \right\} \right] \right| 
		\leq C_{2} \left( \frac{1}{\delta^{2}} + \frac{1}{\delta \sin \theta}  \right),
		\\ 
		\left| g_{\SA}( {\nabla}_{\SA} \chi_{\delta} , 
		{\nabla}_{\SA} \psi_{0} ) \right| 
		&= \left| \frac{1}{\tilde{A}^2 } 
		 \partial_{\theta} \chi_{\delta} \partial_{\theta} \psi_{0} \right| 
		\leq C_{2} \delta^{-1}. 
		\end{aligned}
	\]
Finally, by the continuity of $F$ at the origin, one sees 
$|F(\mathcal{X}(\theta,\varphi)) - F(0)| \leq C_2 \delta$ in $I_\delta$. 
Hence, noting $\psi_0 \in L^\infty(\SA)$, we get 
	\[
		\begin{aligned}
			\int_{\SA}
			\left|F(\mathcal{X}(\theta,\varphi)) - F(0) \right| 
			\left|\Delta_{\SA} 
			\left\{ \chi_{\delta} \left( \mathcal{X}(\theta,\varphi) \right) \right\} \right| 
			| \psi_0 | 	d \sigma 
			&\leq
			C_3  \int_{I_{\delta}} 
			\left( \frac{1}{\delta} + \frac{1}{\sin \theta} \right)
			\tilde{A}^{2} \sin \theta d \theta d \varphi 
			\\
			&\leq C \int_{0}^{C_{0} \delta} 
			( \delta^{-1} \theta + 1 ) d \theta 
			\leq C \delta.
		\end{aligned}
	\]
Similarly, we obtain
	\[
		\int_{\SA} 
		\left| F(\mathcal{X}(\theta,\varphi)) - F(0) \right|
		\left| g_{\SA} ( {\nabla}_{\SA} \chi_{\delta} , 
		{\nabla}_{\SA} \psi_{0} )  \right| d \sigma 
		\leq C \int_{0}^{C_{0} \delta} \theta d \theta 
		\leq C \delta.
	\]
Thus \eqref{eq:est-chi-d} holds and we complete the proof. 
\end{proof}

\subsection{Proof of Proposition \ref{p:varvar2}}

By \eqref{eq:differencediff}, $W_{g_0}' = 0$ and 
Propositions \ref{p:varvar1} and \ref{p:new}, 
for each $\delta \in (0,1/2)$, there exists $\eta_\delta > 0$ such that 
$$
  dW_{g_\e}[\var_\eta] 
  = -  \eta \e^2  \int_{\SA} (1 - \chi_\delta) 
   \left( F\Delta_{\SA} \psi_0 + \Ric_{P}(n_{0},n_{0}) H_{\SA} \psi_0 \right) d \sigma 
    + O(\delta \eta \e^2) + 
    O_\delta \left( o_\eta(1) \eta \e^2 + \eta \e^3 \right)
$$
holds for all $\e \in (0,1/2)$ and $\eta \in (0,\eta_\delta)$. 
The next proposition evaluates the first term in the right-hand side of the above formula 
and will be proved in an appendix as it consists of long explicit computations.

\begin{pro}\label{p:PV} One has 
	\begin{equation}\label{eq:PV}
		 \int_{\SA} (1-\chi_\delta) \left( F\Delta_{\SA} \psi_0 
  		+ \Ric_{P}(n_{0},n_{0}) 
  		H_{\SA} \psi_0 \right) d \sigma 
		= \frac{16}{3} \pi B \tilde{A} ( R_{22} - R_{33} ) 
		+ O(\delta^2). 
	\end{equation}
where $\tilde{A} = \sqrt[4]{2} \sqrt{\pi}$ and $B=\frac{2-\sqrt{2}}{4}$. 
\end{pro}

\begin{rem}\label{r:pv}
	From \eqref{eq:new-est-chi-d} and \eqref{eq:est-chi-d} 
in the proof of Proposition \ref{p:new}, Proposition \ref{p:PV} yields 
	\[
		\int_{\SA} (\Delta_{\SA} F) \psi_0 + \Ric_P(n_0,n_0) H_{\SA} \psi_0 d \sigma 
		= \frac{16}{3} \pi B \tilde{A} (R_{22} - R_{33}). 
	\]
\end{rem}

\

\noindent When considering the variation in $\eta$ of the surface $\Sigma_{\e,P,Id,\omega}$ with $\eta = 1 - |\omega|$, the 
normal component of the variation vector field will be given by 
$$
  \var_{\e ,\eta} := g_{\e,P} \left( \mathcal{Z} \right)
  \left[ \frac{\pa \mathcal{Z}}{\pa \eta}, n_{\e ,P,\eta} \right], 
$$
where $\mathcal{Z}$ is defined in Subsection \ref{ss:32} and 
where $n_{\e ,P,\eta}$ stands for the unit outer normal to $\Sigma_{\e,P,Id,\omega}$ 
in $(\R^3, g_{\e ,P})$.  
By \eqref{eq:ge=d+eh}, for any compact set $K \subset \DD$, 
one finds that 
$$
  \var_{\e ,\eta}  = g_{0}\left[ \frac{\pa \mathcal{Z}}{\pa \eta}, n_{0,\eta} \right] 
  + \kappa_{\e, \eta} 
  = \var_\eta + \kappa_{\e, \eta}, 
$$
where $\kappa_{\e, \eta}$ is a smooth function satisfying  
$$
  \|\kappa_{\e, \eta} \|_{C^4(\Sigma_{\e,P,Id,\omega})} \leq C_K \e ^2 \qquad 
  \quad   \hbox{ for } \o \in K. 
$$
Using Lemma \ref{l:appsol} and $ | \omega | = 1 - \eta$, we find that 
\begin{eqnarray*}  
 \frac{\pa }{\pa \o} W_{g_\e}(\Sigma_{\e,P,Id,\omega})  & = & 
 - dW_{g_\e}[\var_{\e ,\eta}] = - d W_{g_\e} [\var_{\eta} + \kappa_{\e, \eta}] 
 \\ 
 &  = & - dW_{g_\e}[\var_{\eta} ]  - dW_{g_\e}[\kappa_{\e ,\eta} ] 
 =  - dW_{g_\e} [\var_{\eta} ]  + O_K(\e ^4).   
\end{eqnarray*}
By Proposition \ref{p:PV} and the formula before it, the conclusion follows.

\begin{rem}\label{r:varvar2}
Let $R \in SO(3)$ and $r \in [0,1)$. As in Proposition \ref{p:varvar2}, 
we have the following estimate: ($\eta := 1 - r$, $\tilde{\eta} \in (0,\eta_\delta )$)
	\begin{equation}\label{eq:varvar2}
		\left| \frac{\partial}{\partial r} W_{g_\e} 
		( \Sigma_{\e ,P,R,r \ex} ) - \eta \e^2 
		\frac{16}{3} \pi B \tilde{A} \mathcal{F}(P,R) \right|
		\leq \left[ C_0 \delta + C_\delta \left\{o_\eta(1) + \e \right\} \right] \eta \e^2
		+ C_{\tilde{\eta}} \e^4
	\end{equation}
for all $\eta \in (\tilde{\eta}, \eta_\delta]$ and $\e \in (0,1/2)$ 
where $\mathcal{F}(P,R) := \Ric_P ( R \mathbf{e}_y, R \mathbf{e}_y ) 
- \Ric_P ( R \mathbf{e}_z, R \mathbf{e}_z )$.

		To see that \eqref{eq:varvar2} holds, 
we first remark that from the definitions of $\Sigma_{\e ,P,R,r \ex}$ and 
the map $T_{r {\bf e}_x}$ in Proposition \ref{p:disk}, we have 
	\[
	\Sigma_{\e ,P,R,r \ex} 
	= \exp^{g_\e}_P( R \T_{r \ex}  ).  
	\]
Notice that $(R \T_{r \ex} , g_{\e ,P})$ is isometric to $(\T_{r \mathbf{e}_x}, 
R^\ast g_{\e ,P})$. Putting $t = \e^2$ and 
$g_{t,P,R} = R^\ast g_{\e ,P}$, 
it follows from the proof of \eqref{eq:Ric-dot} that 
	\[
		\frac{d}{dt} \Ric_{g_{t,P,R}} (x) \Big|_{t=0} (n_{0,0}, n_{0,0}) 
		= \frac{d}{d t} \Ric_{g_{t,P}} (R x) 
		\Big|_{t=0} ( R n_{0,0}, R n_{0,0} ) 
		= \Ric_P( R  n_{0,0}, R  n_{0,0} ). 
	\]
Moreover, from the proof of Proposition \ref{p:PV} in Appendix I, 
we also observe that 
	\[
		\int_{\SA} (1 - \chi_\delta) 
		\left( F_{P,R} \Delta_{\SA} \psi_0 + \Ric_P(R n_{0,0}, R n_{0,0}) \right) 
		\rd \sigma 
		= \frac{16}{3} \pi B \tilde A \mathcal{F}(P,R ) + O(\delta^2)
	\]
where $F_{P,R} = d  / d t |_{t=0} H_{g_{t,P,R}, \SA}$ 
the metric derivative of the mean curvature of $\SA$. 
Thus, we obtain 
	\[
		\frac{\partial}{\partial r} W_{g_\e}(\Sigma_{\e ,P,R,r \ex}) 
		= \frac{\partial}{\partial r} W_{g_{\e ,P,R}}
		 ( \T_{r \mathbf{e}_x} ) 
		= - \frac{\partial}{\partial \eta} W_{g_{\e ,P,R}}
		( \T_{ \xi_\eta } )
	\]
where $\eta := 1 - |\omega|$. 
Combining these facts, arguing as in the proof of Proposition \ref{p:varvar2}, 
one can check that \eqref{eq:varvar2} holds. 
\end{rem}

\section{Proof of the main theorems}\label{s:pf}
In this section we collect all the estimates and expansions established so far in order to prove our main results, namely Theorems \ref{thm:Exixtence} and \ref{thm:Multiplicity}.
\\

For $r \in (0,1)$ , we consider the compact set of the unit disk $\DD$
$$
K_r = \left\{ |\o | \leq r  \right\}. 
$$
Then by the definition of \eqref{eq:tildetKe} one sees that 
	\[
		\partial \mathcal{T}_{\e , K_r} 
		= \left\{ \exp_P^{g_\e} ( R \T_\o )  \; : \; P \in M, \ 
		R \in SO(3), \ |\omega | = r  \right\} 
		= \left\{  
		\exp_{P}^{g_\e} ( R \T_{r \mathbf{e}_x} ) \;:\; 
		P \in M, \ R \in SO(3) \right\}
	\]
and $\partial \mathcal{T}_{\e ,K_r}$ is parametrised by 
$ M \times SO(3)$ through the map 
$(P,R) \mapsto \exp_{P}^{g_\e} ( R \T_{r \mathbf{e}_x} )$.

\begin{rem}\label{r:5new}
Notice that, from the geometric point of view, the above parametrization  of $\partial \mathcal{T}_{\e,K_r}$ is counting twice each torus: indeed, due to planar symmetry, for every $r<1$ there exists a nontrivial rotation $R\in SO(3)$ such that $R T_{r {\bf e}_x}$ and $T_{r {\bf e}_x}$ are just different parametrizations of the same torus. This is the reason for the appearance of the factor $\frac 12$ in the definition of $\tilde{C}_q$ in \eqref{eq:10-nov}. 
\end{rem}

\noindent 
Using this map, we have the following estimate for 
$W_{g_\e}$ on $\partial\mathcal{T}_{\e ,K_r}$:

\begin{pro}\label{p:globest}
	Fix $\delta \in (0,1/2)$. 
	Then there exist $C_0>0$, $r_\delta \in (0,1)$ and 
	$C_\delta>0$ satisfying the following property: 
	for every $ r_\delta \leq r < \tilde{r} < 1$, 
	one may find $C_{\tilde{r}}>0$ such that  
	$$
		\begin{aligned}
			& \left\| 
			W_{g_\e} (\Sigma_{\e,P,R,r \ex  })  - 8 \pi^2 
			+ \frac{8 \sqrt{2} \pi^2 \e^2}{3} \left( \Sc_P 
			+ \frac{B \tilde{A}}{\sqrt{2} \pi}  (1 - r)^2 \mathcal{F}(P,R) \right)  
			\right\|_{C^2(M \times SO(3))} 
			\\
			\leq & \,
			\e^2 
			\left[  C_0 \e +  o_{\tilde{r}}(1) + \left[ C_0 \delta + C_\delta 
			\left\{o_r(1) + \e \right\} \right] 
			\left( 1 - r \right)^2 + C_{\tilde{r}} \e^2 (\tilde{r}-r)
			\right]
		\end{aligned}
	$$
for all $\e \in (0,1/2)$. 
Here $o_{\tilde{r}}, o_r(1) \to 0$ as $\tilde{r}, r \uparrow 1$. 
\end{pro}

\begin{proof}
Fix $\delta \in (0,1/2)$. 
We claim that there exist $C_0>0$, $r_\delta \in (0,1)$ and $C_\delta > 0$ 
satisfying the following properties: 
for every $ r_\delta \leq r < \tilde{r} < 1$ one may find $C_{\tilde{r}} > 0$ such that 
	\begin{align}
		&\limsup_{r \uparrow 1}   \frac{1}{\e^2} \left\|W_{g_\e}(\Sigma_{\e,P,R,r \ex}) 
		- 8 \pi^2  + \frac{8 \sqrt{2}}{3} \e^2  \pi^2 \Sc_P  
		\right\|_{C^2(M \times SO(3))} 
		\leq C_0 \, \e,
		\label{eq:wto1}
		\\
		& 
		\begin{aligned}
		&\left\|
		\frac{\partial}{\partial r} W_{g_\e} \left( \Sigma_{\e ,P,R, r \ex} \right)
		- \frac{16}{3} \pi B \tilde{A}  (1-r) \e^2 \mathcal{F}(P,R) 
		\right\|_{C^2(M \times SO(3))} 
		\\
		\leq &\, 
		\left[ C_0 \delta + C_\delta \left\{o_r(1) + \e \right\} \right] (1-r) \e^2
		+ C_{\tilde{r}} \e^4
		\end{aligned}
		\label{eq:Wr-Wom}
	\end{align}
for each $\e \in (0,1/2)$.

		We remark that in \cite[Proposition 4.6]{IMM1} (see also Proposition \ref{p:expdegtorus} above) 
and in Proposition \ref{p:varvar2} (Remark \ref{r:varvar2}) 
we have shown \eqref{eq:wto1} and \eqref{eq:Wr-Wom} in $C^0$-sense. 
To prove \eqref{eq:wto1} and \eqref{eq:Wr-Wom} in $C^2$-sense 
with respect to $(P,R)$, 
put $g_{\e ,P,R}(x) := ( \exp_P^{g_\e} \circ R )^\ast g_\e$. 
Then $(\T_{r \mathbf{e}_x}, g_{\e ,P,R})$ is isometric 
to $(\Sigma_{\e ,P,R,r \ex}, g_{\e})$. Moreover, 
it follows from \eqref{eq:ge=d+eh} and \eqref{eq:met-deri} that 
	\begin{equation}\label{eq:x-squared}
		\left| D_{P,R}^k \, g_{\e ,P,R} (x) \right| 
		\leq C_k \e^2 |x|^2. 
	\end{equation}

		For \eqref{eq:wto1}, we argue as in the proof of Proposition 4.6 
in \cite{IMM1}, namely, using the Willmore functional in the Euclidean 
space and decomposing the functional into the handle part and 
the sphere part. By \eqref{eq:x-squared}, we may observe that 
the contribution of the handle part is negligible in $C^2$-sense. 
On the other hand, the sphere part depends smoothly on 
$P$ and $R$. Combining these facts, we see that \eqref{eq:wto1} holds.

		For \eqref{eq:Wr-Wom}, we also proceed 
as in the proof of Proposition \ref{p:varvar2} and 
use the cut-off function $\chi_\delta$. 
By \eqref{eq:x-squared} and the fact that $ \T$ and its conformal deformations 
are critical points for the Euclidean Willmore functional, 
one sees that the handle part is negligible in $C^2$-sense. 
Then since the sphere part depends smoothly on $P$ and $R$, 
it follows from the expressions in Remarks \ref{r:pv} and \ref{r:varvar2} that 
a counterpart of Proposition \ref{p:PV} in $C^2$-sense holds: 
	\[
		\begin{aligned}
			& \left\| 
			\int_{\SA} 
			\left\{ \left( \Delta_{\SA}
			 \frac{d H_{g_{t,P,R}}}{d t} \Big|_{t=0} \right) 
			 + \Ric_P(R n_{0,0},R n_{0,0}) 
			 \right\} (1-\chi_\delta) \psi_0 d \sigma 
			 - \frac{16}{3} \pi B \tilde{A} \mathcal{F}(P,R)
			\right\|_{C^2}
			\leq  C_0 \delta^2. 
		\end{aligned}
	\]
Thus, arguing as in the proof of Proposition \ref{p:varvar2}, 
we can show the estimate \eqref{eq:Wr-Wom}.

		Now integrate \eqref{eq:Wr-Wom} on $[r, \tilde{r}]$ to obtain 
	\[
		\begin{aligned}
			& \left\| 
			W_{g_\e} ( \Sigma_{\e ,P,R, \tilde{r} \ex}  ) 
			- W_{g_\e}  ( \Sigma_{\e ,P,R, r \ex} ) 
			+ \frac{8}{3} \pi B \tilde{A} \e^2 \mathcal{F}(P,R) 
			\left\{ (1-\tilde{r})^2 - (1-r)^2 \right\}
			\right\|_{C^2(M \times SO(3))}
			\\
			\leq &\, 
			\left[ C_0 \delta + C_\delta \left\{ o_{r}(1) + \e \right\} \right] 
			\e^2 \left\{ (1-r)^2 - (1-\tilde{r})^2 \right\} 
			+ C_{\tilde{r}} \e^4 (\tilde{r}-r).
		\end{aligned}
	\]
From \eqref{eq:wto1} it follows that 
	\[
		\left\| W_{g_\e} ( \Sigma_{\e ,P,R, \tilde{r} \ex}  ) 
		- 8 \pi^2 + \frac{8 \sqrt{2}}{3} \pi^2 \e^2 \Sc_P \right\|_{C^2(M\times SO(3))} 
		\leq C_0 \e^2 \left\{\e + o_{\tilde{r}}(1)\right\}
	\]
Therefore, we have 
	\[
		\begin{aligned}
			& \left\| 
			W_{g_\e} ( \Sigma_{\e ,P,R, r \ex}  ) 
			- 8 \pi^2 
			+ \frac{8\sqrt{2}\pi^2}{3}\e^2 
			\left( \Sc_P +  \frac{B \tilde{A}}{\sqrt{2} \pi} 
			\mathcal{F} (P, R) (1-r)^2 \right)
			\right\|_{C^2(M\times SO(3))}
			\\
			\leq &\, 
			\e^2 \left\{ C_0 \e + o_{\tilde{r}}(1)  \right\}
			+ \left[ C_0 \delta + C_\delta \left\{ o_{r}(1) + \e \right\} \right] 
			\e^2 (1-r)^2 
			+ C_{\tilde{r}} \e^4 (\tilde{r}-r)
		\end{aligned}
	\]
and Proposition \ref{p:globest} follows. 
\end{proof}

\

\noindent Recall that, in the spirit of \cite{MVS}, a $C^2$ function defined on a manifold with boundary is said to satisfy the 
{\em general boundary conditions} if its gradient never vanishes at the boundary and if its 
restriction to the boundary is a Morse function. We have then the following result. 

\begin{lem}\label{l:gbdrycond}
For $\e > 0$ small,  let $\Phi_\e$ be defined as in Proposition \ref{p:variational}. 
Then, under the assumptions of Theorem \ref{thm:Exixtence}, 
there exists  $r_0 \in (0,1)$ satisfying the following property: 
for all $r \in [r_0,1)$ one may find $\e_r > 0 $ such that 
if $\e \in (0,\e _r]$, 
$\Phi_\e$ satisfies the general boundary conditions on $\partial \mathcal{T}_{\e ,K_{r}}$ 
\end{lem}

\begin{proof}
For $ 0 < r < 1$, 
set 
	\[
		\mathcal{G}_{r} (P,R) 
		:= - \Sc_P - \frac{B \tilde{A}}{\sqrt{2} \pi} \mathcal{F}(P,R) 
		(1-r)^2. 
	\]
We divide our arguments into several steps: 

\medskip

\noindent
\textbf{Step 1:} {\sl There exist $r_1 \in (0,1)$ and $\zeta_0>0$ such that if 
$ r_1 \leq r < 1$ and a function $\mathcal{H}(P,R) \in C^2(M\times SO(3))$ satisfies
	\begin{equation}\label{eq:err-esti}
		\begin{aligned}
		& \frac{1}{\e^2} 
		\left\| 
		D_P^\ell \left( \mathcal{H}(P,R)
		- \frac{8\sqrt{2}\pi^2}{3} \e^2 \mathcal{G}_{r}(P,R) \right)
		\right\|_{L^\infty (M \times SO(3)) } 
		\\
		& + 
		\frac{1}{\e^2 (1-r)^2} 
		\left\| 
		D_R D_{P,R}^{m} \left( \mathcal{H}(P,R) 
		- \frac{8\sqrt{2}\pi^2}{3} \e^2 \mathcal{G}_{r}(P,R) \right)
		\right\|_{L^\infty (M \times SO(3)) } \leq \zeta_0
		\end{aligned}
	\end{equation}
for $ 1 \leq \ell \leq 2$ and $ 0 \leq m \leq 1$, then 
$\mathcal{H}$ is a Morse function on $M \times SO(3)$ and 
	\begin{equation}\label{eq:num-index}
		\begin{aligned}
			& \sharp \{ (P,R) \in M \times SO(3)  \;:\; 
			\nabla  \mathcal{G}_r(P,R) = 0, \quad 
			\mathrm{index} ( \nabla^2 \mathcal{G}_r(P,R) ) = q \}
			\\
			= &\, 
			\sharp \{ (P,R) \in M \times SO(3)  \;:\; 
			\nabla \mathcal{H}(P,R) = 0, \quad 
			\mathrm{index} ( \nabla^2 \mathcal{H}(P,R) ) = q \}
		\end{aligned}
	\end{equation}
for each $ 0 \leq q \leq 6$. 
Moreover, when $\zeta_0 \to 0$, 
any critical point $(Q,R)$ of $\mathcal{H}(P,R)$ satisfies 
$\min_{1 \leq i \leq k} |Q - P_i|_g \to 0$ where 
$P_i$ are the  critical points of $\Sc_P$ which by assumption satisfy  $(ND1)$ of the Introduction.  
}

\medskip

In fact, since $\Sc_P$ is a Morse function by  assumption and 
	\[
		\left\| 
		D_P^\ell \left\{ \frac{\mathcal{H}(P,R)}{\e^2}
		+ \frac{8\sqrt{2}\pi^2}{3} 
		\left( \Sc_P + \frac{B \tilde{A}}{\sqrt{2} \pi} \mathcal{F}(P,R) 
		(1-r)^2 \right) \right\}
		\right\|_{L^\infty (M \times SO(3)) } \leq \zeta_0,
	\]
holds, we may find some $r_1 \in (0,1)$ and $\zeta_0>0$ 
so that when $r \in [r_1,1)$, 
the function $P \mapsto \mathcal{H}(P,R)$ is a Morse function on $M$, 
the number of critical points  is the same as that of $ P \mapsto \Sc_P$ and 
if $D_P \mathcal{H}(Q,R) = 0$ for some $Q \in M$, then 
$Q$ must be close to one of $P_i$ ($1 \leq i \leq k$) by ($ND1$). 
Therefore, enlarging $r_1 \in (0,1)$ and shrinking $\zeta_0>0$ if necessary, 
($ND2$) of the Introduction implies that 
the function $R \mapsto \mathcal{F}(Q,R)$ is a Morse function 
provided $D_{P} \mathcal{F}(Q,R_0)=0$. 
Since it follows from \eqref{eq:err-esti} that 
	\[
		\left\| 
		D_R D_{P,R}^{m} \left( \frac{\mathcal{H}(P,R)}{\e^2 (1-r)^2 }
		- \frac{8\pi}{3} B \tilde{A} \mathcal{F}(P,R) \right)
		\right\|_{L^\infty (M \times SO(3)) } 
		\leq \zeta_0 \quad 
		(0 \leq m \leq 1),
	\]
one sees that 
$R \mapsto \mathcal{H}(Q,R)$ is a Morse function on $SO(3)$ and 
the number of critical points is the same as of 
$R \mapsto \mathcal{F}(P_i,R)$ 
where $|Q-P_i|_g = \min\{|Q-P_j|_g \; : \; 1 \leq j \leq k\}$ 
provided 
$D_{P} \mathcal{H}(Q,R) = 0$ and $\zeta_0>0$ is sufficiently small. 
Therefore, if $r_1 \leq r < 1$ and $\mathcal{H}$ satisfies \eqref{eq:err-esti}, 
then the number of critical points of $\mathcal{H}$ are the same as that of 
$\mathcal{G}_r(P,R)$.

		Next, let $\nabla \mathcal{H}(P,R) = 0$ and observe 
from \eqref{eq:err-esti} that 
	\[
		\begin{aligned}
			&\mathrm{det} \left(\nabla^2 \mathcal{H}(P,R) \right)
			\\
			=\,& \e^{12} (1-r)^6 \left\{ 
			\mathrm{det} \left( D_P^2 \frac{\mathcal{H}(P,R)}{\e^2} \right) 
			\mathrm{det} \left( D_R^2 \frac{\mathcal{H}(P,R)}{\e^2(1-r)^2} \right) 
			+ O((1-r)^2)
			 \right\}
			 \\
			 =\, & \e^{12} (1-r)^6 \left( \frac{8\sqrt{2}\pi^2}{3} \right)^6
			 \left\{ 
			 \mathrm{det} \left( -D_P^2 \ScP \right) 
			 \mathrm{det} \left( D_R^2 \frac{B \tilde{A}}{\sqrt{2} \pi} 
			 \mathcal{F}(P,R) \right) 
			 + O((1-r)^2) + O(\zeta_0)
			 \right\}
		\end{aligned}
	\]
Thus replacing $r_1$ and $\zeta_0$ by larger and smaller one respectively, 
we observe that $\mathcal{H}$ is a Morse function on $M \times SO(3)$ and 
the indices of $\mathcal{H}$ and $\mathcal{G}_r(P,R)$ coincide 
if $r \in [r_1,1)$ and $\mathcal{H}$ satisfies \eqref{eq:err-esti}.

Finally, it is easily seen that $\min_{1 \leq i \leq k} | Q - P_i |_g \to 0$ 
for any $Q \in M$ satisfying $D_P \mathcal{H}(Q,R) = 0$ as $\zeta_0 \to 0$.

\medskip

\noindent
\textbf{Step 2:} 
{\sl One may find $r_0 \in [r_1,1)$ and $\e_1>0$ such that 
for all $r \in [r_0,1)$, there exists $\tilde{\e}_r > 0$ so that 
if $ 0< \e \leq \tilde{\e}_r$, then 
the function $W_{g_\e }(\Sigma_{\e ,P,R, r \ex} )$ satisfies 
\eqref{eq:err-esti} with $\zeta_0$ replaced by $\zeta_0/2$ .  }

\medskip

We first recall Proposition \ref{p:globest}: 
if $r_\delta \leq r < s < 1$, then 
	\begin{equation}\label{eq:base}
		\begin{aligned}
			& \frac{1}{\e^2} 
			\left\| 
			D_P^\ell \left( W_{g_\e} (\Sigma_{\e ,P,R,r \ex } ) 
			- \frac{8\sqrt{2}\pi^2}{3} \e^2 \mathcal{G}_{r}(P,R) \right)
			\right\|_{L^\infty (M \times SO(3)) } 
			\\
			& + 
			\frac{1}{\e^2 (1-r)^2} 
			\left\| 
			D_R D_{P,R}^{m} \left( W_{g_\e} (\Sigma_{\e ,P,R,r \ex } ) 
			- \frac{8\sqrt{2}\pi^2}{3} \e^2 \mathcal{G}_{r}(P,R) \right)
			\right\|_{L^\infty (M \times SO(3)) } 
			\\
			\leq & \,  C_0 (1-r)^{-2} \e + (1-r)^{-2}o_{s}(1) + 
			\left[ C_0 \delta + C_\delta \left\{ o_{r}(1) + \e \right\} \right] 
			+ C_{s} (1-r)^{-2} (s-r) \e^2
		\end{aligned}
	\end{equation}
for $1 \leq \ell \leq 2$ and $0 \leq m \leq 1$. 
We first fix  $\delta_0>0$ so that 
$C_0 \delta_0 \leq \zeta_0 / 8$. 
Next we select  $r_0 \geq \max\{r_\delta, r_1\}$ so that 
$C_{\delta_0} o_{r}(1) \leq \zeta_0 / 8$ for each $ r \in [r_0,1)$. 
Choose  $s_r$ sufficiently close to 1 to hold 
$(1-r)^{-2} o_{s_r}(1) \leq \zeta_0 / 8$. 
Finally, find  $\tilde{\e}_r > 0$ so that 
$C_0(1-r)^{-2} \e + C_{\delta_0} \e + C_{s_r}(1-r)^{-2} (s_r -r ) \e^2 
\leq \zeta_0 / 8$ for all $\e \leq \tilde{\e}_r$. 
Then for $r \in [r_0,1)$, it follows from \eqref{eq:base} that 
	\[
		\begin{aligned}
			& \frac{1}{\e^2} 
			\left\| 
			D_P^\ell \left( W_{g_\e} (\Sigma_{\e ,P,R,r \ex } ) 
			- \frac{8\sqrt{2}\pi^2}{3} \e^2 \mathcal{G}_{r}(P,R) \right)
			\right\|_{L^\infty (M \times SO(3)) } 
			\\
			& + 
			\frac{1}{\e^2 (1-r)^2} 
			\left\| 
			D_R D_{P,R}^{m} \left( W_{g_\e} (\Sigma_{\e ,P,R,r \ex } ) 
			- \frac{8\sqrt{2}\pi^2}{3} \e^2 \mathcal{G}_{r}(P,R) \right)
			\right\|_{L^\infty (M \times SO(3)) } 
			\leq \frac{\zeta_0}{2}
		\end{aligned}
	\]
for all $\e \in (0,\tilde{\e}_r]$. Thus Step 2 holds.

\medskip

\noindent
\textbf{Step 3:} 
{\sl 
Let $r \in [r_0,1)$ where $r_0$ is the  constant appearing in Step 2. Then we have 
the following estimate: 
	\begin{equation}\label{eq:estC2Phie}
		\left\| \Phi_\e(P,R, \o ) - W_{g_\e}( \Sigma_{\e ,P,R,\o} ) 
		\right\|_{C^2(M \times SO(3) \times \overline{B_{r}(0)})} 
		\leq C_{r} \e^4
	\end{equation}
for all $\e$ where $C_{r}$ depends only on $r$. 
}

\medskip

		Put $A_{r} := M \times SO(3) \times \overline{B_{r}}$.  
For $(P,R,\o) \in A_{r}$, 
we define 
	\[
		\begin{aligned}
			& 
			g_{0,R,\o} := \left( R T_\o \right)^\ast g_0, 
			\quad 
			g_{\e ,P,R,\o} := 
			\left( \exp_P^{g_\e} \circ R \circ T_\o \right)^\ast g_\e 
			= ( R T_\o )^\ast g_{\e ,P}, 
			\\
			& Z_{\e ,P,R,\o}(s;p) := p + s \varphi_\e (P,R,\o ; p) n_{\e ,P,R,\o} (p),
			\quad 
			\T[s \varphi_{\e} (P,R,\o) ] 
			:= \{  Z_{\e ,P,R,\o} (s;p) \; : \; p \in \T  \}
		\end{aligned}
	\]
where $p \in \T$, $ s \in [0,1]$ and $n_{\e ,P,R,\o}$ denotes 
the unit outer normal to $(\T, g_{\e ,P,R,\o})$. 
Remark that  
$(\Sigma_{\e ,P,R,\o}[s \varphi_\e (P,R,\o)],g_\e)$ 
is isometric to $(\T[s \varphi_\e(P,R,\o)], g_{\e ,P,R,\o})$. 
Then it follows from Proposition \ref{p:1-2-var} that 
	\begin{equation}\label{eq:diff}
		\begin{aligned}
			\Phi_\e (P,R,\o) - W_{g_\e}(\Sigma_{\e ,P,R,\o}) 
			&= \int_0^1 \frac{d}{d s} 
			W_{g_\e} ( \Sigma_{\e ,P,R,\o} [s \varphi_\e (P,R,\o)]  ) d s 
			\\
			&= \int_0^1 \int_{\T} 
			W_{g_{\e ,P,R,\o}}'(s) \psi_{\e ,P,R,\o} d \sigma ds 
		\end{aligned}
	\end{equation}
where $W_{g_{\e ,P,R,\o}}'(s)$ stands for the derivative of 
the Willmore functional for $(\T[s \varphi_\e(P,R,\o)], g_{\e ,P,R,\o})$, 
	\[
		\begin{aligned}
			\psi_{\e ,P,R,\o} (s;p) 
			&:= g_{\e} (Z_{\e ,P,R,\o}(s;p)) 
			\left[ \frac{d }{d s} Z_{\e ,P,R,\o} (s;p) , n_{\e ,P,R,\o} (s;p) \right] 
			\\
			&= \varphi_{\e}(P,R,\o ;p ) 
			g_{\e,P,R,\o} (Z_{\e ,P,R,\o}(s;p)) [n_{\e ,P,R,\o}(p) , n_{\e ,P,R,\o} (s;p) ]
		\end{aligned}
	\]
and $n_{\e ,P,R,\o}(s;p)$ is the unit outer normal 
for $(\T[s \varphi_\e (P,R,\o)] , g_{\e ,P,R,\o}  )$. 
Remark that since $\varphi_\e(P,R,\o)$ is small, $\T[s \varphi_\e (P,R,\o)]$ and 
$\T$ are diffeomorphic and we pull back all geometric quantities of 
$\T[s \varphi_\e (P,R,\o)]$ on $\T$.

		Now by Proposition \ref{p:lyap}, one can easily check that 
	\begin{equation}\label{eq:esti-psi}
		\sup_{(P,R,\o) \in A_{r}} 
		\left\|  D_{P,R,\o}^k \psi_{\e,P,R,\o}(s;\cdot) \right\|_{C^{4,\gamma}(\T)} 
		\leq C_{r} \e^2
	\end{equation}
for $k = 0,1,2$. On the other hand, 
$W_{g_{0,R,\o}}'(\T) = 0$ holds for all $(R,\o) \in SO(3) \times \overline{B_{r}}$ 
thanks to Proposition \ref{p:Mobinv}. 
In particular, $D_{R,\o}^k W_{g_{0,R,\o}}'(\T) = 0$ for $k=0,1,2$. 
Since \eqref{eq:ge=d+eh} and \eqref{eq:met-deri} yield
	\[
			\sup_{(P,R,\o) \in A_{r} } 
			\left\|  D_{P,R,\o}^k (g_{\e ,P,R,\o}  - g_{0,R,\o} ) 
			\right\|_{C^\ell (\overline{B_{10}(0)})} 
			\leq  C_{r,\ell} \e^2
	\]
for each $\ell \in \N$, combining the estimates 
of $\varphi_\e (P,R,\o)$ in Proposition \ref{p:lyap}, 
we obtain 
	\begin{equation}\label{eq:esti-W'}
		\sup_{(P,R,\o) \in A_{r}, s \in [0,1] } 
		\left\|  D_{P,R,\o}^k W_{g_{\e ,P,R,\o}}' (s)
		\right\|_{C^{0,\gamma}(\T)} \leq C_{r} \e^2
	\end{equation}
for $k=0,1,2$. 
Thus by \eqref{eq:diff}, \eqref{eq:esti-psi} and \eqref{eq:esti-W'}, 
we have 
	\[
		\sup_{(P,R,\o) \in A_{r} } 
		\left| D_{P,R,\o}^k 
		\left( \Phi_\e (P,R,\o) - W_{g_\e}(\Sigma_{\e ,P,R,\o}) \right) \right| 
		\leq C_{r} \e^4
	\]
for $k=0,1,2$. Thus Step 3 holds.

\medskip

\noindent
\textbf{Step 4:} 
{\sl Conclusion}

\medskip

Fix $r_0 \in (0,1)$ and $\tilde{\e}_r >0 $ from Step 2 and let $r_0 \leq r < 1$. 
By \eqref{eq:estC2Phie}, we choose $\e_r \in (0,\tilde{\e}_r]$ so that 
	\[
		\frac{1}{\e^2(1-r)^2}\left\| \Phi_\e (P,R,r \ex) - W_{g_\e}(\Sigma_{\e ,P,R,r \ex}) 
		\right\|_{C^2(M \times SO(3))} 
		\leq \frac{\zeta_0}{2}
	\]
for all $\e \in (0,\e_r]$. Hence, by Step 2, we observe that 
$\Phi_\e (P,R,r \ex )$ satisfies \eqref{eq:err-esti} with $r$. 
Thus if $\e \in (0,\e_r]$, then 
$\Phi_\e (P,R,r \ex )$ satisfies the general boundary condition 
on $\partial \mathcal{T}_{\e , K_{r}}$ and we complete the proof. 
\end{proof}

\

\begin{proof}[Proof of Theorem \ref{thm:Exixtence}]  
	We  apply the finite-dimensional reduction  
	as described in Subsection \ref{ss:22}.  By Proposition \ref{p:variational} it is sufficient to 
	find critical points of the reduced energy $\Phi_\e$.  For doing this we employ the Morse inequalities for 
	manifolds with boundary from \cite{MVS}: these relate the $q$-th Betti numbers (we choose here 
	$\Z_2$ coefficients) of the underlying manifold  to the number of critical points with  index $q$ of Morse functions 
	satisfying the general boundary conditions. Concerning the latter critical points, one has to count those at 
	the interior, plus the ones (still, of index $q$) for the function restricted to the boundary such that 
	the gradient (which is non-zero by the general boundary conditions) is pointing inwards.

	For our purpose we choose to work with the manifold $\mathcal{T}_{\e ,K_{r}}$ 
	(see \eqref{eq:tildetKe} and the beginning of this section) 
	where $r \in [r_0,1)$ and $r_0$ appears in Lemma \ref{l:gbdrycond}, 
	whose homology can be described as follows. 
	By deforming $r$ to $0$ 
	one can see that $\mathcal{T}_{\e ,K_{r}}$ retracts to the family of (exponentiated) rotated (but not M\"obius inverted) small Clifford tori 
	centred at arbitrary points of $M$. By the invariances of the Clifford torus, this set is homeomorphic 
	to the family of lines (axes of the symmetric tori) in $TM$ passing through  the base points of the tangent spaces, namely an 
	$\R \P^2$-bundle $E$ over $M$. By Remark \ref{r:betti}, since $M$ is parallelizable, we can compute the 
	homology of $E$ using K\"unneth's formula 
	$$
	H_k(E;\Z_2) = \bigoplus_{i+j=k} H_i(M;\Z_2) \oplus H_j(\R \P^2; \Z_2).  
	$$
As $H_k(\R\P^2, \Z_2)=\Z_2$ for $0\leq k \leq 2$ and zero otherwise, 
it follows that the Betti numbers (with $\Z_2$ coefficients) of $\mathcal{T}_{\e ,K_{r}}$ 
are given be the $\tilde{\beta}$'s as in \eqref{eq:defbq}.

	To prove the existence result, we set 
	$$
	  \Psi_{\e ,r}= \Phi_{\e} \big|_{\partial \mathcal{T}_{\e ,K_{r}}}, 
$$	  
	and define   
		\[
			\begin{aligned}
			\tilde{C}_q := \frac 12 \sharp \{ (P,R) \in M \times SO(3) \; : \;
			\nabla \Psi_{\e ,r} (P,R)= 0,  \  \nabla \Phi_\e (P,R, r \ex ) \hbox{ is inward, }
			\mathrm{index}
			\left(\nabla^2 \Psi_{\e,r} (P,R) 
			\right) 
			= q
			\}.
			\end{aligned}
		\]
Notice that due to Lemma \ref{l:gbdrycond}, 
for any $r \in [r_0,1)$ and $\e \in (0,\e_r]$, $\Phi_\e $ satisfies the general boundary condition 
on $\partial \mathcal{T}_{\e ,K_{r}}$, so $\tilde{C}_q$ is well-defined. 
Now we claim that $\tilde{C}_q$'s in the above formula coincide with the numbers in \eqref{eq:deftCq} 
when we fix $r \in (0,1)$ sufficiently close to $1$ and $\e \in (0,\e _r]$.

		First we remark that when $r \in [r_0,1)$ and $\e \in (0,\e_r]$, 
it follows from the proof of Lemma \ref{l:gbdrycond} and \eqref{eq:num-index} that 
	\[
		\begin{aligned}
			&
			\sharp \{ (P,R) \in M \times SO(3) \;: \; \nabla \Psi_{\e ,r} (P,R)= 0, \ 
			\mathrm{index} ( \nabla^2 \Psi_{\e ,r}(P,R) ) = q \} 
			\\
			= &\, 
			\sharp \{ (P,R) \in M \times SO(3) \;:\; \nabla \mathcal{G}_r(P,R) = 0, \ 
			\mathrm{index} ( \nabla^2 \mathcal{G}_r(P,R) ) = q  \}
			\\
			= & \, \sharp \{ (P_i, R_{i,\ell}) \in M \times SO(3) \;:\; 
			1 \leq i \leq k, \ 1 \leq \ell \leq \ell_i , \ 
			\mathrm{index} ( - D_P^2 \Sc (P_i) ) 
			+ \mathrm{index} ( - D_R^2 \mathcal{F}_i(R_{i,\ell})) = q  \}
		\end{aligned}
	\]
since $1-r$ is small. Moreover, we may also observe that 
if $\nabla \Psi_{\e ,r} (P,R) = 0$, then 
$(P,R)$ must be close to $(P_i,R_{i,\ell})$, 
$(i,\ell)$ is uniquely determined and 
the correspondence is one-to-one. 
Hence, ($ND2$) implies that 
$\mathcal{F}(P,R) \neq 0$ provided 
$\nabla  \Psi_{\e ,r} (P,R) = 0$.  
Combining this fact with \eqref{eq:varvar2} 
(or \eqref{eq:Wr-Wom}) and \eqref{eq:estC2Phie}, 
enlarging $r$ and shrinking $\e_r$, 
if $\e \in (0,\e_r]$ and $\nabla \Psi_{\e ,r} (P,R) = 0$, 
then $\nabla \Phi_\e (P,R,r \ex)$ points inward if and only if 
$\mathcal{F}(P,R) < 0$. 
Therefore, noting Remark \ref{r:5new}, our claim holds.

	In order to prove the existence of critical points of $\Phi_\e $, 
let us assume by contradiction that there is no critical point of $\Phi_\e $ in the interior 
to $\mathcal{T}_{\e ,K_{r}}$: then the Morse inequalities in 
\cite{MVS} $\tilde{C}_q \geq \tilde{\beta}_q$  would be violated for $q = 0, 1$, see \eqref{eq:defbq}. 
	This concludes the proof. 
	
Notice also that the index of the constructed  area-constrained Willmore torus coincides with the index of the corresponding  critical point of the reduced functional, since the second variation of the Willmore functional is positive definite on the orthogonal of the Kernel thanks to the work of Weiner \cite{WEINER}.
\end{proof}

\

\begin{rem}\label{r:minmax}
This remark contains a shorter proof of Theorem \ref{thm:Exixtence} under milder assumptions: however, in view of Theorem \ref{thm:Multiplicity}, it was convenient for us  to use the above Morse-theoretical 
general framework. 

Suppose $(M,g)$ is as in Remark \ref{r:mild} $(ii)$. Then, by Proposition \ref{p:globest} and 
Lemma \ref{l:gbdrycond},  for $r \in (r_0,1)$ where $r_0$ from Lemma \ref{l:gbdrycond} and 
$\e > 0$ sufficiently small,  the 
maximum (resp. minimum) of $\Phi_\e$ restricted to $\partial \mathcal{T}_{\e ,K_{r}}$ is attained 
for some $(P,R)$ such that $P$ is close to a global minimum (resp. maximum) point for the scalar curvature on $M$. 
At each of these points $P$ the Ricci tensor is not a multiple of the identity, 
so we may find $R_{P,+}, R_{P,-} \in SO(3)$ satisfying 
$\mathcal{F}(P,R_{P,-}) < 0 < \mathcal{F}(P,R_{P,+})$. 
By Proposition \ref{p:varvar2} or Remark \ref{r:varvar2}, 
the inward derivative of $\Phi_\e$ is positive (resp. negative) for $R_{P,-}$ 
(resp. $R_{P,+}$). 
Therefore the global maximum (resp. minimum) of $\Phi_\e$ on the closure of $\mathcal{T}_{\e ,K_{r}}$ 
is attained at the interior. 
\end{rem}	
	
	\
	
\noindent To prove Theorem \ref{thm:Multiplicity} we  need the following transversality result, 
see  	Theorem 1.1 in \cite{ST}.

\begin{thm}\label{t:st}
Let $X, Y,Z$ be  Banach spaces and let $U \subset X$, $V \subset Y$ be
open subsets. Let $F : V \times U \to Z$ be a $C^k$ map with $k \geq 1$ such that

(i) for any $y \in V$, $F(y, \cdot) : x \mapsto F(y, x)$ is a Fredholm map of index $l$ with
$l \leq k$;

(ii)  the operator $F'(y_0, x_0) : Y \times X \to Z$ is
onto at any point $(y_0, x_0)$ such that $F(y_0, x_0) = z_0$;

(iii) the set of $x \in U$ such that $F(y, x) = z_0$ with $y$ in a compact set of $Y$ is
relatively compact in $U$.

\noindent Then the set $\{y \in V : z_0 \hbox{ is a regular value of } F(y, \cdot)\}$ is a dense open subset of $V$.	
\end{thm}

\	
	
	\begin{proof}[Proof of Theorem \ref{thm:Multiplicity}] 
By the  Morse inequalities in  \cite{MVS} it will be sufficient to show that for generic metrics the reduced 
functional $\Phi_\e$ is a Morse function.  We will apply Theorem \ref{t:st} with $X = \R^7$ being a 
local coordinate system for $\mathcal{T}_{\e ,K_{r}}$, $Z = \R^7$, $Y$ the set of $C^2$-symmetric 
$(2,0)$ tensors $h$ on $M$ and 
$$
  F(h,x) = \nabla \Phi_{\e,g_\e+h_\e }(x), 
$$
where we highlighted the metric dependence in $\Phi_\e$, and where we are scaling coordinates as in Subsection 
\ref{ss:22}. Given any torus in $\mathcal{T}_{\e ,K_{r}}$, one can use formula \eqref{eq:wdot} (where 
$t = \e^2$) to compute the gradient of $W_{g_\e+h_\e}|_{\mathcal{T}_{\e ,K_{r}}}$. 
Localizing the 
metric variation $h_\e$ near finitely-many points of the ($1/\e$-dilated) torus,   one can arbitrarily vary 
$\nabla W_{g_\e+h_\e}|_{\mathcal{T}_{\e ,K_{r}}}$ by vectors of order $\e^2$. By \eqref{eq:estC2Phie}, 
this property will hold true also for $\nabla \Phi_{\e,g_\e+h_\e }$, so $(ii)$ in Theorem \ref{t:st} will be satisfied. 
$(i)$ and $(iii)$ are trivially satisfied as $X$ is finite-dimensional. 
\end{proof}

%
%
%
%
%
%
%
%
%
%
%
%
%
%

\section{Appendices} 
\subsection{Appendix I: proof of Proposition \ref{p:PV}}

		In this appendix we compute each term in \eqref{eq:PV}. 
We first recall our notation here:  
	\[
		\begin{aligned}
			F &= -\sum_{i=1}^{2} e_{i}(h_{ni}) + \sum_{i,j=1}^{2} 
			h_{nj} \langle \nabla^{\R^{3}}_{e_{i}} e_{i}, e_{j} \rangle 
			- \frac{1}{2} h_{nn} H_{\SA} 
			+ \frac{1}{2} \sum_{i=1}^{2} \frac{\partial h}{\partial n_0} ( e_{i}, e_{i} ), 
			\\
			\mathcal{X}(\theta, \varphi) 
			&= \tilde{A} \left( \cos \theta + 1, \sin \theta \cos \varphi , \sin \theta \sin \varphi \right)
			\qquad (\theta, \varphi) \in (0,\pi) \times [0,2\pi], 
			\\
			e_1 &= \tilde{A}^{-1} \partial_\theta \mathcal{X} 
			= (-\sin \theta , \cos \theta \cos \varphi, \cos \theta \sin \varphi ),
			\\
			e_2 &= (\tilde A \sin \theta)^{-1} \partial_\varphi X 
			= (0, - \sin \varphi, \cos \varphi),
			\\
			n_0 &= (\cos \theta, \sin \theta \cos \varphi, \sin \theta \sin \varphi),
			\\
			(h(x))_{\a \b} &= \frac{1}{3} R_{\a \mu \nu \b} x^\mu x^\nu, 
			\quad 
			h_{ni} = h(x)(n_0,e_i), \quad h_{nn} = h(x) (n_0,n_0),
			\\
			\psi_0 &= A \cos \theta + B ( 1 - \cos \theta ) \cos 2 \varphi,
			\quad 
			(A, B) = \left( \frac{\sqrt{2}}{2}, \ \frac{2 - \sqrt{2}}{4} \right).
		\end{aligned}
	\]
We remark that the Riemann curvature tensor can be expressed 
by the Ricci curvature and the scalar curvature as follows: 
(see \cite{LP}): 
	\begin{equation}\label{eq:expr-h}
	h_{\alpha \beta}(x) := 
	\frac{\textrm{Sc}_{P}}{6} (|x|^{2} \delta_{\alpha \beta} - x_{\alpha} x_{\beta} ) 
	- \frac{1}{3} \delta_{\alpha \beta} \Ric_{P}(x,x) 
	- \frac{1}{3} |x|^{2} R_{\alpha \beta} 
	+ \frac{1}{3} (x_{\alpha} R_{\beta \mu} x^{\mu} 
	+ x_{\beta} R_{\alpha \mu} x^{\mu} ).
	\end{equation}

	We first collect some facts 
which will be useful in the proof of the computations below:

	\begin{lem}\label{l:basic-facts}
		The following hold:
		\begin{enumerate}
			\item[\rm (i)] 
			$\langle \nabla^{\R^{3}}_{e_{i}} e_{i}, e_{j} \rangle = 0$ 
			except for $(i,j)=(2,1)$ and 
			$\langle \nabla^{\R^{3}}_{e_{2}} e_{2}, e_{1} \rangle = - \cos \theta/ 
			(\tilde A \sin \theta)$. 
			\item[\rm (ii)] For $\mathcal{X}(\theta,\varphi)$, 
			\[
			\begin{aligned}
			\mathcal{X} &= \tilde{A}(n_0 + \ex), \quad 
			\la \mathcal{X}, n_0 \ra = \tilde A (1 + \la n_0, \ex \ra), \quad  
			\la \mathcal{X}, n_0 \ra^{2} = 
			\tilde{A}^{2} ( 1 + 2 \la n_0, \ex \ra + \la n_0, \ex \ra^{2})
			\\
			|\mathcal{X}|^{2} &= \tilde{A}^{2} \la n_0 + \ex, n_0 + \ex \ra 
			=  \tilde{A}^{2} (2 +  2 \la n_0, \ex \ra), 
			\end{aligned}
			\]
			\item[\rm (iii)] 
			\[
			\begin{aligned}
			0&= \int_{0}^{\pi} \cos^{3} \theta \sin \theta d \theta 
			= 	\int_{0}^{\pi} \sin \theta \cos \theta d \theta
			= \int_0^\pi \cos \theta \sin^3 \theta d \theta, 
			\\
			\frac{\pi}{2} 
			&= \int_{0}^{2\pi} \cos^{2} \varphi \cos 2 \varphi d \varphi 
			= - \int_{0}^{2\pi} \cos 2 \varphi \sin^{2} \varphi d \varphi,
			\\
			0 &= \int_{0}^{2\pi} \cos 2 \varphi \sin \varphi \cos \varphi d\varphi 
			= \int_{0}^{2\pi} \cos 2 \varphi \cos \varphi d \varphi 
			= \int_{0}^{2\pi} \cos 2 \varphi \sin \varphi d \varphi,
			\\
			\frac{4}{3} &= \int_0^\pi \sin^3 \theta d \theta,
			\quad
			\frac{2}{5} = \int_{0}^{\pi} \cos^{4} \theta \sin \theta d \theta, 
			\quad 
			\frac{4}{15} = \int_{0}^{\pi} \sin^{3} \theta \cos^{2} \theta d \theta,
			\quad 
			\frac{2}{3} = \int_{0}^{\pi} \sin \theta \cos^{2} \theta d \theta.
			\\ 
			\end{aligned}
			\]
		\end{enumerate}
	\end{lem}

	\begin{proof} 
		Noting that 
		\[
		\begin{aligned}
		\nabla_{e_{1}} e_{1} = \frac{1}{\tilde{A}} \nabla_{\partial_{\theta} \mathcal{X}} 
		e_{1} = \frac{1}{\tilde{A}} \partial_{\theta} e_{1} 
		= - \frac{n_0}{\tilde{A}}, \quad 
		\nabla_{e_{2}} e_{2} 
		= \frac{1}{\tilde{A} \sin \theta} \nabla_{\partial_{\varphi} \mathcal{X}} e_{2} 
		= \frac{1}{\tilde{A} \sin \theta} \partial_{\varphi} e_{2} 
		= \frac{1}{\tilde{A} \sin \theta} 
		\begin{pmatrix}
		0 \\ - \cos \varphi \\ -\sin \varphi,
		\end{pmatrix}
		\end{aligned}
		\]
		we have 
		\[
		\langle \nabla_{e_{1}} e_{1}, e_{j} \rangle = 0 
		= \langle \nabla_{e_{2}} e_{2}, e_{2} \rangle, \quad 
		\langle \nabla_{e_{2}} e_{2}, e_{1} \rangle 
		= - \frac{\cos \theta}{\tilde{A} \sin \theta}. 
		\]
		Thus (i) holds. (ii) and (iii) can be proven by direct computations.  	\end{proof}

Next we compute the second term in the left hand side of  \eqref{eq:PV}.

	\begin{lem}\label{l:H-Ric}
		The following holds:
		\[
			\int_{\SA} H_{\SA} \Ric_{P} ( n_0, n_0 ) ( 1 - \chi_{\delta} ) \psi_{0} 
			d \sigma 
			= \frac{4}{3} \pi \tilde{A} B ( R_{22} - R_{33} ) + O(\delta^{2}).
		\]
	\end{lem}

	\begin{proof}
Since $\psi_0$ is bounded, it is enough to show 
	\[
		\int_{\SA} H_{\SA} \Ric_{P} ( n_0, n_0 ) \psi_{0} d \sigma 
		= \frac{4}{3} \pi \tilde{A} B ( R_{22} - R_{33} ). 
	\]
First we expand ${\rm Ric}_{P}(n_0,n_0)$ as follows: 
	\[
		\begin{aligned}
			{\rm Ric}_{P} ( n_0, n_0 ) 
			=& ( R_{11} \cos^{2} \theta + R_{22} \sin^{2} \theta \cos^{2} \varphi 
			+ R_{33} \sin^{2} \theta \sin^{2} \varphi ) 
			\\
			& + ( 2 R_{12} \sin \theta \cos \theta \cos \varphi 
			+ 2 R_{23} \sin^{2} \theta \sin \varphi \cos \varphi 
			+ 2 R_{31} \sin \theta \cos \theta \sin \varphi  ).
		\end{aligned} 
	\]
From $H_{\SA} \equiv 2 / \tilde{A}$ and Lemma \ref{l:basic-facts} (iii), 
it follows that 
	\[
		\begin{aligned}
			& \int_{\SA} H_{\SA} {\rm Ric}_{P} ( n_0, n_0 ) \psi_{0} d \sigma 
			\\
			=& 2 \tilde{A} B \int_{0}^{\pi} \int_{0}^{2\pi} 
			( 1 - \cos \theta ) \cos 2 \varphi \sin \theta \Ric_{P} (n_0,n_0) 
			d \varphi d \theta 
			\\
			=& 
			2 \tilde{A} B \int_{0}^{\pi} \int_{0}^{2\pi} 
			( 1 - \cos \theta ) \sin \theta \cos 2 \varphi 
			( R_{22} \sin^{2} \theta \cos^{2} \varphi 
			+ R_{33} \sin^{2} \theta \sin^{2} \varphi ) d \varphi d \theta
			\\
			=& \pi \tilde{A} B ( R_{22} - R_{33}) 
			\int_0^\pi (1 - \cos \theta) \sin^3 \theta d \theta
			= \frac{4}{3} \pi \tilde{A} B (R_{22} - R_{33}),
 		\end{aligned}
	\]
which completes the proof.  
\end{proof}

		Next, we show 
	\begin{lem}\label{l:H-dot-psi0}
		There holds 
		\[
			\int_{S^{2}_{\tilde A}} ( 1 - \chi_{\delta}) 
			F \Delta_{S^{2}_{\tilde A}} \psi_{0} 
			d \sigma 
			= 4 \pi \tilde{A} B  (R_{22} - R_{33} ) + O(\delta^{2}).
		\] 
	\end{lem}

To show Lemma \ref{l:H-dot-psi0}, we first rewrite $F$ as follows:

	\begin{lem}
	One has 
	\begin{equation}\label{eq:H-dot-1}
		\begin{aligned}
			F 
			&= - \sum_{i=1}^{2} e_{i} (h_{ni}) 
			+ \tilde{A} \bigg\{  
			- \frac{\ScP}{6} (1 + \cos \theta)  
			- \frac{1}{3} \Ric_{P}(\mathbf{f}_{1}, e_{1} ) \cos \theta 
			- \frac{1}{3} ( 1 + \cos \theta ) \Ric_{P} ( n_0, \ex ) 
			\\
			& \qquad \qquad \qquad \qquad \quad 
			+ \frac{1}{3} \Ric_{P} ( n_0, n_0 ) \cos \theta  
			+ \frac{1}{3} \Ric_{P} ( \ex, \ex ) 
			 \bigg\}
		\end{aligned}
	\end{equation}
where 
	\[
		\mathbf{f}_{1} = \big( \sin \theta, \ -( 1 + \cos \theta) \cos \varphi, \ 
			- ( 1 + \cos \theta ) \sin \varphi \big).
	\]
	\end{lem}

\begin{proof}
First, we express $h_{nn}(\mathcal{X})$ in terms of $n_0$ and $\ex$. 
By \eqref{eq:expr-h}, notice that 
	\[
		h_{nn}(\mathcal{X}) = \frac{\textrm{Sc}_{P}}{6} 
		( |\mathcal{X}|^{2} - \la \mathcal{X}, n_0 \ra^{2} ) 
		- \frac{1}{3} \Ric_{P}(\mathcal{X},\mathcal{X}) 
		- \frac{1}{3} |\mathcal{X}|^{2} \Ric_{P}(n_0,n_0) 
		+ \frac{2}{3} \la \mathcal{X}, n_0 \ra \Ric_{P} ( \mathcal{X}, n_0).
	\]
Using Lemma \ref{l:basic-facts}, one gets 
	\begin{equation}\label{eq:h-nn}
		\begin{aligned}
			h_{nn}(\mathcal{X}) &= \tilde{A}^{2} 
			\bigg[ \frac{\textrm{Sc}_{P}}{6} ( 1 - \la n_0, \ex \ra^{2}) - 
			\frac{1}{3} \left\{ \Ric_{P} (n_0,n_0) + 2 \Ric_{P}(n_0,\ex) 
			+ \Ric_{P} ( \ex , \ex ) \right\} \\
			& \qquad \quad 
			- \frac{2}{3} ( 1 + \la n_0 , \ex \ra ) \Ric_{P}(n_0,n_0) 
			+ \frac{2}{3} ( 1 + \la n_0, \ex \ra ) 
			\left\{ \Ric_{P}(n_0,n_0) + \Ric_{P}(\ex, n_0) \right\} 
			\bigg]
			\\
			&= 
			\tilde{A}^{2} \left\{ \frac{\textrm{Sc}_{P}}{6}( 1 - \la n_0, \ex \ra^{2}) 
			- \frac{1}{3} \Ric_{P}(n_0,n_0) + \frac{2}{3} \la n_0, \ex \ra 
			\Ric_{P} ( n_0, \ex ) - \frac{1}{3} \Ric_{P}(\ex, \ex) 
			\right\}.
		\end{aligned}
	\end{equation}

		Next, we show 
	\begin{equation}\label{eq:del-h-del-nu}
		\sum_{i=1}^{2} \frac{\partial h}{\partial n_0} (e_{i},e_{i}) 
		= - \frac{2\tilde A}{3} \Ric_{P} ( n_0 + \ex, n_0). 
	\end{equation}
Since 
	\[
		\begin{aligned}
			\partial_{\eta} h_{\alpha \beta}(x) 
			&= \frac{\textrm{Sc}_{P}}{6} ( 2 x_{\eta} \delta_{\alpha \beta} 
			- \delta_{\alpha \eta} x_{\beta} - \delta_{\beta \eta} x_{\alpha} ) 
			- \frac{2}{3} \delta_{\alpha \beta} R_{\eta \mu} x^{\mu} 
			- \frac{2}{3} x_{\eta} R_{\alpha \beta} 
			\\
			&\quad 
			+ \frac{1}{3} 
			( \delta_{\alpha \eta} R_{\beta \mu} x^{\mu} + x_{\alpha} R_{\beta \eta} 
			+ \delta_{\beta \eta} R_{\alpha \mu} x^{\mu} 
			+ x_{\beta} R_{\alpha \eta} ) ,
		\end{aligned}
	\]
we observe that 
	\[
		\begin{aligned}
			\frac{\partial h_{\alpha \beta}}{\partial n_0} (X) 
			&= \la \nabla h_{\alpha \beta}(\mathcal{X}), n_0 \ra 
			\\
			&= \frac{\textrm{Sc}_{P}}{6} ( 2 \la \mathcal{X}, n_0 \ra \delta_{\alpha \beta} 
			- n_{0,\alpha} \mathcal{X}_{\beta} - n_{0,\beta} \mathcal{X}_{\alpha} ) 
			- \frac{2}{3} \delta_{\alpha \beta} \Ric_{P}(n_0,\mathcal{X}) 
			- \frac{2}{3} \la \mathcal{X}, n_0 \ra R_{\alpha \beta} 
			\\
			& \quad 
			+ \frac{1}{3} 
			( n_{0,\alpha} R_{\beta \mu } \mathcal{X}^{\mu} 
			+ \mathcal{X}_{\alpha } R_{\beta \eta} n_0^{\eta} 
			+ n_{0,\beta} R_{\alpha \mu} \mathcal{X}^{\mu} 
			+ \mathcal{X}_{\beta} R_{\alpha \eta} n_0^{\eta}).
		\end{aligned}
	\]
Thus, there holds 
	\[
		{\rm tr}_{\R^3}\, \left( \frac{\partial h}{\partial n_0} \right) 
		= \frac{2}{3} \textrm{Sc}_{P} \la \mathcal{X}, n_0 \ra 
		- 2 \Ric_{P}(n_0,\mathcal{X}) 
		- \frac{2}{3} \la \mathcal{X}, n_0 \ra \textrm{Sc}_{P} 
		+ \frac{4}{3} \Ric_{P}(\mathcal{X},n_0)
		= - \frac{2}{3} \Ric_{P} (\mathcal{X}, n_0).
	\]
We also note 
	\[
		\begin{aligned}
			\frac{\partial h}{\partial n_0}(n_0,n_0)  
			&= \frac{\textrm{Sc}_{P}}{6} ( 2 \la \mathcal{X}, n_0 \ra 
			- 2 \la \mathcal{X}, n_0 \ra ) 
			- \frac{2}{3} \Ric_{P} (\mathcal{X},n_0) 
			- \frac{2}{3} \la \mathcal{X}, n_0 \ra \Ric_{P} (n_0,n_0) 
			\\
			& \quad + \frac{2}{3} 
			\{ \Ric_{P}(n_0,\mathcal{X}) + \la \mathcal{X}, n_0 \ra \Ric_{P} (n_0,n_0)  \}
			\\
			&=0.
		\end{aligned}
	\]
Since $\{e_{1},e_{2},n_0\}$ forms an orthonormal basis of $\R^{3}$, 
we conclude 
	\[
		\sum_{i=1}^{2} \frac{\partial h}{\partial n_0}(e_{i},e_{i}) 
		= {\rm tr}_{\R^3}\, \left( \frac{\partial h}{\partial n_0} \right) 
		- \frac{\partial h}{\partial n_0} (n_0, n_0) 
		= - \frac{2}{3} \Ric_{P} (\mathcal{X},n_0) 
		= - \frac{2 \tilde{A} }{3} \Ric_{P} ( n_0 + \ex , n_0 )
	\]
and \eqref{eq:del-h-del-nu} holds.

		By \eqref{eq:expr-h} and Lemma \ref{l:basic-facts}, we also have 
	\begin{equation}\label{eq:h-ni}
		\begin{aligned}
			h_{ni}(\mathcal{X}) 
			&= \frac{\textrm{Sc}_P}{6} \left( - \la \mathcal{X}, n_0 \ra 
			\la \mathcal{X}, e_i \ra \right)
			- \frac{1}{3} |\mathcal{X}|^2 \Ric_P(n_0,e_i) 
			+ \frac{1}{3} 
			\left\{ \la \mathcal{X}, n_0 \ra \Ric_P(\mathcal{X},e_i) 
			+ \la \mathcal{X}, e_i \ra \Ric(\mathcal{X}, n_0) \right\}
			\\
			&= \tilde{A}^2 
			\bigg[ - \frac{\textrm{Sc}_P}{6} \left( 1 + \la n_0, \ex \ra \right) \la \ex, e_i \ra
			- \frac{2}{3} \left( 1 + \la n_0, \ex \ra \right) \Ric_P(n_0,e_i) 
			\\
			&\qquad\qquad  + \frac{1}{3} \left( 1 + \la n_0, \ex \ra \right) \Ric_P(n_0 + \ex, e_i) 
			+ \frac{1}{3} \la \ex , e_i \ra \Ric_P (n_0 + \ex, n_0) \bigg]
			\\
			&= \tilde{A}^2 
			\left[ -\frac{\textrm{Sc}_P}{6} \left( 1 + \la n_0, \ex \ra \right) \la \ex, e_i \ra
			 + \frac{1}{3} \left( 1 + \la n_0, \ex \ra \right) \Ric_P (\ex - n_0, e_i) 
			 + \frac{1}{3} \la \ex, e_i \ra \Ric_{P} (n_0 + \ex, n_0) 
			\right]. 
		\end{aligned}
	\end{equation}
Now using Lemma \ref{l:basic-facts}, \eqref{eq:h-nn} and 
\eqref{eq:del-h-del-nu}, we have 
	\begin{equation}\label{eq:cal-H-dot}
		\begin{aligned}
			F &= - \sum_{i=1}^{2} e_{i}(h_{ni}) + \sum_{i,j=1}^{2} h_{nj} 
			\la \nabla^{\R^{3}}_{e_{i}} e_{i} , e_{j} \ra 
			- \frac{1}{2} h_{nn} H_{\SA} 
			+ \frac{1}{2} \sum_{i=1}^{2} 
			\frac{\partial h}{\partial n_0} ( e_{i} , e_{i} )
			\\
			&= - \sum_{i=1}^{2} e_{i} (h_{ni}) 
			+ h_{n1} \la \nabla^{\R^{3}}_{e_{2}} e_{2} , e_{1} \ra 
			- \frac{1}{2} h_{nn} H_{\SA} 
			- \frac{\tilde{A}}{3} \Ric_{P} ( n_0 + \ex, n_0 )
			\\
			&= - \sum_{i=1}^{2} e_{i} ( h_{ni} ) 
			- \tilde{A} \frac{\cos \theta}{\sin \theta} 
			\bigg\{ - \frac{\textrm{Sc}_{P}}{6} ( 1 + \la n_0, \ex \ra) \la \ex, e_{1} \ra 
			+ \frac{1}{3} ( 1 + \la n_0, \ex \ra ) \Ric_{P}(\ex - n_0, e_{1}) 
			\\
			& \quad 
			+ \frac{1}{3} \la \ex , e_{1} \ra \Ric_{P} (n_0 + \ex, n_0 ) \bigg\} 
			\\
			& \quad 
			- \tilde A \left\{ \frac{\textrm{Sc}_{P}}{6}( 1 - \la n_0, \ex \ra^{2}) 
			- \frac{1}{3} \Ric_{P}(n_0,n_0) + \frac{2}{3} \la n_0, \ex \ra 
			\Ric_{P} ( n_0, \ex ) - \frac{1}{3} \Ric_{P}(\ex, \ex) 
			\right\}
			\\
			&
			\quad - \frac{\tilde{A}}{3} \Ric_{P}( n_0 + \ex, n_0 ) 
			\\
			&= - \sum_{i=1}^{2} e_{i} ( h_{ni} ) 
			+ \tilde{A} \frac{\cos \theta}{\sin \theta} 
			\bigg\{  \frac{\textrm{Sc}_{P}}{6} ( 1 + \la n_0, \ex \ra) \la \ex, e_{1} \ra 
			- \frac{1}{3} ( 1 + \la n_0, \ex \ra ) \Ric_{P}(\ex - n_0, e_{1}) 
			\\
			& \quad 
			- \frac{1}{3} \la \ex , e_{1} \ra \Ric_{P} ( n_0 + \ex, n_0 ) \bigg\} 
			\\
			& \quad - 
			\tilde A \left\{ \frac{\textrm{Sc}_{P}}{6}( 1 - \la n_0, \ex \ra^{2}) 
			+ \frac{1}{3} ( 1 + 2 \la n_0, \ex \ra )
			\Ric_{P} ( n_0, \ex ) - \frac{1}{3} \Ric_{P}(\ex, \ex) 
			\right\}.
		\end{aligned}
	\end{equation}
Noting that 
	\[
		\la n_0, \ex \ra = \cos \theta, \quad 
		\la \ex, e_{1} \ra = - \sin \theta,
	\]
one obtains 
	\[
		\begin{aligned}
			& \frac{\cos \theta }{\sin \theta} \frac{\textrm{Sc}_{P}}{6} 
			( 1 + \la n_0, \ex \ra ) \la \ex, e_{1} \ra 
			- \frac{\textrm{Sc}_{P}}{6} ( 1 - \la n_0, \ex \ra^2 ) 
			\\
			=&\frac{\textrm{Sc}_{P}}{6} ( 1 + \la n_0, \ex \ra ) 
			\left\{ \frac{\cos \theta}{\sin \theta} \la \ex, e_{1} \ra 
			- 1 + \la n_0, \ex \ra
			\right\} 
			= - \frac{\textrm{Sc}_{P}}{6} ( 1 + \cos \theta ).
		\end{aligned}
	\]
Similarly, since 
	\[
		\begin{aligned}
			(1+ \la n_0, \ex \ra ) ( \ex - n_0 ) 
			&=
			 \big( 1 - \cos^{2} \theta, \ - ( 1 + \cos \theta) \sin \theta \cos \varphi, \ 
			- ( 1 + \cos \theta) \sin \theta \sin \varphi \big)
			\\
			&= \sin \theta \big( \sin \theta, \ -( 1 + \cos \theta) \cos \varphi, \ 
			- ( 1 + \cos \theta ) \sin \varphi \big) \\
			&=: \mathbf{f}_{1} \sin \theta
		\end{aligned}
	\]
where 
	\[
		\mathbf{f}_{1} = \big( \sin \theta, \ -( 1 + \cos \theta) \cos \varphi, \ 
			- ( 1 + \cos \theta ) \sin \varphi \big),
	\]
we have 
	\[
		-\frac{1}{3} \frac{\cos \theta}{\sin \theta} ( 1 + \la n_0, \ex \ra ) 
		\Ric_{P} ( \ex - n_0, e_{1} ) 
		= - \frac{1}{3}  \Ric_{P} ( \mathbf{f}_{1}, e_{1} ) \cos \theta.
	\]
Finally, from 
	\[
		-\frac{1}{3}\frac{\cos \theta}{\sin \theta} \la \ex, e_{1} \ra 
		\Ric_{P} ( n_0 + \ex , n_0 ) 
		=  \frac{1}{3}  \Ric_{P}( n_0 + \ex , n_0 ) \cos \theta, 
	\]
it follows that 
	\[
		\begin{aligned}
			& - \frac{1}{3} \frac{\cos \theta}{\sin \theta}  \la \ex, e_{1} \ra 
			\Ric_{P} ( n_0 + \ex , n_0 ) 
			-
			\frac{1}{3} ( 1 + 2 \la n_0, \ex \ra ) \Ric_{P} ( n_0, \ex ) 
			\\
			= & \frac{1}{3} \Ric_P (n_0,n_0) \cos \theta 
			- \frac{1}{3} ( 1 + \cos \theta ) \Ric_{P} (n_0,\ex) . 
		\end{aligned}
	\]
Substituting these formulas into \eqref{eq:cal-H-dot}, 
we have \eqref{eq:H-dot-1}. 
\end{proof}

Now we complete the proof of Lemma \ref{l:H-dot-psi0}.

\begin{proof}[Proof of Lemma \ref{l:H-dot-psi0}]
Due to \eqref{eq:H-dot-1}, 
we first compute the following quantities one by one:
	\[
		\begin{aligned}
			\int_{S^{2}_{\tilde A}} ( 1 - \chi_{\delta} ) 
			(\Delta_{S^{2}_{\tilde A}} \psi_{0} ) 
			\tilde{A} \bigg\{ 
			&
			-\frac{\ScP}{6} ( 1 + \cos \theta ) 
			- \frac{1}{3} \Ric_{P} ( \mathbf{f}_{1}, e_{1} ) \cos \theta 
			- \frac{1}{3} ( 1 + \cos \theta ) \Ric_{P} ( n_0, \ex )
			\\
			& + \frac{1}{3} \Ric_{P}(n_0,n_0) \cos \theta 
			 +\frac{1}{3} \Ric_{P} ( \ex, \ex ) 
			\bigg\} 
			d \sigma. 
		\end{aligned}
	\]
Here we recall that $\chi_{\delta}(\mathcal{X}(\theta,\varphi))$ does not depend on 
$\varphi$ and we use this fact repeatedly in the following computations. 
We also recall that 
	\[
		\begin{aligned}
			\Delta_{\SA} \psi_0 
			&= \frac{1}{\tilde{A}^2} \left[ - 2 A \cos \theta 
			+ 2 B \cos 2 \varphi 
			\left\{ \cos \theta - \frac{2 (1-\cos\theta)}{\sin^2 \theta} \right\} \right].
		\end{aligned}
	\]

\bigskip

\noindent
$\bullet$ $ \displaystyle \int_{S^{2}_{\tilde A}}  \frac{1}{3} 
( 1 - \chi_{\delta} ) \Ric_{P}( \ex, \ex ) \Delta_{S^{2}_{\tilde A}} \psi_{0}  
d \sigma = O(\delta^{2})$.

\bigskip

		Since $\Ric_{P} ( \ex, \ex ) = R_{11}$, 
it follows from Lemma \ref{l:basic-facts} that 
	\[
		\begin{aligned}
			&\int_{S^{2}_{\tilde A}} \frac{1}{3} (1 - \chi_{\delta}) \Ric_{P} ( \ex, \ex ) 
			\Delta_{S^{2}_{\tilde A}} \psi_{0} d \sigma 
			\\
			= &
			\frac{R_{11}}{3} \int_{0}^{\pi} ( 1 - \chi_{\delta}) \int_{0}^{2 \pi} 
			 \sin \theta 
			\left[ - 2 A  \cos \theta + 2 B \cos 2 \varphi 
			\left\{ \cos \theta 
			- \frac{2 ( 1 - \cos \theta)}{\sin^{2} \theta} 
			\right\}
			\right] 
			d \varphi d \theta 
			\\
			=& 
			-\frac{4 \pi R_{11}A}{3}  \int_{0}^{\pi} 
			( 1 - \chi_{\delta} ) \cos \theta \sin \theta 
			d \theta =
			\frac{4 \pi R_{11} A}{3} \int_0^\pi \chi_\delta \cos \theta 
			\sin \theta d \theta =  O(\delta^{2}).
		\end{aligned}
	\]

\bigskip

\noindent
$\bullet$ 
$\displaystyle 
\int_{S^{2}_{\tilde A}} \frac{1}{3} (1 - \chi_{\delta}) 
\Ric_{P}(n_0,n_0) \cos \theta 
\Delta_{S^{2}_{\tilde A}} \psi_{0} d \sigma 
= - \frac{8}{15} \pi A R_{11} - \frac{8}{45} \pi A (R_{22} + R_{33})
+ \frac{8}{15} \pi B (R_{22} - R_{33} ) + O(\delta^{2}).
$

\bigskip

		We first note that 
	\[
		\begin{aligned}
			\Ric_{P}(n_0,n_0) &= (R_{11} \cos^{2} \theta 
			+ R_{22} \sin^{2} \theta \cos^{2} \varphi
			+ R_{33} \sin^{2} \theta \sin^{2} \varphi ) 
			\\
			& \quad + 
			(2 R_{12} \cos \theta \sin \theta \cos \varphi 
			+ 2R_{23} \sin^{2} \theta \cos \varphi \sin \varphi 
			+ 2 R_{13} \cos \theta \sin \theta \sin \varphi )
			\\
			&=: R_{I} + R_{II}.
		\end{aligned}
	\]
Using Lemma \ref{l:basic-facts}, we get 
	\[
		\begin{aligned}
			&\int_{S^{2}_{\tilde A}} ( 1 - \chi_{\delta}) \Ric_{P}(n_0,n_0) \cos \theta 
			 \Delta_{S^{2}_{\tilde A}} \psi_0 
			d \sigma 
			\\
			=&
			\int_{0}^{\pi} (1 - \chi_{\delta}) \int_{0}^{2\pi} 
			R_{I} \sin \theta \cos \theta 
			\left[ -2 A \cos \theta 
			+ 2 B \cos 2 \varphi \left\{ \cos \theta 
			- \frac{2(1-\cos \theta)}{\sin^{2} \theta}
			\right\} \right]
			d \varphi d \theta 
			\\
			=& 
			- 2 A \int_{0}^{\pi} (1 - \chi_{\delta}) \sin \theta \cos^{2} \theta 
			\left\{ 2 \pi R_{11} \cos^{2} \theta  
			+ \pi (R_{22} + R_{33}) \sin^{2} \theta \right\} d \theta 
			\\
			&+ 2 B \int_{0}^{\pi} 
			( 1 - \chi_{\delta} ) \sin \theta \cos \theta 
			\left\{ \cos \theta - \frac{2(1-\cos \theta)}{\sin^{2} \theta}  \right\} 
			\frac{ \pi } {2} (R_{22} - R_{33} ) \sin^{2} \theta 
			d \theta
			\\
			=& - \frac{8}{5} \pi A R_{11} - \frac{8}{15} \pi A (R_{22} + R_{33}) 
			+ \pi B ( R_{22} - R_{33} )  \left\{ \frac{4}{15} + \frac{4}{3} \right\} 
			+ O(\delta^{2}).
		\end{aligned}
	\]
Hence, 
	\[
		\begin{aligned}
			\int_{S^{2}_{\tilde A}} \frac{1}{3} ( 1 - \chi_{\delta}) 
			\Ric_{P}(n_0,n_0) \cos \theta 
			\Delta_{S^{2}_{\tilde A}} \psi d \sigma 
			= - \frac{8}{15} \pi A R_{11} - \frac{8}{45} \pi A (R_{22} + R_{33}) 
			+ \frac{8}{15} \pi B ( R_{22} - R_{33} ) + O(\delta^{2}).
		\end{aligned}
	\]

\bigskip

\noindent
$\bullet$ 
$\displaystyle 
\int_{S^{2}_{\tilde A}} - \frac{1}{3} ( 1 - \chi_{\delta}) ( 1 + \cos \theta ) 
\Ric_{P} ( n_0, \ex ) \Delta_{S^{2}_{\tilde A}} \psi_{0} d \sigma 
=  \frac{8}{9} \pi A R_{11} + O(\delta^{2}).
$

\bigskip

		From 
	\[
		\Ric_{P} ( n_0, \ex ) 
		= R_{11} \cos \theta + R_{12} \sin \theta \cos \varphi 
		+ R_{13} \sin \theta \sin \varphi, 
	\]
it follows that 
	\[
		\begin{aligned}
			&\int_{S^{2}_{\tilde A}} (1 - \chi_{\delta}) ( 1 + \cos \theta ) 
			\Ric_{P} ( n_0, \ex ) \Delta_{S^{2}_{\tilde A}} \psi_{0} d \sigma 
			\\
			=& \int_{0}^{\pi} ( 1 - \chi_{\delta}) \int_{0}^{2\pi} 
			( 1 + \cos \theta ) R_{11} \cos \theta  \sin \theta 
			( - 2 A \cos \theta ) d \varphi d \theta 
			\\
			=& - 4 \pi A R_{11}\int_{0}^{\pi} (1 - \chi_{\delta}) 
			\cos^{2} \theta \sin \theta (1+\cos \theta) d \theta  + O(\delta^2)
			= - \frac{8}{3} \pi A R_{11} + O(\delta^{2}). 
		\end{aligned}
	\]

\bigskip

\noindent
$\bullet$ 
$\displaystyle 
	\int_{S^{2}_{\tilde A}} 
	- \frac{1}{3} ( 1 - \chi_{\delta} ) 
	\Ric_{P} ( \mathbf{f}_{1} , e_{1} ) 
	\cos \theta \Delta_{S^{2}_{\tilde A}} \psi_{0} d \sigma 
	= -\frac{16}{45} \pi A R_{11} - \frac{4}{15} \pi A (R_{22} + R_{33}) 
	- \frac{14}{45} \pi B (R_{22} - R_{33} ) + O(\delta^{2}).
$

\bigskip

		From 
$\mathbf{f}_{1} = ( \sin \theta, -( 1 + \cos \theta ) \cos \varphi, 
- ( 1 + \cos \theta ) \sin \varphi )$ and 
	\begin{equation}\label{eq:ric-fe}
		\begin{aligned}
			\Ric_{P} ( \mathbf{f}_{1} , e_{1} ) 
			&= \Ric_{P} \left( 
				\begin{pmatrix} 
					\sin \theta \\ -(1+\cos \theta) \cos \varphi \\
					- ( 1 + \cos \theta) \sin \varphi
			 	\end{pmatrix}
				,
				\begin{pmatrix}
					- \sin \theta \\ \cos \theta \cos \varphi \\ 
					\cos \theta \sin \varphi
				\end{pmatrix}
			\right)
			\\
			&= 
			\left\{ - R_{11} \sin^{2} \theta - R_{22} ( 1 + \cos \theta ) \cos \theta 
			\cos^{2} \varphi - R_{33} ( 1 + \cos \theta ) \cos \theta \sin^{2} \varphi 
			\right\} 
			\\
			& \quad + 
			\left\{(R_{12} \cos \varphi + R_{13} \sin \varphi )
			 ( 1 + 2 \cos \theta ) \sin \theta 
			- 2 R_{23} ( 1 + \cos \theta ) \cos \theta \sin \varphi \cos \varphi 
			\right\}, 
		\end{aligned}
	\end{equation}
one observes that 
	\[
		\begin{aligned}
			&\int_{S^{2}_{\tilde A}} (1 - \chi_{\delta}) 
			\Ric_{P} ( \mathbf{f}_{1} , e_{1} ) \cos \theta 
			\Delta_{S^{2}_{\tilde A}} \psi_{0} d \sigma 
			\\
			=& \int_{0}^{\pi} (1 - \chi_{\delta}) \int_{0}^{2\pi} 
			\left\{ - R_{11} \sin^{2} \theta - R_{22} ( 1 + \cos \theta ) \cos \theta 
			\cos^{2} \varphi - R_{33} ( 1 + \cos \theta ) \cos \theta \sin^{2} \varphi 
			\right\} 
			\\
			& \qquad \times 
			\cos \theta \sin \theta 
			\left[ - 2 A \cos \theta 
			+ 2 B \cos 2 \varphi \left\{ \cos \theta - \frac{2(1-\cos \theta)}
			{\sin^{2} \theta} \right\} \right] d \varphi d \theta 
			\\
			=& 2\pi A \int_{0}^{\pi} (1 - \chi_{\delta}) \left\{  2 R_{11} \sin^{2} \theta 
			+ (R_{22} + R_{33}) ( 1 + \cos \theta) \cos \theta
			\right\} \cos^{2} \theta \sin \theta 
			d \theta 
			\\
			&+ 2B \int_{0}^{\pi} (1 - \chi_{\delta}) \frac{\pi}{2} ( - R_{22} + R_{33} ) 
			 ( 1 + \cos \theta ) \cos^{2} \theta \sin \theta 
			 \left\{ \cos \theta - \frac{2(1-\cos \theta)}
			{\sin^{2} \theta} \right\} d \theta 
			\\
			=& 2 \pi A \left\{ \frac{8}{15} R_{11} 
			+ \frac{2}{5} (R_{22} + R_{33})  \right\} 
			\\
			&+ \pi B ( -R_{22} + R_{33}) \int_{0}^{\pi}  (1 - \chi_{\delta})
			\cos^{2} \theta \sin \theta 
			\left\{ ( 1 + \cos \theta ) \cos \theta - 2 \right\} d \theta 
			+ O(\delta^{2})
			\\
			=& \frac{16}{15} \pi A R_{11} 
			+ \frac{4}{5} \pi A (R_{22} + R_{33} ) 
			 + \pi B (- R_{22} + R_{33} ) 
			\left( \frac{2}{5} - \frac{4}{3} \right) + O(\delta^{2})
			\\
			=& \frac{16}{15} \pi A R_{11} + \frac{4}{5} \pi A (R_{22} + R_{33} ) 
			+ \frac{14}{15}  \pi B ( R_{22} - R_{33}) + O(\delta^{2}).
		\end{aligned}
	\]

\bigskip

\noindent
$\bullet$ 
$\displaystyle 
\int_{S^{2}_{\tilde A}} - (1 - \chi_{\delta}) \frac{\ScP}{6} ( 1 + \cos \theta ) 
\Delta_{S^{2}_{\tilde A}} \psi_{0} d \sigma 
=  \frac{4}{9} \pi A \ScP + O(\delta^{2}).
$

\bigskip

By 
	\[
		\tilde{A}^{2} ( 1 + \cos \theta ) \Delta_{S^{2}_{\tilde A}} \psi_0
		= \left[ - 2 A ( 1 + \cos \theta ) \cos \theta 
		+ 2 B \cos 2 \varphi \left\{ ( 1 + \cos \theta ) \cos \theta 
		- 2  \right\} \right],
	\]
we obtain 
	\[
		\begin{aligned}
			&\int_{S^{2}_{\tilde A}} (1 - \chi_{\delta})
			( 1 + \cos \theta ) \Delta_{S^{2}_{\tilde A}} \psi_0
			d \sigma 
			\\
			=& 
			\int_{0}^{\pi} (1 - \chi_{\delta}) \int_{0}^{2\pi} 
			\left[ - 2 A ( 1 + \cos \theta ) \cos \theta 
			+ 2 B \cos 2 \varphi \left\{ ( 1 + \cos \theta ) \cos \theta 
			- 2  \right\} \right] \sin \theta d \varphi d \theta 
			\\
			=& 
			- 4 \pi A \int_{0}^{\pi} (1 - \chi_{\delta}) ( 1 + \cos \theta ) 
			\cos \theta \sin \theta d \theta
			= - \frac{8}{3} \pi A + O(\delta^{2}).  
		\end{aligned}
	\]

\ 

\

Collecting the above results, we have 
	\begin{equation}\label{eq:H-dot-2}
		\begin{aligned}
			&\int_{S^{2}_{\tilde A}} ( 1 - \chi_{\delta} ) 
			(\Delta_{S^{2}_{\tilde A}} \psi_{0} ) 
			\tilde{A} \bigg\{ 
				- \frac{\ScP}{6} ( 1 + \cos \theta ) 
				- \frac{1}{3} \Ric_{P} ( \mathbf{f}_{1}, e_{1} ) \cos \theta 
				- \frac{1}{3} ( 1 + \cos \theta ) \Ric_{P} ( n_0, \ex )
			\\
			& \hspace{10mm}
			+ \frac{1}{3} \Ric_{P}(n_0,n_0) \cos \theta 
			 + \frac{1}{3} \Ric_{P} ( \ex, \ex ) 
			\bigg\} 
			d \sigma
			\\
			=& \tilde{A}\left\{  \frac{4}{9} \pi A \ScP
			- \frac{4}{9} \pi A ( R_{22} + R_{33} ) 
			+ \frac{2}{9} \pi B ( R_{22} - R_{33} ) 
			\right\} + O(\delta^{2}).
		\end{aligned}
	\end{equation}

\ 

\

		Next, we compute 
$-\sum_{i=1}^2  \int_{\SA} e_{i} (h_{ni}) (1 - \chi_{\delta}) \Delta_{\SA} \psi_{0} d \sigma$. 
We recall the following expressions of $h_{ni}$ from \eqref{eq:h-ni}: 
	\[
		h_{ni}(X) = 
		\tilde{A}^{2} \left\{ 
		- \frac{\ScP}{6} ( 1 + \la n_0, \ex \ra) \la \ex, e_{i} \ra 
		+ \frac{1}{3} ( 1 + \la n_0, \ex \ra ) \Ric_{P}(\ex - n_0, e_{i}) 
		+ \frac{1}{3} \la \ex , e_{i} \ra \Ric_{P} (n_0 + \ex, n_0 )
		\right\}.
	\]

\bigskip

\noindent
$\bullet$ 
$\displaystyle 
\int_{S^{2}_{\tilde A}} - (1 - \chi_{\delta} ) 
e_{2} ( h_{n2} ) \Delta_{S^{2}_{\tilde A}} \psi_{0} d \sigma 
=  \frac{20}{9} \pi \tilde{A} B ( R_{22} - R_{33} ) + O(\delta^{2}). 
$

\

		By $\la \ex, e_{2} \ra = 0$ and 
$(1+\la n_0, \ex \ra ) (\ex - n_0 ) = \mathbf{f}_{1} \sin \theta$, we have 
	\[
		h_{n2} = \frac{\tilde{A}^{2}}{3} \Ric_{P} ( \mathbf{f}_{1} , e_{2} ) \sin \theta.
	\]
By simple calculations, one may see 
	\[
		\begin{aligned}
			\Ric_{P} ( \mathbf{f}_{1}, e_{2} ) 
			= &\frac{1}{2} ( R_{22} - R_{33} ) ( 1 + \cos \theta ) \sin 2 \varphi 
			- R_{23} ( 1 + \cos \theta ) \cos 2 \varphi 
			- R_{12} \sin \theta \sin \varphi + R_{13} \sin \theta 
			\cos \varphi
			,
			\\
			\frac{ \partial }{\partial \varphi } 
			\Ric_{P} ( \mathbf{f}_{1} , e_{2} ) 
			=& 
			( R_{22} - R_{33} ) ( 1 + \cos \theta ) \cos 2 \varphi 
			+ 2 R_{23} ( 1 + \cos \theta ) \sin 2 \varphi
			- R_{12} \sin \theta \cos \varphi - R_{13} \sin \theta \sin \varphi.
		\end{aligned}
	\]
Since $e_{2}(f)=( \tilde{A} \sin \theta)^{-1} \partial_{\varphi} f$, 
we get 
	\[
		\begin{aligned}
			&\int_{S^{2}_{\tilde A}} (1 - \chi_{\delta} )
			 e_{2} ( h_{n2} ) \Delta_{S^{2}_{\tilde A}} \psi_{0} 
			d \sigma 
			\\
			=& \int_{0}^{\pi} (1 - \chi_{\delta} ) \int_{0}^{2\pi} 
			\frac{\tilde{A}^{2}}{3} \frac{1}{\tilde{A} \sin \theta} 
			\frac{\partial}{\partial \varphi } 
			\left( \Ric_{P} ( \mathbf{f}_{1} , e_{2} ) \sin \theta \right) 
			\\
			& \qquad \quad \times
			\left[ - 2 A \cos \theta + 2 B \cos 2 \varphi 
			\left\{ \cos \theta - \frac{2(1-\cos \theta)}{ \sin^{2} \theta } 
			\right\} \right] \sin \theta d \varphi d \theta 
			\\
			=& \frac{2}{3}  \tilde A B \int_{0}^{\pi} (1 - \chi_{\delta} ) \int_{0}^{2\pi} 
			(R_{22} - R_{33} ) ( 1 + \cos \theta ) \sin \theta 
			\cos^{2} 2 \varphi \left\{ \cos \theta - 
			\frac{2(1-\cos \theta)}{ \sin^{2} \theta } \right\} d \varphi d \theta
			\\
			=& \frac{2}{3} \pi \tilde A B (R_{22} - R_{33} ) 
			\int_{0}^{\pi} (1 - \chi_{\delta} ) \sin \theta 
			\left\{ \cos \theta ( 1 + \cos \theta) - 2  \right\} d \theta 
			\\
			=& \frac{2}{3} \pi \tilde A B (R_{22} - R_{33} ) 
			\left( \frac{2}{3} - 4 \right) + O(\delta^{2})
			= - \frac{20}{9} \pi \tilde A B (R_{22} - R_{33} ) 
			+ O(\delta^{2}). 
		\end{aligned}
	\]

\bigskip

\noindent
$\bullet$ 
$\displaystyle 
\int_{S^{2}_{\tilde A}} - ( 1 - \chi_{\delta} ) e_{1} ( h_{n1} ) 
\Delta_{S^{2}_{\tilde A}} \psi_{0} d \sigma 
= \tilde{A} \left\{ - \frac{4}{9} \pi A R_{11} 
+ \frac{14}{9} \pi B ( R_{22} - R_{33} ) 
\right\} + O(\delta^{2}).
$

\bigskip

		First, we notice that 
	\[
		h_{n1} = 
		\tilde{A}^{2} \left\{ 
		\frac{\ScP}{6} ( 1 + \cos \theta ) \sin \theta 
		+ \frac{1}{3} \Ric_{P} ( \mathbf{f}_{1}, e_{1} ) \sin \theta 
		- \frac{1}{3} \Ric_{P} ( n_0 + \ex, n_0 ) \sin \theta
		\right\}
	\]
and that by \eqref{eq:ric-fe}, 
	\[
		\begin{aligned}
			\Ric_{P}(\mathbf{f}_{1}, e_{1} ) \sin \theta 
			&= - \left\{ R_{11} \sin^{2} \theta 
			+ ( R_{22} \cos^{2} \varphi + R_{33} \sin^{2} \varphi ) 
			(1+\cos \theta ) \cos \theta
			\right\} \sin \theta
			\\
			&\qquad + R_{1}(\theta,\varphi) =: R_{\mathbf{f}_{1}e_{1}} 
			+ R_{1} ( \theta, \varphi ),
			\\
			\Ric_{P}(n_0 + \ex, n_0 ) \sin \theta 
			&= \left\{ R_{11} \cos \theta ( 1 + \cos \theta )
			+ R_{22} \sin^{2} \theta \cos^{2} \varphi 
			+ R_{33} \sin^{2} \theta \sin^{2} \varphi  
			\right\} \sin \theta 
			\\& \qquad + R_{2}(\theta, \varphi ) 
			=: R_{n \ex}(\theta,\varphi ) + R_{2} ( \theta, \varphi ) 
		\end{aligned}
	\]
where 
	\[
		0=\int_{0}^{2\pi} \partial_\theta R_{1} ( \theta, \varphi ) d \varphi 
		= \int_{0}^{2\pi} \partial_\theta R_{1} ( \theta, \varphi ) \cos 2 \varphi d \varphi 
		= \int_{0}^{2\pi} \partial_\theta R_{2} ( \theta, \varphi ) d \varphi 
		= \int_{0}^{2\pi} \partial_\theta R_{2} ( \theta, \varphi ) \cos 2 \varphi d \varphi.
	\]
We also remark that in a neighbourhood of $\theta = -\pi$, one sees that 
	\begin{equation}\label{eq:est-R-pi}
		\begin{aligned}
			\left|\frac{\ScP}{6} ( 1 + \cos \theta ) \sin \theta \right| 
			+ | R_{\mathbf{f}_{1} e_{1}} ( \theta, \varphi ) | 
			+ | R_{n \ex} ( \theta, \varphi ) | 
			\leq C_{0} \sin^{3} \theta.
		\end{aligned}
	\end{equation}
Thus since $e_{1}(f) = \tilde{A}^{-1} \partial_{\theta} f$ and 
$| \partial_{\theta} ( \chi_{\delta}(\mathcal{X}) ) | \leq C_{0} \delta^{-1}$ by 
\eqref{eq:deri-chi-d-2}, 
it follows from \eqref{eq:est-R-pi} and integration by parts in $\theta$ that 
	\[
		\begin{aligned}
			&\int_{S^{2}_{\tilde A}} 
			( 1 - \chi_{\delta} ) e_{1} ( h_{n1} ) \Delta_{S^{2}_{\tilde A}} \psi_{0} 
			d \sigma 
			\\
			=& \tilde{A} 
			\int_{0}^{\pi} \int_{0}^{2\pi} 
			( 1 - \chi_{\delta}) 
			\partial_{\theta} \left\{ \frac{\ScP}{6} ( 1 + \cos \theta ) 
			\sin \theta + \frac{1}{3} \left( R_{\mathbf{f}_{1} e_{1}}  + R_{1} \right) 
			- \frac{1}{3}  \left( R_{n \ex} + R_{2}  \right) \right\}
			\\
			& \qquad \qquad 
			\times \left[ -2 A \cos \theta 
			+ 2 B \cos 2 \varphi \left\{ \cos \theta 
			- \frac{2(1-\cos \theta)}{\sin^{2} \theta} \right\} \right] 
			\sin \theta d \varphi d \theta 
			\\
			=& \tilde{A} 
			\int_{0}^{\pi} \int_{0}^{2\pi} 
			( 1 - \chi_{\delta}) 
			\partial_{\theta} \left\{ \frac{\ScP}{6} ( 1 + \cos \theta ) 
			\sin \theta + \frac{1}{3}  R_{\mathbf{f}_{1} e_{1}}  
			- \frac{1}{3}  R_{n \ex}  \right\}
			\\
			& \qquad \qquad 
			\times \left[ - A \sin 2 \theta  
			+ 2 B \cos 2 \varphi \left\{ \frac{\sin 2 \theta}{2}  
			- \frac{2(1-\cos \theta)}{\sin \theta} \right\} \right]  d \varphi d \theta
			\\
			=&  - \tilde A \int_{0}^{\pi} \int_{0}^{2\pi} 
			( 1 - \chi_\delta) 
			\left\{ \frac{\ScP}{6} ( 1 + \cos \theta ) 
			\sin \theta + \frac{1}{3}  R_{\mathbf{f}_{1} e_{1}}  
			- \frac{1}{3}  R_{n \ex}  \right\} 
			\\
			& \qquad \qquad 
			\times \left[ 
				- 2 A \cos 2 \theta  + 2 B \cos 2 \varphi 
				\left\{ \cos 2 \theta - 2 
				+ 2 \frac{(1-\cos \theta) \cos \theta}{\sin^{2} \theta} 
				\right\} 
			\right] d \varphi d \theta + O (\delta^{2}). 
		\end{aligned}
	\]

		Next, from 
	\[
		\begin{aligned}
			\int_{0}^{2 \pi} R_{\mathbf{f}_{1} e_{1}}  d \varphi 
			&= - \pi \left\{ 2 R_{11} \sin^{2} \theta 
			+ (R_{22} + R_{33}) ( 1 + \cos \theta ) \cos \theta \right\} 
			\sin \theta,
			\\
			\int_{0}^{2\pi} R_{n \ex} d \varphi 
			&= \pi \left\{ 2 R_{11} ( 1 + \cos \theta ) \cos \theta 
			+ (R_{22} + R_{33}) \sin^{2} \theta  \right\} \sin \theta,
			\\
			\int_{0}^{2\pi} R_{\mathbf{f}_{1} e_{1}} \cos 2 \varphi d \varphi 
			&= - \frac{\pi}{2} ( R_{22} - R_{33} ) ( 1 + \cos \theta ) \cos \theta 
			\sin \theta, 
			\\
			\int_{0}^{2\pi} R_{n \ex} \cos 2 \varphi d \varphi 
			&= \frac{\pi}{2} ( R_{22} - R_{33} ) \sin^{3} \theta, 
		\end{aligned}
 	\]
it follows that 
	\[
		\begin{aligned}
			\int_{0}^{2\pi} (R_{\mathbf{f}_{1} e_{1}} - R_{n \ex} ) d \varphi 
			&= - \pi \left\{ 
				2 R_{11} ( 1 + \cos \theta ) + (R_{22} + R_{33} ) ( 1 + \cos \theta ) 
			\right\} \sin \theta 
			\\
			&= - \pi \left\{\ScP +  R_{11} \right\} ( 1 + \cos \theta ) \sin \theta
			,\\
			\int_{0}^{2\pi} (R_{\mathbf{f}_{1} e_{1}} - R_{n \ex} ) \cos 2 \varphi 
			d \varphi 
			&= - \frac{\pi}{2} (R_{22} - R_{33} ) ( 1 + \cos \theta ) \sin \theta. 
		\end{aligned}
	\]
Hence, we have 
	\[
		\begin{aligned}
			\int_{0}^{2\pi} 
			\left\{ \frac{\ScP}{6} ( 1 + \cos \theta ) \sin \theta 
			+ \frac{1}{3} (R_{\mathbf{f}_{1} e_{1}} - R_{n \ex} ) \right\} 
			d \varphi
			&= \frac{\pi}{3} \left\{ \ScP - (\ScP + R_{11}) \right\}
			( 1 + \cos \theta ) \sin \theta  
			\\
			&= - \frac{\pi}{3} R_{11} ( 1 + \cos \theta ) \sin \theta,
			\\
			 \int_{0}^{2\pi}
			\left\{ \frac{\ScP}{6} ( 1 + \cos \theta ) \sin \theta 
			+ \frac{1}{3} (R_{\mathbf{f}_{1} e_{1}} - R_{n \ex} ) \right\} 
			\cos 2 \varphi d \varphi 
			&= - \frac{\pi}{6} (R_{22} - R_{33}) ( 1 + \cos \theta ) \sin \theta .
		\end{aligned}
	\]
Thus 
	\[
		\begin{aligned}
			& - \tilde{A} \int_{0}^{\pi} \int_{0}^{2\pi} 
			(1 - \chi_\delta) 
			\left\{ \frac{\ScP}{6} ( 1 + \cos \theta ) \sin \theta 
			+ \frac{1}{3}(R_{\mathbf{f}_{1} e_{1}} - R_{n \ex}) \right\} 
			\\
			& \qquad \qquad \times 
			\left[ - 2 A \cos 2 \theta + 2 B \cos 2 \varphi 
			\left\{ \cos 2 \theta - 2 
			+ 2 \frac{(1-\cos \theta) \cos \theta}{ \sin^{2} \theta }
			\right\} \right] d \varphi d \theta 
			\\
			=& - \frac{2}{3} \pi  \tilde{A} A R_{11}
			\int_{0}^{\pi} (1 - \chi_\delta) 
			( 1 + \cos \theta ) \sin \theta \cos 2 \theta d \theta 
			\\
			& + \frac{\pi}{3} \tilde{A} B (R_{22} - R_{33}) 
			\int_{0}^{\pi} (1 - \chi_\delta) ( 1 + \cos \theta ) \sin \theta 
			\left\{ \cos 2 \theta - 2 
			+ 2 \frac{(1-\cos \theta) \cos \theta}{ \sin^{2} \theta }
			\right\} d \theta 
			\\
			=& - \frac{2}{3} \pi  \tilde{A} A R_{11}
			\int_{0}^{\pi} (1 - \chi_\delta) ( 1 + \cos \theta ) 
			( 2 \cos^{2} \theta - 1 ) \sin \theta  d \theta 
			\\
			& + \frac{\pi}{3}  \tilde{A}  B (R_{22} - R_{33})  \int_{0}^{\pi} 
			(1 - \chi_\delta) \sin \theta \left\{ 
			( 2 \cos^{2} \theta - 1 ) ( 1 + \cos \theta) 
			- 2 ( 1 + \cos \theta) + 2 \cos \theta 
			\right\} d \theta 
			\\
			=& - \frac{2}{3} \pi  \tilde{A} R_{11} A
			\int_{0}^{\pi} ( 2 \cos^{2} \theta - 1 ) \sin \theta d \theta 
			+ \frac{\pi}{3} \tilde{A} B (R_{22} - R_{33})  \int_{0}^{\pi} 
			\sin \theta ( 2 \cos^{2} \theta - 3 ) d \theta  + O(\delta^2)
			\\
			=& - \frac{2}{3} \pi  \tilde{A} A R_{11} 
			\left( \frac{4}{3} - 2 \right) + \frac{\pi}{3} \tilde{A} B
			(R_{22} - R_{33} ) \left( \frac{4}{3} - 6 \right) + O(\delta^2)
			\\ 
			=& \frac{4}{9} \pi \tilde{A} A R_{11} 
			- \frac{14}{9} \pi \tilde{A} B (R_{22} - R_{33} ) + O(\delta^2). 
		\end{aligned}
	\]
Thus we get 
	\[
		- \int_{\SA} (1 - \chi_\delta) e_1(h_{n1}) \Delta_{\SA} \psi_0 d \sigma 
		= \tilde{A} 
		\left\{ - \frac{4}{9} \pi A R_{11} + \frac{14}{9} \pi B (R_{22} - R_{33}) \right\} 
		+ O(\delta^2). 
	\]

\bigskip

Combining \eqref{eq:H-dot-2}, we have 
	\[
		\int_{S^{2}_{\tilde A}} ( 1 - \chi_{\delta}) 
		F  \Delta_{S^{2}_{\tilde A}} \psi_0 d \sigma 
		= 4 \pi  \tilde{A} B (R_{22} - R_{33} ) + O(\delta^{2})
	\]
and we complete the proof. 
\end{proof}

\begin{proof}[Proof of Proposition \ref{p:PV}]
From Lemmas \ref{l:H-Ric} and \ref{l:H-dot-psi0}, we have 
	\[
		\begin{aligned}
			\int_{S^{2}_{\tilde A}} 
			( 1 - \chi_{\delta} ) \left\{
			F \Delta_{S^{2}_{\tilde A}} \psi_{0}  
			+ H_{S^{2}_{\tilde A}}\Ric_{P} ( n_0, n_0 ) \psi_{0} \right\} d \sigma
			&=  4 \pi  \tilde{A} B (R_{22} - R_{33} ) + \frac{4}{3} \pi \tilde{A} B
			(R_{22} - R_{33} ) + O(\delta^{2}) 
			\\
			&= \frac{16}{3} \pi \tilde{A} B (R_{22} - R_{33}) 
			+ O(\delta^{2}),
		\end{aligned}
	\]
which gives \eqref{eq:PV}. 
Thus we completed the proof. 
\end{proof}

\subsection{Appendix II: study of  $\mathcal{F}(P,R)$}
In this appendix we prove the following  result, which guarantees condition $(ND2)$ in the introduction.
\begin{pro}\label{p:FNonDeg} Let $S$ be a symmetric bilinear form on $\R^3$, and denote by $\a_1,\a_2,\a_3$ the eigenvalues  corresponding to the eigenvectors $\bbe_1, \bbe_2, \bbe_3$. Consider the function $F:SO(3)\to \R$ defined by 
\begin{equation}\label{eq:defFS}
F(R):=S(R\bbe_2, R \bbe_2)-S(R \bbe_3, R \bbe_3).
\end{equation}
 Then $F$ is a Morse function if and only if the eigenvalues are distinct: $\a_i \neq \a_j$ for $i\neq j$.
 \\ In this case $F$ has exactly $24$ critical points $\{R_{(ij)}\}$ satisfying  $R\bbe_2=\pm \bbe_i$, $R \bbe_3=\pm \bbe_j$ for $i,j=1,2,3$ with $i\neq j$ and the eigenvalues of the Hessian  $\nabla^2 F(R_{(ij)})$  are $\a_k-\a_i, 2(\a_j-\a_i),\a_j-\a_k$, where $\{k\}=\{1,2,3\}\setminus\{i,j\}$, and $F(R_{(ij)})=\a_i-\a_j$.
 In particular
 \begin{itemize}
 \item $F$ has exactly 4 critical points of index 3 given by $\{R_{(ij)}\}$ with $i=\pm 3$ and $j=\pm 1$. They all satisfy $F(R_{(ij)})>0$.
 \item $F$ has exactly 8 critical points of index 2 given by $\{R_{(ij)}\}$ with $i=\pm 3$ and $j=\pm 2$, or $i=\pm 2$ and $j=\pm 1$. They all satisfy $F(R_{(ij)})>0$.
 \item $F$ has exactly 8 critical points of index 1 given by $\{R_{(ij)}\}$ with $i=\pm 2$ and $j=\pm 3$, or $i=\pm 1$ and $j=\pm 2$. They all satisfy $F(R_{(ij)})<0$.
 \item $F$ has exactly 4 critical points of index 0 given by $\{R_{(ij)}\}$ with $i=\pm 1$ and $j=\pm 3$. They all satisfy $F(R_{(ij)})<0$.
 \end{itemize}
\end{pro}

\begin{proof}
First of all let us show that if $S$ has multiple eigenvalues then $F$ cannot be Morse. Up to relabelling we can assume $\a_1=\a_2$. Let $\bar{R}\in SO(3)$ be the rotation such that $\bar{R} \bbe_2=\bbe_1$ and $\bar{R} \bbe_3=\bbe_2$; then for every rotation $R_\theta$, $\theta \in S^1$, with axis $\bbe_3$ (i.e. a rotation of the plane spanned by $\bbe_1$ and $\bbe_2$) we have
$$F(R_\theta \circ \bar{R})=S(R_\theta \circ \bar{R} \bbe_2)-S(R_\theta \circ \bar{R} \bbe_3)=S(R_\theta \bbe_1)-S(R_\theta \bbe_2)=\a_1-\a_1=0.$$ 
Since $F$ is constant on a one-dimensional submanifold it cannot be Morse. 

From now on we therefore assume $\a_i\neq \a_j$ for $i\neq j$.
Throughout the proof all  the vectors and matrices of $\R^3$ will be expressed in coordinates with respect to  the basis $(\bbe_1, \bbe_2, \bbe_3)$ of eigenvectors of $S$.
Notice that a rotation $R\in SO(3)$ is uniquely  determined by the coordinates $(x_1,x_2,x_3)$ of $R \bbe_2$ and by the coordinates $(x_4,x_5,x_6)$ of  $R \bbe_3$, note also that such coordinates satisfy the following  non degenerate system of three constraints:
\begin{equation}\label{eq:constr}
\left\{\sum_{i=1}^3 x_i^2=1, \; \sum_{j=4}^6 x_j^2=1, \; \sum_{i=1}^3 x_i x_{i+3}=0 \right\} .
\end{equation}
Therefore finding a critical point of $F:SO(3)\to \R$ is equivalent to finding critical points of the corresponding function defined on $\R^6$ under the constraints \eqref{eq:constr} which, by the Lagrange multipliers rule, is in turn equivalent to look for free  critical points, in $x\in \R^6$, of the Lagrange function
\begin{equation}\label{eq:defLagrange}
L(x_1, \ldots, x_6, \l, \mu, \nu):= \sum_{i=1}^3 \a_i(x_i^2-x_{i+3}^2)-\l\sum_{i=1}^3 x_i^2- \mu \sum_{i=4}^6 x_i^2-\nu\sum_{i=1}^3 x_i x_{i+3}.
 \end{equation}
This corresponds to solving the following system of nine equations in $(x_1, \ldots, x_6, \l, \mu, \nu)$. 
Notice that the first six equations are linear and correspond to the optimization of $L$ in $x$, 
the last three equations are quadratic and correspond to the constraints \eqref{eq:constr}:
\begin{equation}\label{eq:CSys}
\left\{
\begin{array}{l}
2(\a_i -\l) x_i - \nu x_{i+3}=0, \quad i=1,2,3 \\[1mm] 
\nu x_i +2 (\mu+\a_i) x_{i+3}=0 \quad i=1,2,3 \\ [1mm] 
\sum_{i=1}^3 x_i^2=1, \; \sum_{j=4}^6 x_j^2=1, \; \sum_{i=1}^3 x_i x_{i+3}=0 .
\end{array}
\right.
\end{equation}

		As the first step, we show $\nu = 0$. 
Let $x=(x_1,\ldots,x_6)$ satisfy \eqref{eq:CSys}. Then 
it follows from \eqref{eq:CSys} that 
	\[
		\begin{aligned}
			\nu &= \sum_{i=1}^3 \nu x_{i+3}^2 
			= \sum_{i=1}^3 2 ( \alpha_i - \lambda) x_i x_{i+3} 
			= 2 \sum_{i=1}^3 \alpha_i x_i x_{i+3}, 
			\\
			\nu &= \sum_{i=1}^3 \nu x_i^2 
			= \sum_{i=1}(-2) (\mu + \alpha_i) x_i x_{i+3} 
			= - 2 \sum_{i=1}^3 \alpha_i x_i x_{i+3}.
		\end{aligned}
	\]
Summing these two equations, we obtain $\nu = 0$.

		Since we know $\nu = 0$, 
by the assumptions $\alpha_i \neq \alpha_j$ for $i \neq j$, 
it is immediate to check that the solutions of \eqref{eq:CSys} are given by 
	\begin{equation}\label{eq:solCP}
		\{(x_1,\ldots, x_6) \; : \; 
		x_i=\pm 1, x_j=\pm 1, x_k=0, \l=\a_i, \mu=-\a_j, \nu=0 \},
	\end{equation}
for exactly one $i \in \{1,2,3\}$, one $j \in \{4,5,6\}$ with 
$j - 3 \neq i$ and  for all $k \in \{1,\cdots,6\} \setminus \{i,j\}$. 
Notice that these 24 solutions correspond to the rotations $R_{(ij)}\in SO(3)$ 
described in the statement of the proposition.

In order to know the index of these 24 critical points, observe that it is enough to perform a second order analysis at $R=Id=R_{(23)}\in SO(3)$: indeed,  the index of $F$ at $R_{(ij)}$ is the same as $F\circ R_{(ij)}$ at $Id$, so the general case just follows by a suitable relabelling of the indices.

By using \eqref{eq:solCP},  at the critical point $R=R_{(23)}=Id$ the Lagrange function \eqref{eq:defLagrange} takes the form
$$ 
L(x, \l=\a_2, \mu=-\a_3, \nu=0)= (\a_1-\a_2)x_1^2+(\a_3-\a_2) x_3^2+(\a_3-\a_1)x_4^2+(\a_3-\a_2)x_5^2.
$$ 
 Since  $v\in \R^6$ is tangent to the constraints \eqref{eq:constr} at $\bar{x}=(0,1,0,0,0,1)$ if and only if it has the form $v=(v_1,0,v_3, v_4,-v_3,0)$, the  Hessian in $x$ of $L$ on the tangent space to the constraint manifold  at $\bar{x}$ is 
 $$\n_x^2 L(\bar{x},\l=\a_2, \mu=-\a_3, \nu=0)[v]=(\a_1-\a_2) v_1^2+ 2(\a_3-\a_2) v_3^2+(\a_3-\a_1) v_4^2. $$
 But such  a constrained Hessian corresponds to the Hessian of $F$ at $R_{(23)}=Id$: $\nabla^2 F(R_{(23)})$. It follows that $\nabla^2 F(R_{(23)})$ is non degenerate if and only if $\a_i\neq \a_j$ for $i\neq j$, that the eigenvalues of $\nabla^2 F(R_{(23)})$ are $\a_1-\a_2, 2(\a_3-\a_2), \a_3-\a_1$. Moreover, assuming $\a_1<\a_2<\a_3$,    the index of $\nabla^2F(R_{(23)})$ is  one and $F(R_{(23)})=\a_2-\a_3<0$. As mentioned above, the second order analysis of $F$ at the general critical point $R_{(ij)}$ follows then by a relabelling argument. 
 \end{proof}


\begin{thebibliography}{99}
\bibitem{ab1} A. Ambrosetti, M. Badiale, 
{\it Homoclinics: Poincar\'e-Melnikov type results via a variational approach,}
Ann. Inst. Henri Poincar\'e Analyse Non Lin\`eaire 15, (1998), 233--252 . 

\bibitem{ab2} A. Ambrosetti, M. Badiale, 
{\it Variational perturbative methods and bifurcation of bound states from the essential spectrum}
Proc. R. Soc. Edinb. \textbf{18}, (1998), 1131--1161. 


\bibitem{am2} A. Ambrosetti, A. Malchiodi, 
{\it Perturbation methods and semilinear elliptic problems in $\R^n$,} Progress in mathematics, Birkhauser, (2006).


\bibitem{BFLM} M. Barros, A.  Ferr\'andez, P. Lucas, M. A. Merono, 
{\it Willmore tori and Willmore-Chen submanifolds in
pseudo-Riemannian spaces}, J. Geom. Phys., Vol. 28, (1998), 45--66.

\bibitem{BK}  M. Bauer, E. Kuwert, 
{\it Existence of minimizing Willmore surfaces of prescribed genus,} Int. Math. Res. Not., Vol.10, (2003), 553-576.

\bibitem{BR}  Y. Bernard, T.  Rivi\`ere , 
{\it Energy quantization for Willmore surfaces and applications,} Annals of Math., Vol. 180, Num. 1, (2014).   87--136.  

\bibitem{Besse} A. L. Besse,
{\it Einstein Manifolds}, Classics in Mathematics, Springer-Verlag, Berlin, 2008 (Reprint of the 1987 edition),  xii+516.


\bibitem{bressan}  A. Bressan, {\it Hyperbolic systems of conservation laws}. The one-dimensional Cauchy problem. Oxford Lecture Series in Mathematics and its Applications, 20. Oxford University Press, Oxford, 2000.


\bibitem{CM} A. Carlotto, A. Mondino, 
{\it Existence of Generalized Totally Umbilic $2$-Spheres in Perturbed $3$-Spheres}, Int. Math. Res. Not.,  Vol.  2014, Num. 21, (2014),   6020--6052. 
 

\bibitem{ChenLi}  J. Chen, Y. Li,
{\it Bubble tree of a class of conformal mappings and applications to Willmore functional}, preprint arXiv:1112.1818, (2011), American Journ. Math., Vol. 134, Num. 4, (2014), 1107--1154. 
 
 
 \bibitem{Hatcher}  A. Hatcher,
 {\it Algebraic topology},  Cambridge University Press, Cambridge, (2002). 
 
 \bibitem{IMM1} N. Ikoma, A. Malchiodi, A. Mondino,
 {\it Embedded area-constrained Willmore tori of small area in Riemannian 3-manifolds  I: Minimization}, preprint  arXiv:1411.4396,  (2014).

\bibitem{Kus} R. Kusner,
{\it Estimates for the biharmonic energy on unbounded planar domains, and the existence of surfaces of every genus that minimize the squared-mean-curvature integral}, Elliptic and Parabolic Methods in Geometry (Minneapolis, Minn, 1994), A. K. Peters, Massachusetts, (1996), 67-72.


\bibitem{KMS} E. Kuwert, A. Mondino, J. Schygulla, 
{\it Existence of immersed spheres minimizing curvature functionals
in compact 3-manifolds},  Math. Ann. Vol. 359, Num. 1-2, 379--425,  (2014).


\bibitem{KS} E. Kuwert, R. Sch\"atzle,
{\it Removability of isolated singularities of Willmore surfaces,} Annals of Math. Vol. 160, Num. 1, (2004), 315--357. 

\bibitem{LM1} T. Lamm, J. Metzger,
{\it Small surfaces of Willmore type in Riemannian manifolds,} Int. Math. Res. Not. 19 (2010), 3786--3813. 

\bibitem{LM2} T. Lamm, J. Metzger,
{\it Minimizers of the Willmore functional with a small area constraint,} 
Annales IHP-Anal.  Non Lin., Vol.  30, (2013), 497--518. 

\bibitem{LMS} T. Lamm, J. Metzger, F. Schulze,
{\it Foliations of asymptotically flat manifolds by surfaces of Willmore type,} Math. Ann., Vol. 350, Num. 1, (2011), 1--78.
 
\bibitem{LM} P. Laurain,  A. Mondino,
{\it Concentration of small Willmore spheres in Riemannian 3-manifolds,} Analysis \& PDE, Vol. 7, Num. 8, (2014), 1901--1921.


\bibitem{LP}  J. M. Lee, T. H. Parker, 
{\it The {Y}amabe problem,} Bull. Amer. Math. Soc. (N.S.),  Vol. 17, Num. 1, (1987), 37--91.


\bibitem{LY} P. Li, S. T. Yau,
{\it A new conformal invariant and its applications to the Willmore conjecture
and the first eigenvalue of compact surfaces,} Invent. Math.,  Vol. 69, Num. 2, (1982), 269--291.



\bibitem{MN} F. Marques, A. Neves, 
{\it Min-max theory and the Willmore conjecture,} Annals of Math., Vol. 179, Num. 2,  (2014),   683--782.  

\bibitem{MilS} J. Milnor, J. D. Stasheff,
{\it Characteristic classes}, Annals of Math. Studies, Num. 76., Princeton University Press,  (1974). 

\bibitem{Mon1} A. Mondino, 
{\it Some results about the existence of critical points for the Willmore functional,} Math. Zeit., Vol. $266$, Num. $3$, (2010), 583--622.

\bibitem{Mon2} A. Mondino, 
{\it The conformal Willmore Functional: a perturbative approach,}   J. Geom. Anal.,  Vol. 23, Num. 2, (2013),  764--811.

\bibitem{MonNgu} A. Mondino, H. T. Nguyen, 
{\it A Gap Theorem for Willmore Tori and an application to the Willmore Flow,} Nonlinear Analysis TMA, Vol. 102, (2014), 220--225.



\bibitem{MR1} A. Mondino, T. Rivi\`ere, 
{\it Immersed Spheres of Finite Total Curvature into Manifolds,} Advances  Calc.  Var., Vol. 7, Num. 4, (2014), 493--538.  

\bibitem{MR2} A. Mondino, T.Rivi\`ere, 
\textit{ Willmore spheres in compact Riemannian manifolds,} Advances Math., Vol. 232, Num. 1, (2013), 608--676.

\bibitem{MonSch} A. Mondino, J. Schygulla, 
{\it Existence of immersed spheres minimizing curvature functionals in non-compact 3-manifolds,}  Annales IHP-Anal. Non Lin.,  Vol. 31, (2014),  707--724.

\bibitem{MonRos} S. Montiel, A. Ros, 
{\it Minimal immersions of surfaces by the first Eigenfunctions and conformal area,} Invent. Math., Vol. 83, (1986), 153--166.

\bibitem{MVS} M. Morse, G. B. Van Schaack,
{\it The Critical Point Theory Under General Boundary Conditions}, Annals of Math., Vol. 35. Num. 3, (1934), 545--571.


\bibitem{ST} J.C. Saut, R. Temam, {\it Generic properties of nonlinear boundary value problems,} Comm.
Partial Differential Equations 4 (1979), no. 3, 293--319. 

\bibitem{Riv1} T. Rivi\`ere,
{\it Analysis aspects of Willmore surfaces,}  Invent. Math., Vol. $174$, Num. 1, (2008), 1--45.

\bibitem{Riv2} T. Rivi\`ere,
{\it Variational Principles for immersed Surfaces with $L^2$-bounded Second Fundamental Form,}  J. Reine. Angew. Math., Vol. 2014, Num. 695, (2014), 41--98.

\bibitem{Ros} A. Ros,
{\it The Willmore conjecture in the real projective space,} Math. Res. Lett., Vol. 6,
(1999), 487--493.

\bibitem{SiL} L. Simon,
{\it Existence of surfaces minimizing the Willmore functional, } Comm. Anal. Geom., Vol. $1$, Num. $2$, (1993), 281--325. 

\bibitem{Top} P. Topping,
{\it Towards the Willmore conjecture.} Calc. Var. and PDE, Vol. 11, (2000), 361--393.


\bibitem{Urb} F. Urbano, 
{\it Minimal surfaces with low index in the three-dimensional sphere}, Proc.
Amer. Math. Soc., Vol. 108, (1990), 989--992. 

\bibitem{Wang} P. Wang,
{ \it On the  Willmore functional of 2-tori in some product Riemannian manifolds}, Glasgow Math. J., Vol. 54, (2012), 517--528. 

\bibitem{WEINER} J. L. Weiner,  
{\it On a problem of Chen, Willmore, et al., } Indiana Univ. Math. J., Vol  27,  Num. 1, (1978), 19--35.


\bibitem{Will} T.J. Willmore,
{\it Riemannian Geometry, } Oxford Science Publications, Oxford University Press (1993).


\end{thebibliography}
\end{document}